\documentclass[11pt, reqno]{amsart}
\usepackage[includehead,includefoot,margin=20mm]{geometry}
\usepackage[latin1]{inputenc}
\usepackage{amssymb}
\usepackage{amsmath}
\usepackage{amsfonts}
\usepackage{mathtools}
\usepackage{mathrsfs}
\usepackage{hyperref}
\usepackage[all]{xy}  
\usepackage{verbatim}
 \usepackage{graphicx}
\usepackage[dvipsnames]{xcolor}
\usepackage{xy}
\hypersetup{colorlinks=true,citecolor=blue,linkcolor=red,urlcolor=red}
\usepackage{tikz}
\usetikzlibrary{arrows.meta, positioning}
\usetikzlibrary{arrows.meta}
\usepackage{nicematrix}
\usepackage{csquotes}
\usepackage{epigraph}
\theoremstyle{plain}
\newtheorem{theorem}{Theorem}[section]

\newtheorem{lemma}[theorem]{Lemma}
\newtheorem{proposition}[theorem]{Proposition}

\newtheorem{corollary}[theorem]{Corollary}
\newtheorem*{theorem*}{Theorem}

\theoremstyle{definition}
\newtheorem{definition}[theorem]{Definition}

\newtheorem{example}[theorem]{Example}

\theoremstyle{remark}

\newtheorem{remark}[theorem]{Remark}

\title[Fano Generalized Bott-Samelson Varieties]{Fano Generalized Bott-Samelson Varieties}

\author[B.~N.~Chary]{Narasimha Chary Bonala}
\address{Narasimha Chary Bonala\\
Department of Mathematics, Indian Institute of Technology Kanpur,
U.P. India, 208016.
}
\email{chary@iitk.ac.in}

\author[S.~Pattanayak]{Santosha Pattanayak}
\address{Santosha Pattanayak\\
Department of Mathematics, Indian Institute of Technology Kanpur, U.P. India, 208016.
}
\email{santosha@iitk.ac.in}

\author[Y. Singh]{Yogendra Singh}
\address{Yogendra Singh\\ Department of Mathematics, Indian Institute of Technology Kanpur,
U.P. India, 208016.}
\email{ysingh23@iitk.ac.in}

\thanks{ }


\begin{document}
\begin{abstract}

The Bott-Samelson varieties provide natural desingularizations of Schubert varieties, and their generalizations, called generalized Bott-Samelson varieties, were constructed by Nicolas Perrin as towers of locally trivial fibrations with fibers isomorphic to Schubert varieties. In this article, we give a complete characterization of Fano and weak Fano generalized Bott-Samelson varieties. As a consequence, we recover the corresponding results for  ($G$-)Bott-Samelson varieties and minuscule generalized Bott-Samelson varieties. 
\end{abstract}
\maketitle
\section{Introduction}

 Let $G$ be a semisimple simply-connected complex algebraic group. Fix a maximal torus $T \subset G$ and a Borel subgroup $B \subset G$ containing $T$. Let 
$
W := N_G(T)/T
$
denote the Weyl group of $G$ with respect to $T$, where $N_G(T)$ is the normalizer of $T$ in $G$. For a parabolic subgroup $P$ containing $B$, denote by $W_P$ the Weyl group of $P$, 
and by $W^{P}$ the set of minimal length representatives of the cosets in $W/W_P$. For $w\in W^{P}$, consider the associated Schubert variety \[
X(w):=\overline{BwP}/P \subseteq G/P.
\]
It is well known that $X(w)$ is normal, Cohen-Macaulay, and has rational singularities 
\cite{Ram85, And85, LSe86}. However, in general, it is not smooth. Schubert varieties are central objects in representation theory and have been extensively studied. There are several approaches to understanding their geometry and singularities. One important approach is to study the singular locus, its irreducible components, and the nature of singularity at a general point of each component. This problem is largely understood for Schubert varieties in partial flag varieties, but remains incomplete, in general. Another approach involves studying cohomology and vanishing theorems for line bundles, often via resolutions of singularities. For a survey on singularities of Schubert varieties, we refer to \cite{BL00}. 

For a sequence $\tilde{w} = (s_{i_1}, s_{i_2}, \dots, s_{i_r})$ of simple reflections 
such that $w = s_{i_1}s_{i_2}\cdots s_{i_r}\in W^{P}$ is a reduced expression, we can associate a
Bott-Samelson variety $Z(\tilde{w})$, which provides a natural 
desingularization of the Schubert variety $X(w)$. 
These varieties were originally introduced by Bott and Samelson \cite{BS58} in a differential-geometric and topological setting, and later independently adapted by Demazure \cite{Dem74} and Hansen \cite{Han73} to the algebro-geometric context. 
The Bott-Samelson variety $Z(\tilde{w})$ can be viewed as an iterated $\mathbb{P}^1$-bundle associated with the reduced expression $\tilde{w}$.  The geometry of these varieties depends on the given expression of the Weyl group element corresponding to the Schubert variety (see, for instance, \cite[Page 32]{BKP15} and \cite{BK17}).

The Bott-Samelson resolutions are rarely small: fibers often have large dimensions and there are many contracted subvarieties. Also, the automorphisms of the target need not lift to the resolution (see \cite[Section~7]{BKP15} for explicit examples).
The \emph{generalized} Bott-Samelson varieties include small resolutions of Schubert varieties constructed by Zelevinsky \cite{Zel83}, by Sankaran and Vanchinathan \cite{SV94, SV95}, and by Perrin in full generality (see \cite[Cor.~7.9]{Per07}). 
 They are constructed as a tower of locally trivial fibrations with fibers isomorphic to
Schubert varieties. 
 From the construction, it is easy to observe that the generalized Bott-Samelson varieties
$\widehat{X}(\widehat{w})$ are normal, Cohen-Macaulay, and have rational singularities (see Lemma \ref{rmk:smoothness} and Theorem~\ref{prop:rational-sing}). These descriptions are well suited for 
the computations of Kazhdan-Lusztig polynomials.

The Gorenstein property provides a measure of how far an algebraic variety is from being smooth: every smooth variety is Gorenstein, and every Gorenstein variety is Cohen-Macaulay. In general, a  (normal) variety is Gorenstein if it is Cohen-Macaulay and its canonical sheaf is a line bundle. Recall that a  projective Gorenstein variety $X$ is called Fano (resp. weak Fano) if its anticanonical divisor $-K_X$ is ample (resp. nef and big).

There is a classification of minuscule Schubert varieties that are Gorenstein (or Fano); see \cite{WY06, Per09}. Recently, in \cite{CKM25}, Li, Rietsch, and Yang gave a classification of Fano Schubert varieties in the simply laced case.

The first author characterized Fano (weak Fano) Bott-Samelson varieties associated with reduced expressions in \cite{BC18}. Bott-Samelson varieties admit natural toric degenerations to Bott manifolds, Fano and weak Fano Bott manifolds were classified in \cite{BC18a}. Subsequently, as an application of this toric degeneration, a classification of Fano and weak Fano Bott-Samelson varieties was obtained in \cite{BC21}. In \cite{BS24}, Bhaumik and Saha considered arbitrary expressions (not necessarily reduced) and provided a different characterization of Fano (weak Fano) Bott-Samelson varieties, and also classified Fano (weak Fano) $G$-Bott-Samelson varieties. By the work of Perrin in \cite{Per07}, one can give a classification of minuscule generalized Bott-Samelson varieties that are Fano (weak Fano); see \cite[Proposition 6.3 and Corollary 6.7]{Per07}.

In this paper, we classify the Fano (weak Fano) \emph{generalized Bott-Samelson varieties}.
To state our main results, we first set up some notation.

Let $w \in W$ and let $\widehat{w} = (w_1, \dots, w_m)$ be an \emph{admissible 
generalized (not necessarily reduced) decomposition} of $w$. Denote by $\widehat{X}(\widehat{w})$ the associated generalized Bott-Samelson variety. See Section \ref{sec:gBS definition} for more details.

We first describe the line bundles and divisors on $\widehat{X}(\widehat{w})$.
Let $S$ be the set of simple roots of $G$ with respect to $B$, and let $P^w$ be the largest parabolic subgroup of $G$ containing $B$ such that $w$ is a minimal-length representative in $W/W_{P^w}$. For $1 \leq j \leq m$, let $I^{w_j} \subseteq S$ be the subset corresponding to $P^{w_j}$. 

For each $1 \leq j \leq m$ and each $\alpha \in S \setminus I^{w_j}$, we associate a line bundle $\mathcal{L}_{j,\alpha}$ on $\widehat{X}(\widehat{w})$ (see Section~\ref{map:pi} and Eq.\eqref{eq: Line bundles on gBS}). We prove that the Picard group of $\widehat{X}(\widehat{w})$ is a free abelian group with a basis consisting of these $\mathcal{L}_{j, \alpha}$'s (see Theorem~\ref{thm:linebundles}). Moreover, we provide a characterization of very ample line bundles on $\widehat{X}(\widehat{w})$ (see Theorem~\ref{thm:very ample}), generalizing the result of Lauritzen and Thomsen for Bott-Samelson varieties (see Theorem~3.1 in \cite{LT04}). As a consequence, we show that any ample line bundle on $\widehat{X}(\widehat{w})$ is very ample (Corollary~\ref{cor:ample}). We further prove that a line bundle on $\widehat{X}(\widehat{w})$ is globally generated if and only if it is nef (Corollary~\ref{cor:globally}).

We construct a finite sequence $\operatorname{Peaks}\widehat{X}(\widehat{w})$ of roots of $(G,T)$ and prove that it is in one-to-one correspondence with a basis of the divisor class group $\operatorname{Cl}(\widehat{X}(\widehat{w}))$, see Proposition \ref{pro: divisor classes for gBS}. This construction generalizes the notion of peaks of a quiver associated with minuscule generalized Bott-Samelson variety in \cite[Section 4]{Per07}. For further details, see Section~\ref{sec: Fano gBS}.
Using this notion, we obtain the following description of the anticanonical line bundle on $\widehat{X}(\widehat{w})$ (see Proposition ~ \ref{anti-can}). 
\begin{proposition}
The anticanonical divisor class of $\widehat{X}(\widehat{w})$ is given by
\[
[-K_{\widehat{X}(\widehat{w})}]
=
\sum_{\widehat{\eta}\in \operatorname{Peaks}\widehat{X}(\widehat{w})}
\big(\operatorname{ht}(\widehat{\eta}^\vee)+1\big)\widehat{D}_{\widehat{\eta}},
\]
where $\operatorname{ht}(\widehat{\eta}^\vee)$ denotes the height of the coroot $\widehat{\eta}^\vee$ and $\widehat{D}_{\widehat{\eta}}$ is the divisor associated with the peak $\widehat{\eta}$.
\end{proposition}

If the anticanonical divisor $-K_{\widehat{X}(\widehat{w})}$ is Cartier, then there exist integers $c_{j,\alpha}\in \mathbb{Z}$ such that\begin{equation}\label{eq: coff intro}
\mathcal O_{\widehat{X}(\widehat{w})}\big(-K_{\widehat{X}(\widehat{w})}\big)=\sum_{j=1}^{m}\sum_{\alpha\in S\setminus I^{w_j}}c_{j,\alpha}\mathcal{L}_{j,\alpha}.
\end{equation}

We denote the divisor class associated to $\mathcal{L}_{j,\alpha}$ by $[\mathcal{L}_{j,\alpha}]$. We then express the divisor classes $[\mathcal{L}_{j,\alpha}]$ in the basis 
$\{\widehat{D}_{\widehat{\eta}} \mid \widehat{\eta}\in \operatorname{Peaks}\widehat{X}(\widehat{w})\}$ 
(see Proposition~\ref{prop: divisor classes for line bundles on gBS}) and compare coefficients in Eq.\eqref{eq: coff intro}. 
This yields a linear system (cf. Eq.\eqref{eq:linear systems}), from which we derive a subsystem of linear equations and provide an explicit method for solving it and computing the coefficients $c_{j,\alpha}$ (see Subsection~\ref{subsection: M1}).

When $G$ is simply laced type, for each $1\leq j\leq m$, there is a bijection
\[
\mu_{w_j} : S\setminus I^{w_j} \longrightarrow \mathcal B_{P^{w_j}} \subseteq \operatorname{Peaks}X(w_j),
\]
such that $\langle \omega_\alpha ,\mu_{w_j}(\beta)\rangle=\delta_{\alpha\beta}$
for all $\alpha, \beta \in S\setminus I^{w_j}$, where $\operatorname{Peaks}X(w_j)$ denotes the set of peaks of the Schubert variety $X(w_j)$ (see Subsection~\ref{subsub: Fano for Xw} and Lemma~\ref{lem:B_w,P^w}).
Then, we prove the following.

\begin{theorem}[see Theorem~\ref{thm: Fano gBS}]\label{thm:main} Assume that $G$ is of simply laced type.  The generalized Bott-Samelson variety $\widehat{X}(\widehat{w})$ is Fano (resp. weak Fano) if and only if the following conditions hold:

\begin{enumerate}
    \item For each $1 \le j \le m$ and  
    each $
    \eta \in \operatorname{Peaks} X(w_j) \setminus \mathcal{B}_{P^{w_j}},
    $
    we have
    \begin{equation*}\label{eq:height relations for Xwi}
    \operatorname{ht}(\eta^\vee) + 1
    = \sum_{\alpha \in S \setminus I^{w_j}} 
      d_\alpha \big( \operatorname{ht}(\mu_{w_j}(\alpha)^\vee) + 1 \big),
    \end{equation*}
    where $d_\alpha \in \mathbb{Z}_{\ge 0}$ denotes the coefficient of $\alpha^\vee$ in the expansion of $\eta^\vee$.
    \item For all $1 \le j \le m$ and all $\alpha \in S \setminus I^{w_j}$, we have 
    \[
    c_{j,\alpha} > 0 \quad (\text{resp. } c_{j,\alpha} \ge 0).
    \]
\end{enumerate}
\end{theorem}

\begin{remark}\ 

\begin{enumerate}
\item Condition (1) in Theorem~\ref{thm:main} gives a characterization of Gorenstein generalized Bott-Samelson varieties. 
    \item When $G$ is not simply laced, such functions $\mu_{w_j}$ need not exist; see Remark~\ref{B5} and Remark~\ref{rem:nonsimplylaced1}. However, for non simply laced group $G$, if we additionally assume that for each $1 \leq j \leq m$, there exists an injective map
$$
\mu_{w_j} : S \setminus I^{w_j} \rightarrow \mathrm{Peaks}(X(w_j)),
$$
such that $\langle \omega_\alpha ,\mu_{w_j}(\beta)\rangle=\delta_{\alpha\beta}$
for all $\alpha, \beta \in S\setminus I^{w_j}$, then conditions (1) and (2) in Theorem~\ref{thm:main} yield a classification of Fano (resp. weak Fano) generalized Bott-Samelson varieties. 
    \item In Section \ref{sec:Recover}, we prove that ($G$-)Bott-Samelson varieties
and flag Bott-Samelson varieties (see \cite[Remark 2.2]{FLS20}) satisfy this additional assumption. Therefore, our results recover those of \cite{BC18} and \cite{BS24}; see Theorem \ref{thm:FanoBS} and Theorem \ref{thm: GBS Fano}.
\item Theorem \ref{thm:main} also recovers the results for minuscule generalized Bott-Samelson varieties in \cite{Per07}; see Section~\ref{sec:minuscle} and Remark~\ref{rmk:minuscluse1}. 
\end{enumerate}

\end{remark}

\subsection*{Structure of the paper}

The paper is organized as follows. In Section~\ref{sec:basics}, we review the necessary background material and recall the basic definitions used throughout the paper. In Section~\ref{sec:gBS definition}, we present the construction of generalized Bott-Samelson varieties and explain their relationship with classical Bott-Samelson varieties. In Section~\ref{map:pi}, we prove that generalized Bott-Samelson varieties have rational singularities. We then construct line bundles on these varieties and provide a characterization of very ample line bundles. Section~\ref{sec: Fano Xw} is devoted to the study of divisors and peaks on Schubert varieties. In Section~\ref{sec: Fano gBS}, we establish the relationship between peaks and divisors on generalized Bott-Samelson varieties. We then describe the canonical divisor in terms of peaks and provide a characterization of Fano and weak Fano generalized Bott-Samelson varieties.
In Section~\ref{sec:Recover}, we derive explicit formulas for the $c_{j,\alpha}$'s in the special cases of ($G$-)Bott-Samelson varieties and minuscule generalized Bott-Samelson varieties, from which we recover the known results. In Section \ref{sec:examples}, we present some examples of (weak)Fano and non-(weak)Fano generalized Bott-Samelson varieties.

\section{Preliminaries}\label{sec:basics}
In this section, we briefly set up the notation and review relevant definitions. For the necessary background on linear algebraic groups and Schubert varieties, we refer to \cite{BK07}, \cite{Hum2}, \cite{Bor}, and \cite{Spr}.

Let $G$ be a semisimple simply-connected complex algebraic group of rank $n$. Fix a maximal torus $T \subset G$ and a Borel subgroup $B \subset G$ containing $T$.
Let $R$ denote the root system of $(G, T)$ and let $R^+ \subset R$ be the subset of positive roots corresponding to $(B, T)$. 
Let $R^- = -R^+$ be the corresponding set of negative roots, and let
\[
S = \{ \alpha_1, \ldots, \alpha_n \} \subset R^+
\]
be the set of simple roots.  The Weyl group $W = N_G(T)/T$ is generated by the simple reflections $s_1, \ldots, s_n$ corresponding to the simple roots. For $w\in W$, we denote by $w^{\perp}$ the set of simple roots $\alpha$ such that $s_\alpha$ commutes with $w$. That is, $
w^\perp:=\{\alpha\in S\mid s_\alpha w=ws_\alpha\}.
$ 
Let $\mathfrak{g}$ and $\mathfrak{h}$ denote the Lie algebras of $G$ and $T$, respectively.
The group of all characters of $T$ is denoted by $X(T)$. Then
\[
X(T)\otimes_{\mathbb{Z}}\mathbb{R}
=\operatorname{Hom}_{\mathbb{R}}(\mathfrak{h}_{\mathbb{R}},\mathbb{R}),
\]
the dual of the real form $\mathfrak{h}_{\mathbb{R}}$ of $\mathfrak{h}$. The positive-definite $W$-invariant bilinear form on
$\operatorname{Hom}_{\mathbb{R}}(\mathfrak{h}_{\mathbb{R}},\mathbb{R})$ induced by the Killing form on $\mathfrak{g}$ is denoted by $(\cdot,\cdot)$. For any
$\mu\in X(T)\otimes_{\mathbb{Z}}\mathbb{R}$ and $\alpha\in R$, define
\[
\langle \mu,\alpha\rangle
:=\frac{2(\mu,\alpha)}{(\alpha,\alpha)}=(\mu, \alpha^\vee).
\]

For any root $\beta \in R$, denote by $U_\beta \subset G$ the associated root subgroup and by $G_\beta \subseteq G$ the subgroup generated by $U_\beta$ and $U_{-\beta}$. For any simple root $\alpha\in S$,  the maximal and minimal standard parabolic subgroups  associated with $\alpha$ denoted by $P^{\alpha}$ and  $P_{\alpha}$, respectively.
We also consider the coroot system $R^\vee$, whose simple coroots $\alpha_1^\vee, \ldots, \alpha_n^\vee$ form a basis of the cocharacter lattice $Y(T)$. 
The dual basis of the character lattice $X(T)$ consists of the fundamental weights $\omega_1, \ldots, \omega_n$. More precisely, for any simple root $\alpha$, we denote by $\omega_\alpha$ the fundamental weight defined by
\[
{\langle \omega}_\alpha , \beta\rangle =
\begin{cases}
1 & \text{if } \beta = \alpha, \\
0 & \text{if } \beta \neq \alpha,
\end{cases}
\]
for all simple roots $\beta$. Let $\rho:=\omega_1+\cdots+\omega_n$.  
Then the height of any coroot $\gamma^{\vee} \in R^\vee$ is defined by $\operatorname{ht}(\gamma^{\vee}):=\langle \rho,\gamma \rangle$.
Equivalently, if $\gamma^\vee=\sum_{i=1}^n a_i\alpha_i^\vee$, then $\operatorname{ht}(\gamma^\vee)=\sum_{i=1}^n a_i$.
 Let $X$ be a projective variety equipped with a transitive $G$-action. Choose a point $x\in X$ and assume that the stabilizer subgroup $G_x$ is connected. Then $X$ can be identified with the homogeneous space $G/P$, where $P:=G_x$ is a parabolic subgroup of $G$. Under this identification, the points $x \in X$ and $eP \in G/P$ are both referred to as base points.

Given a subset $I \subseteq S$, let $P_I$ denote the standard parabolic subgroup of $G$ corresponding to $I$, that is, $P_I$ is the subgroup of $G$ generated by $B$ and the root subgroups $U_{-\alpha}, \alpha \in I$. Let $W_I$ be the subgroup of $W$ generated by the simple reflections $s_\alpha$ for $\alpha \in I$. Then
\[
W_I = N_{L_I}(T)/T,
\]
where $L_I$ is a Levi factor of $P_I$. Define $W^I$ as the subset of $W$ consisting of elements $w$ such that $w(\alpha) \in R^+$ for all $\alpha \in I$; this subset is often denoted by $W^{P_I}$ in the literature. Then $W^I$ forms a set of minimal length representatives of the cosets in $W/W_I$.

For $w \in W$, the \textit{support} $\operatorname{Supp}(w)$ of $w$ is the set of simple roots that appear in a reduced expression for $w$. Let $G_w$ denote the subgroup of $G$ generated by the root subgroups $U_{\pm \alpha}$ for $\alpha \in \operatorname{Supp}(w)$. Then $G_w$ is the derived subgroup of the Levi subgroup $L_{\operatorname{Supp}(w)}$ and is therefore semisimple, normalized by $T$ and contains a representative of $w$.

Let $P^w$ be the largest parabolic subgroup of $G$ such that $B \subseteq P^w$ and $w \in W^{P^w}$. Then
\[
P^w = P_{I^w}, \quad \text{where } I^w := \{ \alpha \in S \mid w(\alpha) \in R^+ \}.
\]
Equivalently, $P^w$ is generated by $B$ and $U_{-\alpha}$ with $\alpha\in I^{w}$. 
Consider the base point $x = eP^w \in G/P^w$ and the point $wx \in G/P^w$. The $B$-orbit $Bwx$ is a locally closed subvariety of $G/P^w$ and the associated \textit{Schubert variety} $X(w)$ is defined as its Zariski closure:
\[
X(w) := \overline{Bwx} \subseteq G/P^w.
\]

Let $P_w$ be the subgroup of $G$ stabilizing $X(w)$, i.e.,
$
P_w := \{ g \in G \mid g X(w) = X(w) \}.
$
Then $P_w$ is a parabolic subgroup containing $B$, and is of the form $P_{I_w}$, where
\[
I_w := \{ \alpha \in S \mid s_\alpha w \leq w \}
\]
with respect to the Bruhat order on $W^{I^w}$. 
Finally, observe that $P_w \cap G_w$ is a parabolic subgroup of $G_w$ corresponding to $I_w\cap \operatorname{Supp}(w)$, and we have an isomorphism
\[
X(w) = \overline{P_w w P^w}/P^w \simeq \overline{(P_w \cap G_w) w (P^w \cap G_w)}/(P^w \cap G_w) \subseteq G_w / (P^w \cap G_w).
\]

\smallskip

The following lemma will be used for the computations in the sequel.

\begin{lemma}\label{lem: Supp(w)}
    For any $w\in W$, we have $ S\setminus I^{w}\subseteq \operatorname{Supp}(w).$
\end{lemma}

\begin{proof}
   Let \( w \in W \) and 
$
\tilde w = s_{i_1}s_{i_2}\cdots s_{i_r}
$
be a reduced expression of $w$. 
Then, for any simple root \( \alpha \in S \), one has
\begin{equation}\label{eq:action-on-beta}
    s_{i_1}s_{i_2}\cdots s_{i_r}(\alpha)
    = \alpha 
      - \sum_{j = 1}^{r} 
        \langle \alpha, \alpha_{i_j} \rangle \, \beta_j,
\end{equation}where $\beta_j:= s_{i_1}\cdots s_{i_{j-1}}(\alpha_{i_j})$. Note that $\beta_{j}>0$ for all $1\leq j\leq r$.
Moreover, if $\alpha\in S\setminus \text{Supp}(w)$ then $\langle \alpha, \alpha_{i_j} \rangle \leq 0$
for all $1\leq j \leq r$. Then by Eq.\eqref{eq:action-on-beta}, we have $w(\alpha) = s_{i_1}s_{i_2}\cdots s_{i_r}(\alpha)  >0.$ Therefore, $ S\setminus \text{Supp}(w)\subseteq I^w$ and $ S\setminus I^w \subseteq \text{Supp}(w).$ 
\end{proof}

\section{Generalized Bott-Samelson varieties \texorpdfstring{$\widehat{X}(\widehat{w})$}{X(w-hat)} }\label{sec:gBS definition}

In this section, we outline a generalization of the classical Bott-Samelson construction, based on suitable factorizations in the Weyl group. We follow \cite{Per07}.

We first recall the following action:
Let an algebraic group $H$ act on varieties $X$ and $Y$, on the right on $X$ and on the left on $Y$. The \emph{twisted fiber product} $X \times^H Y$ is defined as the quotient of $X \times Y$ by the relation
\[
(x, y) \sim (xh, h^{-1}y), \quad \text{for all } h \in H.
\]
For a given $w\in W$, let us denote 
\[
\widehat{X}(w):=\overline{(P_w \cap G_w) w (P^w \cap G_w)}/(P^w \cap G_w).
\]
Note that $\widehat{X}(w)\simeq X(w)\subseteq G/P^w$. 

Let $w, w_1, w_2 \in W$ such that $w = w_1w_2$ and $P^{w_1} \cap G_{w_1} \subseteq P_{w_2}$. Define
\[
\widehat{X}(w_1,w_2) := \overline{(P_{w_1} \cap G_{w_1}) w_1 (P^{w_1} \cap G_{w_1})} \times^{P^{w_1} \cap G_{w_1}} \widehat{X}(w_2),
\]
where $\widehat{X}(w_2) \simeq X(w_2)$. 
This variety is projective and admits a Zariski locally-trivial fibration
\[
f_{w_1,w_2} : \widehat{X}(w_1,w_2) \to \widehat{X}(w_1), \quad [p_1, p_2] \mapsto [p_1],
\]
with fiber $\widehat{X}(w_2)$ and is equivariant under $P_{w_1} \cap G_{w_1}$ acting on the first component. 

There is also a morphism
\[
\pi_{w_1,w_2} : \widehat{X}(w_1,w_2) \to G/P^w, \quad [p_1, p_2] \mapsto p_1p_2P^w.
\]
This construction extends inductively under suitable conditions.

\begin{definition}\label{def:gBS} Let $w\in W$. Define:

\begin{enumerate} 
    \item 
A sequence $\widehat{w} = (w_1, \dots, w_m)$ in $W$ is a \emph{generalized decomposition} of $w$ if $w = w_1 \cdots w_m$.
\item For any generalized decomposition $\widehat{w}$ of $w$, inductively define a sequence of parabolic subgroups $(P{_{J_i}})_{1\leq i\leq m}$ as follows: \[\quad {J_m}:=I_{w_m}\quad \text{and} \quad J_i:=( J_{{i+1}}\cup\operatorname{Supp}(w_i))\cap( w_{i}^{\perp} \cup \operatorname{Supp}(w_i))\cap (I_{w_i}\cup \operatorname{Supp}(w_i)^c),\]
here $\operatorname{Supp}(w_i)^c:=S\setminus \operatorname{Supp}(w_i)$.
\item A generalized decomposition $\widehat{w}$ of $w$ is \emph{admissible} if for all $1\leq i<m$:\[
P^{w_i}\cap G_{w_i}\subseteq P_{J_{i+1}}, \quad\text{equivalently}\quad I^{w_i}\cap \operatorname{Supp}(w_i)\subseteq J_{i+1}.
\]
\item 
 A generalized decomposition $\widehat{w}$ of $w$ is \emph{good} if for all $1 \le i < m$, we have
\[
P^{w_i} \cap G_{w_i} \subseteq P_{w_{i+1}\cdots w_m}
\quad \text{and} \quad
 I_{w_i \cdots w_m} \subseteq w_i^\perp \cup \operatorname{Supp}(w_i).
\]
\end{enumerate}
\end{definition}

\begin{definition}
An admissible generalized decomposition $\widehat{w}$ of $w\in W$ is \emph{reduced} if $\ell(w) = \sum_{j=1}^m \ell(w_j)$.
\end{definition}

    We refer to \cite[Remark 3.2]{BS26} for comparisons between the notions of admissible and good decompositions.
    
    Throughout this paper, we assume that $\widehat{w}$ is an admissible generalized (not necessarily reduced) decomposition, unless otherwise stated.

\begin{remark}\label{rmk: Parabolic subset} \ 

\begin{enumerate}
\item The condition $P^{w_i} \cap G_{w_i} \subseteq P_{w_{i+1}\cdots w_m}$ given in Definition \ref{def:gBS}, is equivalent to 
\[
I^{w_i}\cap \operatorname{Supp}(w_i)\subseteq I_{w_{i+1}\cdots w_m}.\]

    \item If $\widehat{w} = (w_1,\dots,w_m)$ is an admissible generalized decomposition of $w$, then each tail $(w_{i+1},\dots,w_m)$ is also an admissible generalized decomposition of $w_{i+1}\cdots w_m$.
    \end{enumerate}
\end{remark}
For any $1\leq i<m$, since $
I_{w_{i+1}}\cap\operatorname{Supp}(w_{i+1})\subseteq J_{i+1}$,
the action of $P_{w_{i+1}}\cap G_{w_{i+1}}$ on $\widehat{X}(w_{i+1},\dots,w_m)$ extends to an action of $P_{J_{i+1}}$ on $\widehat{X}(w_{i+1},\dots,w_m)$; see \cite[Lemma~5.1]{Per07}. The condition $P^{w_i}\cap G_{w_i}\subseteq P_{J_{i+1}}$ in the definition of admissible ensures that there exists a well defined twisted fiber product \[\widehat{X}(w_i,\dots,w_m):=\overline{(P_{w_i}\cap G_{w_i})w_i(P^{w_i}\cap G_{w_i})}\times^{P^{w_i}\cap G_{w_i}} \widehat{X}(w_{i+1},\dots,w_{m}).\] Thus, we have the following definition. 

\begin{definition}
    Let $w\in W$ and $\widehat{w} = (w_1, \dots, w_m)$ be an admissible 
    generalized decomposition of $w$. The associated projective variety $\widehat{X}(\widehat{w})$ is defined recursively and admits a Zariski locally trivial fibration
\[
\widehat{f} : \widehat{X}(\widehat{w}) \to \widehat{X}(w_1), \quad [p_1,\dots,p_m] \mapsto [p_1],
\]
with fiber $\widehat{X}(w_2,\dots,w_m)$. This is called the \emph{generalized Bott-Samelson variety}.

\end{definition}

Recall that the action of $P_{w_1} \cap G_{w_1}$ on $\widehat{X}(\widehat{w})$ extends to the action of the parabolic subgroup $P_{J_1}$, see \cite[Lemma 5.1(ii)]{Per07}. Moreover, there is an $(P_{w_1} \cap G_{w_1})$-equivariant morphism
\begin{equation}\label{eq:pi map}
\widehat{\pi}: \widehat{X}(\widehat{w}) \to {X}(w), \quad [p_1,\dots,p_m] \mapsto p_1 \cdots p_m P^w.
\end{equation}

\begin{lemma}\label{rmk:smoothness}
 Let $w\in W$ and $\widehat{w} = (w_1, \dots, w_m)$ be an 
 admissible generalized decomposition of $w$. 
\begin{enumerate}
    \item  $\widehat{X}(\widehat{w})$ is normal and Cohen-Macaulay, and 
    \item  $\widehat{X}(\widehat{w})$ is smooth (locally factorial) if and only if the Schubert variety $X(w_i)$ is smooth (resp. locally factorial) for $i = 1, \dots, m$.  
\end{enumerate}
    \end{lemma}

\begin{proof}
   It is well known that each Schubert variety $X(w_i)$ is both Cohen--Macaulay 
and normal. Since these properties are local and the morphism 
\[
\widehat{f} : \widehat{X}(\widehat{w}) \to \widehat{X}(w_1)
\]
is a locally trivial fibration with fiber $\widehat{X}(w_2, \dots, w_m)$, 
it follows by induction on $m$ that $\widehat{X}(\widehat{w})$ also has these properties. This proves (1). Similarly, the proof of (2) follows. 
\end{proof}

\subsection{Link with Bott-Samelson Varieties}\label{sec:linktoBS}

In this subsection, we obtain a Bott-Samelson variety as a desingularization of $\widehat{X}(\widehat{w})$.

Let $w\in W$ and let $\tilde{w}=s_{i_1}\cdots s_{i_r}$ be an expression (not necessarily reduced)
of $w$ in terms of simple reflections. Then the associated \textit{Bott-Samelson variety} $Z(\tilde{w})$ is the quotient 
\begin{center}
    $Z(\tilde{w}):=\big(P_{\alpha_{i_1}}\times P_{\alpha_{i_2}}\times\cdots\times P_{\alpha_{i_r}}\big)/B^{r}$
\end{center}
where $B^r$ acts from right on $P_{\alpha_{i_1}}\times P_{\alpha_{i_2}}\times\cdots\times P_{\alpha_{i_r}}$ by \begin{center}
    $(p_1,\dots,p_r)\cdot(b_1,\dots,b_r):=(p_1b_1, b_1^{-1}p_2b_2,\dots,b_{r-1}^{-1}p_rb_r)$
\end{center}
for $p_j\in P_{\alpha_{i_j}}$ and $b_{j}\in B$ for all $1 \leq j \leq r$. Moreover, if the expression $\tilde w$ is reduced then $Z(\tilde{w})$ is a desingularization of the Schubert variety $X(w)$.

\begin{remark}
    We may also view $Z(\tilde{w})$ as 
 $P_{\alpha_{i_1}}\times^B P_{\alpha_{i_2}}\times^B\cdots \times^B P_{\alpha_{i_r}}/B$. Thus, it is a twisted product of $\mathbb P^1$'s.
\end{remark}

\begin{lemma}\label{lem:reduced decomposition}
Let $\widehat{w} = (w_1, \dots, w_m)$ be an admissible generalized decomposition of $w\in W$, and
$\tilde w_j= \prod_{k=1}^{r_j} s_{k,j}$, where $r_j=\ell(w_j)$ for all $1\leq j\leq m$. Then the sequence
\[
\tilde{w} = (s_{1,1}, \dots, s_{r_1,1}, \dots, s_{1,m}, \dots, s_{r_m,m})
\]
is also an admissible generalized decomposition of $w$.
\end{lemma}
\begin{proof}
   This follows readily from the definition as the sequence of simple reflections is admissible.
\end{proof} 

Consider the natural morphism $
\tilde{\pi}: Z(\tilde{w}) \to \widehat{X}(\widehat{w}),
$
given by
\begin{equation}\label{eq:tilde pi}
[p_{1,1}, \dots, p_{r_1,1}, \dots, p_{1,m}, \dots, p_{r_m,m}] 
\mapsto 
\left[\prod_{k=1}^{r_1} p_{k,1}, \dots, \prod_{k=1}^{r_m} p_{k,m}\right],
\end{equation} 
where $p_{k,j}\in P_{\alpha_{k,j}}$ for all $1\leq j \leq m $ and $1\leq k \leq r_j$. 

\begin{corollary}\label{cor:birational} 
The morphism $
\tilde{\pi}: Z(\tilde{w}) \to \widehat{X}(\widehat{w})
$ is birational.
\end{corollary}
\begin{proof} 
Since $\tilde w_j= \prod_{k=1}^{r_j} s_{k,j}$ is a reduced expression for each $1\leq j\leq m$, we have the Bott-Samelson resolution $Z(\tilde w_j)\to X(w_j)$. Then, by the construction of $\widehat{X}(\widehat{w})$, the morphism $\tilde{\pi}$ sends the open $B$-stable neighbourhood $B\tilde{x}$ of the base point $\tilde{x}=[s_{1,1}, \dots, s_{r_1,1}, \dots, s_{1,m}, \dots, s_{r_m,m}] \in Z(\tilde{w})$ isomorphically onto the open $B$-stable neighbourhood $B\widehat{x}$ of the base point $\widehat{x} =[w_1, \ldots, w_m]\in \widehat{X}(\widehat{w})$.  Hence, the result follows.   
\end{proof}

\begin{remark}
    If $\widehat{w}$ is not a reduced decomposition, then $\widehat{X}(\widehat{w})$ is not birational to the Schubert variety $X(w)$.
\end{remark}

\section{Line bundles on \texorpdfstring{$\widehat{X}(\widehat{w})$}{X(w-hat)}}\label{map:pi}

In this section, we construct a basis for the Picard group of $\widehat{X}(\widehat{w})$ and give a characterization of very ample and nef line bundles. As a consequence, we prove that every ample line bundle on $\widehat{X}(\widehat{w})$ is very ample. Moreover, we show that a line bundle on $\widehat{X}$ is globally generated if and only if it is nef. 

In this direction, we first briefly recall the line bundles on  Bott-Samelson varieties studied by Lauritzen and Thomsen in \cite{LT04}. We then extend the construction to the generalized Bott-Samelson varieties.

\subsection{Line bundles on \texorpdfstring{$Z(\tilde{w})$}{Z(tilde w)}} \label{subsec:line}

Let $w \in W$ with an expression $\tilde w = s_{i_1} \cdots s_{i_m}$. We denote the sequence  $(i_1, \dots, i_m)$ by $\bf i$. Then, $\widehat{w} = (s_{i_1}, \dots, s_{i_m})$ is an admissible generalized decomposition. For $1 \leq j \leq m$, note that $P^{s_{i_j}} = P_{S \setminus \{\alpha_{i_j}\}}$ is a maximal parabolic subgroup with $G_{s_{i_j}} = G_{\alpha_{i_j}}$, and $P_{s_{i_j}} = P_{\{\alpha_{i_j}\} \cup s_{i_j}^\perp}$. Thus,
\[
P^{s_{i_j}} \cap G_{s_{i_j}} = B \cap G_{\alpha_{i_j}}
\]
is a Borel subgroup of $G_{\alpha_{i_j}}$. It follows that
\begin{align*}
\widehat{X}(s_{i_1}, \dots, s_{i_m}) &= G_{\alpha_{i_1}} \times^{B \cap G_{\alpha_{i_1}}} \widehat{X}(s_{i_2}, \dots, s_{i_m}) \\
&\simeq P_{\alpha_{i_1}} \times^B \widehat{X}(s_{i_2}, \dots, s_{i_m}) \\
&\simeq P_{\alpha_{i_1}} \times^B \cdots \times^B P_{\alpha_{i_m}}/B,
\end{align*}
which is a Bott-Samelson variety. Here, $\widehat{w} = \tilde{w}$ and $\widehat{X}(\widehat{w}) \simeq Z(\tilde{w})$.

For $1 \leq j \leq m$, let $Z({\tilde{w}}(j))$ be the Bott-Samelson variety associated to the subexpression $\tilde{w}(j) = (s_{i_1}, \dots, s_{i_j})$. Consider the morphisms
\[
\widehat{\pi}_j \colon Z({\tilde{w}}(j)) \to G/B, \quad [p_1, \dots, p_j] \mapsto p_1 \cdots p_j B,
\]
and define  \begin{align*}    
\widehat{f}_j: Z(\tilde{w})&\to Z(\tilde{w}(j))\\ [p_1,\dots, p_m]&\mapsto [p_1,\dots,p_j]. \end{align*} 
Now, we define the line bundle on $Z(\tilde{w})$ by
\begin{equation}\label{eq: lines on BSvs}
\mathcal{L}_j := \widehat{f}_j^* \widehat{\pi}_j^* \mathcal{L}_{G/B}({\omega}_{i_j}),
\end{equation}
where $\mathcal{L}_{G/B}({\omega}_{i_j})$ is the line bundle on $G/B$ associated to the fundamental weight ${\omega}_{i_j}$. For a $m$-tuple $\mathbf{n} = (n_1, \dots, n_m) \in \mathbb{Z}^m$, set
\[
\mathcal{L}_{\mathbf{i}, \mathbf{n}} := \mathcal{L}_1^{n_1} \otimes \cdots \otimes \mathcal{L}_m^{n_m}.
\]
\begin{theorem}{\cite[Theorem~3.1]{LT04}}\label{thm:LT line bundles}
\begin{enumerate}\ 
    \item The isomorphism classes of $\mathcal{L}_1, \dots, \mathcal{L}_m$ form a basis of $\mathrm{Pic}(Z(\tilde{w}))$. In particular, the map $\mathbb{Z}^m \to \mathrm{Pic}(Z(\tilde{w}))$ given by $\mathbf{n} \mapsto \mathcal{L}_{\mathbf{i}, \mathbf{n}}$ is an isomorphism.
    \item The line bundle $\mathcal{L}_{\mathbf{i}, \mathbf{n}}$ is (very) ample if and only if $n_j > 0$ for all $1\leq j \leq m$.
\end{enumerate}
\end{theorem}

\smallskip

\subsection{Line bundles on \texorpdfstring{$\widehat{X}(\widehat{w})$}{X(w-hat)}}
 Let $w\in W$ and $\widehat{w} = (w_1, \dots, w_m)$ be an admissible generalized decomposition of $w$.
We begin by recalling that the variety $\widehat{X}(\widehat{w})$ admits a closed embedding into the product $\prod_{j=1}^{m} G/P^{w_j}$. More precisely, we have the following.

\begin{proposition}[\cite{BS26},~Proposition~4.2] \label{line}
The variety $\widehat{X}(\widehat{w})$ is embedded as a closed subvariety inside $\prod_{j=1}^{m} G/P^{w_j}$. The embedding is given by  
\begin{align*} 
\Phi_{\widehat{w}}: \widehat{X}(\widehat{w}) &\to \prod_{j=1}^{m} G/P^{w_j} \\ 
[p_1,\dots,p_m] &\mapsto (p_1 P^{w_1}, p_1 p_2 P^{w_2}, \dots, (\prod_{j=1}^{m} p_j) P^{w_m}),
\end{align*} 
where $p_j \in \overline{(P_{w_j} \cap G_{w_j}) w_j (P^{w_j} \cap G_{w_j})} \subseteq G_{w_j}$ for all $1 \leq j \leq m$.
\end{proposition} 

\begin{remark}    
We note that in \cite[Proposition~4.2]{BS26}, we assumed that the decomposition was reduced. However, the same proof works for arbitrary decompositions, not necessarily reduced.
\end{remark}

To describe the line bundles on $\widehat{X}(\widehat{w})$, we use an equivalent definition for the Bott-Samelson variety $Z({\tilde{w}})$ associated with an expression $\tilde w$ of $w\in W$ constructed by Magyar in \cite{Mag98}.

Let $\tilde{w} = (s_{i_1}, \dots, s_{i_m})$ be an expression of $w\in W$ and let $\pi_{j}: G/B \to G/P_{\alpha_{i_{j}}}$ be the projection morphism whose fibers are isomorphic to $P_{\alpha_{i_j}}/B$. Note that $P_{\alpha_{i_j}}$ is the minimal parabolic subgroup, so $P_{\alpha_{i_j}}/B \cong \mathbb{P}^1$. Let $\mathbb{P}(x,\alpha_{i_j})$ denote the projective line $\pi^{-1}_{j}(\pi_{j}(x))$ for $x\in G/B$.

For any $x\in G/B$ the map $G/B\to G/P^{\alpha_{i_j}}$ restricted to $\mathbb{P}(x,\alpha_{i_j})$ is an isomorphism onto its image called $\overline{\mathbb{P}}(x,\alpha_{i_j})$. If $y\in G/B$ and $\overline{x}\in \overline{\mathbb{P}}(y,\alpha_{i_{j}})$ for $x\in \mathbb{P}(y, \alpha_{i_{j}})$ we abuse the notation by writing $\overline{\mathbb{P}}(\overline{x}, \alpha_{i_{j+1}})=\overline{\mathbb{P}}({x}, \alpha_{i_{j+1}})$. Under the morphism $\Phi_{\tilde{w}}$ defined in Proposition \ref{line}, the variety $Z(\tilde{w})$ is isomorphic to \begin{equation}\label{eq: equivalent to Zw} \tilde{X}(\tilde{w}):= \Big\{\, (x_1,\dots,x_m) \in \prod_{j=1}^{m} G/P^{s_{i_j}} \;\Big|\; 
x_0 = eB \ \text{and} \ x_j \in \overline{\mathbb{P}}(x_{j-1}, \alpha_{i_j}) \ \text{for all } 1 \leq j \leq m \,\Big\}
\end{equation}
 see, \cite[Theorem~1]{Mag98} and \cite[Section~2]{Per07}. In this setting, for $1\leq j \leq m$, we have natural projections $pr_j: \tilde{X}(\tilde{w})\to G/P^{s_{i_j}}$ and we define the line bundles \[\mathcal{M}_j:= pr_j^*(\mathcal{O}_{G/P^{s_{i_j}}}(\omega_{\alpha_{i_j}}))\] on $\tilde{X}(\tilde{w})$.

 \begin{remark}\label{remark3} For $1\leq j \leq m$,  let $\pi^j:G/B\to G/P^{s_{i_j}}$ be a natural morphism.
Consider the following commutative diagram
 \begin{center}
\begin{tikzpicture}[scale=1.5, >=Stealth, every node/.style={font=\small}]

  \node (P1) at (-1,6) {$Z(\tilde{w})$};
  \node (P2) at (-1,4.5) {$Z(\tilde{w}(j))$};
  \node (P3) at (2.5,4.5) {$G/B$};
  \node (P4) at (2.5,6) {$\tilde{X}(\tilde{w})\subseteq \prod_{k=1}^{m}G/P^{s_{i_k}}$};
  \node (P5) at (4,4.5) {$G/P^{s_{i_j}}$};

  \draw[->] (P1) -- node[left] {$\widehat{f}_j$} (P2);
  \draw[->] (P2) -- node[below] {$\widehat{\pi}_j$}(P3);
  \draw[->] (P1) -- node[above] {$\Phi_{\tilde{w}}$}(P4);
  \draw[->] (P4) -- node[right=8pt] {$pr_j$}(P5);
  \draw[->] (P3) -- node[below] {$\pi^j$}(P5);
\end{tikzpicture}
\end{center}

 Then, we have 
 \begin{center}
$pr_j \circ \Phi_{\tilde{w}}= \pi^j\circ\widehat{\pi}_{j}\circ \widehat{f}_j $
 \end{center}
 for all $1\leq j\leq m$. Therefore, the pullback of the line bundle $ \mathcal{M}_j$ under $\Phi_{\tilde{w}}$ to $Z({\tilde{w}})$ is the same as the line bundle $\mathcal{L}_j$. That is,\[
\mathcal{L}_j = \Phi_{\tilde{w}}^*\big( \mathcal{M}_j \big)
\quad \text{for all } 1 \leq j \leq m.
\]
 \end{remark}
 
\begin{proposition}

The set $ \big\{ \mathcal{M}_j \mid 1 \leq j \leq m \big\} $ forms a basis for $ \operatorname{Pic}(\tilde{X}(\tilde{w})) $.
\end{proposition}
\begin{proof}
    Since $Z(\tilde{w})\simeq \tilde{X}(\tilde{w})$ and by Theorem \ref{thm:LT line bundles}(1), the set $ \big\{ \mathcal{L}_j \mid 1 \leq j \leq m \big\} $ forms a basis for $Z(\tilde{w})$, by Remark \ref{remark3} we get the result.
\end{proof}

Now, we define the line bundles on $\widehat{X}(\widehat{w})$ as follows:

For $1\leq j \leq m$, let $pr_j: \prod_{i=1}^{m}G/P^{w_i}\to G/P^{w_j}$ be the projections and recall the map $\Phi_{\widehat{w}}$ from Proposition \ref{line} and set
\begin{equation}\label{eq: Line bundles on gBS}
\mathcal{L}_{j,\alpha} := \Phi_{\widehat{w}}^* pr_j^* \mathcal{O}_{G/P^{w_j}}(\omega_\alpha),
\end{equation}
where $\mathcal{O}_{G/P^{w_j}}(\omega_{\alpha})$ is the homogeneous line bundle on $G/P^{w_j}$ associated to the fundamental weight $\omega_{\alpha}$ with $\alpha\in (S\setminus I^{w_j})\cap\operatorname{Supp}(w_j)$. Note that by Lemma \ref{lem: Supp(w)}, we have $(S\setminus I^{w_j})\subseteq \operatorname{Supp}(w_j)$.

\begin{lemma}\label{lem: rank inequality}
  Let $w\in W$ and $\widehat{w} = (w_1, \dots, w_m)$ be an admissible generalized 
  decomposition of $w$. Then the set of line bundles $\{\mathcal{L}_{j,\alpha}\mid 1\leq j\leq m, \alpha\in S\setminus I^{w_j}\}$ is linearly independent. In particular, \[
    \operatorname{rank}\big(\operatorname{Pic}\widehat{X}(\widehat{w})\big)\geq \sum_{j=1}^{m}\#(S\setminus I^{w_j}).
    \]
\end{lemma}

To prove this lemma, we begin by constructing certain indexing sets associated with a generalized decomposition of an element $w\in W$. 

Consider the two admissible generalized decompositions of $w \in W$, namely $$\widehat{w} = (w_1, \dots, w_m)\quad \text{and}\quad \tilde{w} = (s_{1,1}, \dots, s_{r_1,1}, \dots, s_{1,m}, \dots, s_{r_m,m}),$$ as given by Lemma~\ref{lem:reduced decomposition}.
We denote the expression $\tilde{w}$ by the sequence 
\[{\bf i}=(i_{1,1}, \ldots, i_{r_1, 1}, \ldots, i_{1, r_m,},\ldots, i_{r_m, m}).\]

For $1\leq j \leq m$ and $1\leq d\leq r_j$, we set \begin{equation}\label{eq:m_j(d)}   
m_j(d): = \max \{q \mid 1\leq q\leq r_j~,~ \ s_{q,j}= s_{d,j}
\}.
\end{equation} Thus, $m_j(d)$ is the largest index at which the simple reflection $s_{d,j}$ appears in the chosen reduced expression of $\tilde w_j$. Let
$$m(\widehat{w}):=\{m_j(d)\mid 1\leq j\leq m, 1\leq d\leq r_j\}.$$

\begin{example}\label{ex:indexing sets}
    In $SL_4(\mathbb{C})$, consider the element $
w = s_1 s_2 s_3 s_1 s_2.$
Then, two admissible generalized   decompositions of $w$ are given by
\[
\widehat{w} = (w_1, w_2) = (s_1 s_2 s_3 s_1,\, s_2)
\]
and
\[
\tilde{w} = (s_1, s_2, s_3, s_1, s_2)
= (s_{1,1}, s_{2,1}, s_{3,1}, s_{4,1}, s_{1,2}).
\] For each $1\leq i\leq 2$, note that  $S\setminus I^{w_i}\subseteq \operatorname{Supp}(w_i)$. More precisely, we have
  \[
    S\setminus I^{w_1}= \{\alpha_1=\alpha_{1,1}=\alpha_{4,1}, \alpha_3=\alpha_{3,1}\} \quad \text{and} \quad S\setminus I^{w_2}= \{\alpha_2=\alpha_{1,2}\}.
    \]
    Then, by Eq.\eqref{eq:m_j(d)}, We have \[
    m_1(1)=m_1(4)=4 \quad \text{and} \quad m_2(1)=1.
    \]
\end{example}

\begin{proposition}\label{pull back} 
   For $1 \leq j \leq m$ and  $\alpha:=\alpha_{d,j}\in S\setminus I^{w_j}$ for some $1\leq d\leq r_j$,   
   let $\mathcal{L}_{j,\alpha}$ denote the line bundle on $\widehat{X}(\widehat{w})$ defined in Eq.\eqref{eq: Line bundles on gBS}. Then, for $a_{j,\alpha}\in \mathbb{Z}$, the pullback $\tilde{\pi}^*\mathcal{L}^{\otimes{a_{j,\alpha}}}_{j,\alpha}$ is given by $$\tilde{\pi}^*\mathcal{L}^{\otimes{a_{j,\alpha}}}_{j,\alpha} = \mathcal{L}_{\mathbf{i}, \mathbf{n}},$$ where $\mathbf{n}=(n_{1,1},\dots,n_{r_1,1},\dots,n_{1,m},\dots,n_{r_m,m})$ and the integers $n_{q,t}$ for $1\leq t\leq m$, $1\leq q \leq r_t$ with  \[
n_{q,t} = 
\begin{cases}
    a_{j,\alpha} & \text{if } t=j\,\, \text{and}\,\, q= m_j(d) , \\
    0 & \text{otherwise}.
\end{cases}
\]
\end{proposition}

\begin{proof}
   Since $S\setminus I^{w_j}\subseteq \operatorname{Supp}(w_j)$, by Eq.\eqref{eq:m_j(d)}, we get $\alpha= \alpha_{d,j}=\alpha_{m_j(d),j}$. For any
   $m_j(d)\in m(\widehat{w})$, we consider a natural morphism $\pi^{m_j(d),j}: G/P^{w_j}\to G/P^{s_{m_j(d),j}}$, and note that the following diagram commutes:
   \begin{center}
\begin{tikzpicture}[scale=1.5, >=Stealth, every node/.style={font=\small}]

  \node (P1) at (-1,6) {$Z(\tilde{w})$};
  \node (P2) at (-1,4.5) {$\widehat{X}(\widehat{w})$};
  \node (P3) at (1.8,4.5) {$\prod_{k=1}^{m}G/P^{w_k}$};
  \node (P4) at (2.5,6) {$\tilde{X}(\tilde{w})\subseteq \prod_{k=1}^{m}\prod_{i=1}^{r_k} G/P^{s_{i,k}}$};
  \node (P5) at (6,4.5) {$G/P^{s_{m_j(d),j}}$};
  \node (P6) at (4,4.5) {$G/P^{w_j}$};

  \draw[->] (P1) -- node[left] {$\tilde{\pi}$} (P2);
  \draw[->] (P2) -- node[below] {$\Phi_{\widehat{w}}$}(P3);
  \draw[->] (P1) -- node[above] {$\Phi_{\tilde{w}}$}(P4);
  \draw[->] (P4) -- node[right=8pt] {$pr_{m_j(d),j}$}(P5);
  \draw[->] (P6) -- node[below] {$\pi^{m_j(d),j}$}(P5);
  \draw[->] (P3) -- node[below] {$pr_j$}(P6);
\end{tikzpicture}
 \end{center}
   where $pr_{i,j}:\prod_{1\leq p\leq m}\prod_{1\leq q\leq r_p}G/P^{s_{q,p}}\to G/P^{s_{i,j}}$ is the projection. Then, we have \begin{equation}\label{eq:equality of maps}
   pr_{m_j(d),j}\circ \Phi_{\tilde{w}}= \pi^{m_j(d),j}\circ pr_j \circ \Phi_{\widehat{w}}\circ \tilde{\pi}.
   \end{equation}
   Therefore, by Remark \ref{remark3} and the above equality of maps, we obtain:\[
   \mathcal{L}_{m_j(d),j}= \tilde{\pi}^*\mathcal{L}_{j,\alpha}.
   \]  
   Hence, by the definition of $\mathcal L_{{\bf i}, n}$, the result follows.
\end{proof}  

\begin{example}\label{eg1} Notation is as in Example \ref{ex:indexing sets}. Then by Proposition \ref{pull back}, we have
\[\tilde{\pi}^*\mathcal{L}^{a_{1,\alpha_1}}_{1,\alpha_1}=\mathcal{L}_{\textbf{i},(0,0,0, a_{1,\alpha_1}, 0)}, ~
\tilde{\pi}^*\mathcal{L}^{a_{1,\alpha_3}}_{1,\alpha_3}=\mathcal{L}_{\textbf{i},(0,0,a_{1,\alpha_3},0,0)}
\]
and \[
\tilde{\pi}^*\mathcal{L}^{a_{2,\alpha_2}}_{2,\alpha_2}=\mathcal{L}_{\textbf{i},(0,0,0,0,a_{2,\alpha_2})}.
\]\end{example}

\begin{proof}[Proof of Lemma \ref{lem: rank inequality}]
    By Proposition \ref{pull back}, the set \[
    \{\tilde{\pi}^*\mathcal{L}_{j,\alpha}\mid 1\leq j\leq m, \alpha\in S\setminus I^{w_j} \}\] is a subset of the set $$\{\mathcal{L}_{k,i}\mid 1\leq i\leq m, 1\leq k\leq r_i\}.$$  The latter set forms a basis of $\operatorname{Pic}Z(\tilde{w})$, thanks to Theorem \ref{thm:LT line bundles}. 
  Suppose that \[\sum_{j=1}^m \sum_{\alpha\in S\setminus I^{w_j}}b_{j, \alpha}\mathcal L_{j, \alpha} =0 \quad \text{for some}\quad b_{j, \alpha} \in \mathbb Z\]
  By applying $\tilde{\pi}^*$, we obtain $\sum_{j=1}^m \sum_{\alpha\in S\setminus I^{w_j}}b_{j, \alpha} \tilde{\pi}^*\mathcal L_{j, \alpha} =0$. Hence, $b_{j,\alpha}=0$ for all $1\leq j\leq m$ and $\alpha\in S\setminus I^{w_j}$. This proves the lemma.     
\end{proof}

\smallskip

\subsubsection{Picard group of \texorpdfstring{$\widehat{X}(\widehat{w})$}{X(w-hat)}}  

To compute $\operatorname{Pic}\widehat{X}(\widehat{w})$, we proceed in two steps. 
First, we show that generalized Bott-Samelson varieties have rational singularities. 
Next, we prove the vanishing of the higher cohomologies of the structure sheaf on 
$\widehat{X}(\widehat{w})$, namely,
\[
H^i\big(\widehat{X}(\widehat{w}), \mathcal{O}_{\widehat{X}(\widehat{w})}\big) = 0 
\quad \text{for all } i > 0.
\]

\begin{definition}
An algebraic variety $Y$ is said to have \emph{rational singularities} if there exists 
a desingularization $f \colon X \to Y$ such that 
$f_* \mathcal{O}_X = \mathcal{O}_Y$ and 
$R^i f_* \mathcal{O}_X = 0$ for all $i \geq 1$.
\end{definition}

We recall the following theorem of Ren\'ee Elkik from \cite{Elk78}.

\begin{theorem}[\cite{Elk78},~Theorem~5]\label{thm: rational singularities for gBS}
    Let $f: X\to Y $ be a flat
morphism of complex algebraic varieties. Then $X$ has rational singularities if $Y$ and every fibre have rational singularities.
\end{theorem}

\begin{proposition}\label{prop:rational-sing}
The generalized Bott-Samelson varieties have rational singularities.
\end{proposition}

\begin{proof}
Recall that Schubert varieties have rational singularities (see \cite{Bri05}, Theorem~2.2.3). 
Moreover, locally trivial fibrations are flat. 
Applying Theorem~\ref{thm: rational singularities for gBS} and by induction on $m$, we obtain the desired result.
\end{proof}

We have the following result (see \cite{Kum02}, Lemma A.25). For the reader's convenience, we include a proof here for completeness.
\begin{proposition}\label{prop:cohoiso}
Let $Y$ be a projective variety with rational singularities, and let 
$f : X \to Y$ be a desingularization. Then, for any vector bundle $V$ on $Y$, 
the natural pullback map
\[
H^i(Y, V) \longrightarrow H^i(X, f^*V)
\]
is an isomorphism for all $i \geq 0$.
\end{proposition}

\begin{proof}
    The Leray-Serre spectral sequence yields 
\[
E_2^{p,q}=H^p(Y,R^q f_* f^*V)\Longrightarrow H^{p+q}(X,f^*V).
\]
By the projection formula (see \cite{Har77}, Exercise~8.3 of Chap. $\mathrm{III}$),
\[
R^q f_*(\mathcal{O}_X\otimes f^*V)\simeq (R^q f_*\mathcal{O}_X)\otimes V.
\]
Since $Y$ has rational singularities, $R^q f_*\mathcal{O}_X=0$ for $q>0$. Thus, $E_2^{p,q}=0$ for $q>0$, and hence $H^p(Y,V)\simeq E_2^{p,0}$ for all $p$, and the result follows.\end{proof}

 Recall that the generalized Bott-Samelson varieties are normal; see Lemma~\ref{rmk:smoothness}. 
 By Zariski's Main Theorem (see \cite{Har77}, Chap.~III, Cor.~11.4), the normality of $\widehat{X}(\widehat{w})$ implies that
\begin{equation}\label{Eq:rational singular}    
\tilde{\pi}_* \mathcal{O}_{Z(\tilde{w})}
= \mathcal{O}_{\widehat{X}(\widehat{w})},
\end{equation}
where $\tilde{\pi} \colon Z(\tilde{w}) \to \widehat{X}(\widehat{w})$ is the desingularization described in Corollary~\ref{cor:birational}.

\begin{proposition}\label{prop:Vanishing}
We have
\begin{equation}\label{eq:cohomology vanished for gBS}
H^i\big(\widehat{X}(\widehat{w}), \mathcal{O}_{\widehat{X}(\widehat{w})}\big) = 0
\quad \text{for all } i > 0.
\end{equation}
\end{proposition}

\begin{proof} By Propoistion~\ref{prop:rational-sing} and Proposition~\ref{prop:cohoiso}, we have $$H^i(\widehat{X}(\widehat{w}),\mathcal{O}_{\widehat{X}(\widehat{w})})\simeq H^i(Z(\tilde{w}),\mathcal{O}_{Z(\tilde{w})})$$ for all $i\geq 0$.
Recall that
\[
H^i\big(Z(\tilde{w}), \mathcal{O}_{Z(\tilde{w})}\big) = 0
\quad \text{for all } i > 0,
\]
see \cite[Theorem~8.1.8]{Kum02}. 
Thus, the result follows.
\end{proof}

\smallskip

By Proposition \ref{line}, we have the restriction map:
\begin{align*}
       res: \text{Pic}(\prod_{i=1}^{m}G/P^{w_i}) &\to \text{Pic}(\widehat{X}(\widehat{w})).
       \end{align*}

\begin{theorem}\label{thm:linebundles} We have:
\begin{enumerate} 
    \item The restriction map $res$ is an isomorpshim.  
    \item  The Picard group of $\widehat{X}(\widehat{w})$ is isomorphic to $\bigoplus_{j=1}^{m}\mathbb Z^{\#(S\setminus I^{w_j})}$. 
    \item  Any line bundle $\mathcal L$ on $\widehat{X}(\widehat{w})$ is of the form \begin{equation}\label{eq:linebundles}
\mathcal{L} = \bigotimes_{j=1}^m \bigotimes_{\alpha \in S \setminus I^{w_j}} \mathcal{L}_{j, \alpha}^{\otimes a_{j,\alpha}}, \quad\text{for some}\quad a_{j,\alpha}\in \mathbb Z.
\end{equation}
\end{enumerate} 
 \end{theorem}
\begin{proof}

We first prove $(2)$ and $(3)$. We proceed by induction on $m$. For $m=1$, \[
    \widehat{X}(\widehat{w})\simeq X(w_1)\subseteq G/P^{w_1},
    \]
    and for any $w\in W^{P^{w}}$, we have \[
    \operatorname{Pic}X(w)\simeq \bigoplus_{\alpha\in (S\setminus I^{w})\cap \operatorname{Supp}(w)}\mathbb{Z}\omega_{\alpha}.
    \]
    By Lemma \ref{lem: Supp(w)}, $S\setminus I^{w}\subseteq \operatorname{Supp}(w)$. Hence, \[
    \operatorname{Pic}X(w_1)\simeq \bigoplus_{\alpha\in S\setminus I^{w_1}}\mathbb{Z}\omega_{\alpha}\simeq  \mathbb{Z}^{\#(S\setminus I^{w_1})}.
    \]
    Together with Lemma~\ref{lem: rank inequality}, it follows that the set $\{\mathcal{L}_{1,\alpha}\mid \alpha\in S\setminus I^{w_1}\}$ forms a bases of $\operatorname{Pic}\widehat{X}(\widehat{w})$. 
    
    Suppose $m=2$. The variety $\widehat{X}(w_1,w_2)$ admits the Zariski locally trivial fibration \[
\widehat{f}_1:\widehat{X}(w_1,w_2)\to X(w_1)
\]
with fibers are isomorphic to $X(w_2)$. Recall that any Schubert variety $X(w)$ is normal and has rational singularities. We have \begin{equation}\label{eq:vanishes cohomology for Xw}
  H^i(X(w),\mathcal{O}_{X(w)})=0  
\end{equation}
for all $i>0$. By \cite[$\mathrm{III}$, Exercise~12.6(b)]{Har77}, for any open subset $U\subseteq X(w_1)$, we have \[
\operatorname{Pic}\big(X(w_2)\times U\big)\simeq \operatorname{Pic}X(w_2)\oplus \operatorname{Pic}(U).
\]
 Now by \cite[Theorem 5]{And75}, we have the following exact sequence: \begin{equation}\label{eq:exact seq}
     0\to \operatorname{Pic}X(w_1)\to \operatorname{Pic}\widehat{X}(w_1,w_2)\to \operatorname{Pic}X(w_2)\to H^2(X(w_1),\mathcal{O}_{X(w_1)}^*).
\end{equation}

$\textbf{Claim:}$ $H^2(X(w_1),\mathcal{O}^*_{X(w_1)})=0$.

 To prove this, we consider the exact sequence
 \[
0\to \mathcal{O}_{\mathbb{Z}}\to \mathcal{O}_{X(w_1)}\to \mathcal{O}^*_{X(w_1)}\to 0.
\]

We then get a long exact sequence:  \begin{equation}\label{eq: exact chain for Hi}
 \cdots\to H^2(X(w_1),\mathcal{O}_{X(w_1)})\to H^2(X(w_1),\mathcal{O}^*_{X(w_1)})\to H^3(X(w_1),\mathcal{O}_{\mathbb{Z}})\to\cdots.  \end{equation}
 
Since the homology group $H_*(X(w_1),\mathbb{Z})$ is free  
(see \cite[Lemma~3]{Kum98}), it follows that each group 
$H_n(X(w_1),\mathbb{Z})$ is also free abelian group. 
Then, we have
\[
\operatorname{Ext}_{\mathbb{Z}}\big(H_n(X(w_1),\mathbb{Z}),\mathbb{Z}\big)
= 0 
\quad \text{for all } n \geq 0.
\]

Since every Schubert variety admits a cell decomposition, we have 
\[
H_{\text{odd}}(X(w_1),\mathbb{Z})=0.
\]
Now, by the Universal Coefficient Theorem, we obtain
\begin{align*}
H^3(X(w_1),\mathbb{Z})
&\simeq 
H_3(X(w_1),\mathbb{Z})
\oplus 
\operatorname{Ext}_{\mathbb{Z}}\big(H_2(X(w_1),\mathbb{Z}),\mathbb{Z}\big) \\
&= 0.
\end{align*}

Moreover, by Eq.\eqref{eq:vanishes cohomology for Xw} and Eq.\eqref{eq: exact chain for Hi}, we obtain
\[
H^2(X(w_1),\mathcal{O}_{X(w_1)})=0 \quad \text{and so} \quad
H^2(X(w_1),\mathcal{O}^*_{X(w_1)})=0.
\]

Since $\operatorname{Pic} X(w_2)$ is a free $\mathbb{Z}$-module, the exact sequence in Eq.\eqref{eq:exact seq} splits. Hence,
\[
\operatorname{Pic}\widehat{X}(w_1,w_2)
=
\operatorname{Pic}X(w_1)
\oplus
\operatorname{Pic}X(w_2)
\simeq
\bigoplus_{i=1}^{2}
\mathbb{Z}^{\#(S\setminus I^{w_i})}.
\]

Applying Lemma~\ref{lem: rank inequality} with $m=2$,  the set
$
\{\mathcal{L}_{j,\alpha}\mid 1\leq j\leq 2,\ \alpha\in S\setminus I^{w_j}\}
$
forms a basis of $\operatorname{Pic}\widehat{X}(w_1,w_2)$.

Now, assume that the result holds for $m-1$. Consider the Zarisiki locally trivial fibration \[
\widehat{f}_{1}:\widehat{X}(\widehat{w})\to \widehat{X}({w_1})
\] with fibers are isomorphic to $Z:=\widehat{X}(w_2,\dots,w_m)$. By Proposition \ref{prop:Vanishing}, we have \begin{equation}\label{eq:vanishes}
  H^i(\widehat{X}(\widehat{w}),\mathcal{O}_{\widehat{X}(\widehat{w})})=0  
\end{equation}
for all $i>0$. By an argument similar to above, we obtain the following exact sequence: \[
0\to \operatorname{Pic}{X}({w_1})\to \operatorname{Pic}\widehat{X}(\widehat{w})\to \operatorname{Pic}(Z)\to H^2({X}({w_1}),\mathcal{O}_{X(w_1)}^*).
\]
 By induction $\operatorname{Pic} Z$ is a free $\mathbb{Z}$-module.
Since $H^2(X(w_1),\mathcal{O}^*_{X(w_1)})=0$, the above exact sequence splits. Hence, we have
\begin{align*}
\operatorname{Pic}\widehat{X}(\widehat{w})
&= \operatorname{Pic}(X(w_1)) \oplus \operatorname{Pic} Z 
\simeq \bigoplus_{i=1}^{m}\mathbb{Z}^{\#(S\setminus I^{w_i})}.
\end{align*}
Now, by Lemma~\ref{lem: rank inequality}, the set 
\[
\{\mathcal{L}_{j,\alpha} \mid 1\leq j\leq m,\ \alpha\in S\setminus I^{w_j}\}
\]
forms a basis for $\operatorname{Pic}\widehat{X}(\widehat{w})$. This proves $(2)$ and $(3)$.

We now prove $(1)$. By \cite[Ch.~III, Ex.~12.6(b)]{Har77}, we have
\[
\operatorname{Pic}\!\left(\prod_{j=1}^{m} G/P^{w_j}\right)
= \bigoplus_{j=1}^{m} \operatorname{Pic}(G/P^{w_j}).
\]
We also have,
\[
\operatorname{Pic}(G/P^{w_j})
\simeq \bigoplus_{\alpha \in (S\setminus I^{w_j}) \cap \operatorname{Supp}(w_j)} \mathbb{Z}\omega_\alpha.
\]
By Lemma~\ref{lem: Supp(w)}, we get
\[
\operatorname{Pic}(G/P^{w_j})
\simeq \mathbb{Z}^{\#(S\setminus I^{w_j})}.
\]
Thus,
\[
\operatorname{Pic}\!\left(\prod_{j=1}^{m} G/P^{w_j}\right)
\simeq \bigoplus_{j=1}^{m} \mathbb{Z}^{\#(S\setminus I^{w_j})}.
\]
Hence, the map $res$ is an isomorphism. This completes the proof.
\end{proof}

The line bundles $\mathcal{L}$ on $\widehat{X}(\widehat{w})$ obtained in Eq.\eqref{eq:linebundles} of Theorem~\ref{thm:linebundles} are referred to as the line bundles associated with the $a_{j,\alpha}$'s. By abuse of notation, we may also use additive notation in place of tensor product notation for the line bundles on $\widehat{X}(\widehat{w})$.
Using arguments similar to those in the proof of \cite[Theorem 4.18]{BS26}, we obtain a characterization of (very) ample and nef line bundles on $\widehat{X}(\widehat{w})$.

\begin{theorem} \label{thm:very ample}
    We keep the notation as in Theorem~\ref{thm:linebundles}. Let $\mathcal{L}$ be a line bundle on $\widehat{X}(\widehat{w})$. Then $\mathcal{L}$ is very ample (nef) if and only if 
    \[
    a_{j,\alpha} > 0 ~(\text{resp.} \geq 0) \quad \text{for all } 1 \leq j \leq m \text{ and for all } \alpha \in S \setminus I^{w_j}.
    \]
        \end{theorem}

As an immediate consequence, we obtain the following two corollaries.

\begin{corollary}\label{cor:ample}
   Any ample line bundle on $\widehat{X}(\widehat{w})$ is very ample.
\end{corollary}

\begin{corollary}\label{cor:globally}
    Let $\mathcal{L}$ be a line bundle on $\widehat{X}(\widehat{w})$. Then,
        $\mathcal{L}$ is globally generated if and only if 
        $\mathcal{L}$ is nef.
   
\end{corollary}

\begin{proof}  
   Assume that there exists a globally generated line bundle $\mathcal{L}$ such that 
$a_{j,\alpha} < 0$ for some $1 \leq j \leq m$ and $\alpha \in S \setminus I^{w_j}$. 
Since the tensor product of an ample line bundle with a globally generated line bundle is ample, that is very ample by Corollary \ref{cor:ample}. This leads to a contradiction to Theorem~\ref{thm:very ample}. Hence, $\mathcal L$ is nef.

Suppose $\mathcal L$ is nef. Then, by Theorem \ref{thm:very ample}, we have $a_{j,\alpha} \geq 0$. 
Now note that each $\mathcal{L}_{j,\alpha}$ is globally generated, as it is the pullback of the globally generated line bundle 
$\mathcal{L}_{G/P^{w_j}}(\omega_\alpha)$ (see Eq.\eqref{eq: Line bundles on gBS}). 
Therefore, $\mathcal{L}$ is globally generated, being a tensor product of globally generated line bundles. 
\end{proof}

\section{Fanoness of the Schubert variety \texorpdfstring{$X(w)\subseteq G/P^{w}$}{X(w)}}\label{sec: Fano Xw}

In this section, we give a criterion for the Schubert variety $X(w)$ in $G/P^{w}$ to be Gorenstein. Moreover, we observe that such Gorenstein Schubert varieties are Fano.

We first recall the divisor classes and the anticanonical sheaf of the Bott-Samelson variety $Z(\tilde{w})$. Let $\tilde{w}= s_{i_1}\cdots s_{i_r}$ be a reduced expression of $w\in W$. For $1\leq j\leq r$, we define \begin{align*}
f_j: Z(\tilde{w}(j)) &\to Z(\tilde{w}(j-1)), \quad 
[p_1,\dots,p_j] \to [p_1,\dots,p_{j-1}]
\end{align*}
where $\tilde{w}(j)=s_{i_1}\cdots s_{i_j}$. This morphism admits a section $\sigma_j $ defined by 
\begin{align*}
\sigma_j: Z(\tilde{w}(j-1))&\to Z(\tilde{w}(j)), \quad
[p_1,\dots,p_{j-1}] \to [p_1,\dots,p_{j-1},1].
\end{align*}
Consider $Z_j:= Z(s_{i_1}\cdots \widehat{s}_{i_j}\cdots s_{i_r})$, where $\widehat{a}$ means that $a$ is omitted. Then $$Z_j= f_{r}^{-1}\cdots f_{j+1}^{-1}(Im(\sigma_{j})),$$ and so each $Z_j$ is a divisor in $Z(\tilde{w})$.
Recall the morphism $\tilde{\pi}$ from Eq.\eqref{eq:tilde pi} for the case $m=1$, we have 
\begin{equation}\label{eq: pi map}
   \tilde \pi: Z(\tilde{w})\to G/P^w, \quad [p_1,\dots,p_r] \mapsto p_1\cdots p_rP^w.
\end{equation}
\begin{remark}\label{rmk:image1}
Note that the image of $\tilde \pi$ of $Z_j$ is the Schubert subvariety $X(y)$ where $y$ is the longest element that can be written as a subword of $w$ without the term $s_{i_j}$.      
\end{remark}

Denote by $\xi_{j}$ the class of $Z_j$ in the Chow ring $A^*(Z(\tilde{w}))$. We recall the following result of Lauritzen and Thomsen from \cite{LT04}.

\begin{proposition}[\cite{LT04},~Proposition~$3.5$]
    The divisors $(\xi_j)_{j\in [1,r]}$ form a basis of the class group of $Z(\tilde{w})$ .
\end{proposition}

\subsection{Anticanonical divisor for \texorpdfstring{$Z(\tilde{w})$}{}}
We define a sequence of roots $(\beta_i)_{1\leq i\leq r}$ associated to the expression  $\tilde w=s_{i_1}\cdots s_{i_r}$ by
\[
\beta_1:=\alpha_{i_1}, \qquad 
\beta_j:=s_{i_1}\cdots s_{i_{j-1}}(\alpha_{i_j}) \ \text{for } 2\le j\le r.
\]
Note that these roots are positive if the expression $\tilde{w}$ is reduced; otherwise, they need not be positive.

Denote by $T_j$ the pullback on $Z(\tilde{w})$ of the relative tangent sheaf of the fibration $ f_j$. Then, by \cite[Proposition~2.11]{Per07}, we have the following equality: $$T_j= \sum_{k=1}^{j}\langle \beta_j ,\beta_{k}\rangle \xi_{k}.$$ As $Z(\tilde{w})$ is a sequence of $\mathbb{P}^1$-fibrations with $T_j$ as a relative tangent bundle, we get $$-K_{Z(\tilde{w})}= \sum_{i=1}^{r}T_i.$$ 
Together with the preceding formula of $T_j$, we obtain the equality:

\begin{equation}\label{eq:ACS for BSv}
[-K_{Z(\tilde{w})}]= \sum_{k=1}^{r}\Big(\sum_{j=k}^{r}\langle \beta_{j},\beta_k\rangle \Big)\xi_{k}.   
\end{equation}

Recall the definition of the line bundles $\mathcal L_j, 1\leq j \leq m$ on $Z(\tilde{w})$ from Eq.\eqref{eq: lines on BSvs}. 
We now have two bases of $\operatorname{Pic}(Z(\tilde w))$. These two bases are related by the following change of basis formula due to Demazure \cite[Section~4.2, Proposition~1]{Dem74}. 
In particular, we can write the divisor classes $[\mathcal{L}_j]$ in terms of the basis $(\xi_k)_{k\in[1,r]}$. For a character $\lambda$ of  $T$ and $1\le k\le r$, set 
\[\mathcal O_k(\lambda):=\hat{f}^*_k{\hat\pi_k}^*\mathcal L_{G/B}(\lambda)\]

\begin{proposition}\label{Demazure}
     There is an isomorphism of line bundles on $Z(\tilde{w})$ 
    \[
\mathcal O_k(\lambda) \cong
\mathcal{O}_{Z(\tilde{w})}
\left(
\sum_{l=1}^{k}
\left\langle
\lambda,\,
s_{i_k}\cdots s_{i_{l+1}}(\alpha_{i_l})
\right\rangle
\xi_l
\right).
\]
\end{proposition}
The anticanonical line bundle $K^{-1}_{Z(\tilde{w})}$ on $Z(\tilde{w})$ is isomorphic to
\[
K^{-1}_{Z(\tilde{w})}
\cong
\mathcal{O}_{Z(\tilde{w})}
\left(\sum_{j=1}^{r}\xi_j\right)
\otimes
\mathcal{O}_{r}(\rho)
\]
(see \cite[Proof of Proposition~10]{MR85} and \cite[p.~67, Proposition~2.2.2]{BK07}).
Therefore, 
\begin{corollary}\label{cor:antican} We have
\[
K^{-1}_{Z(\widetilde{w})}
\cong
\mathcal{O}_{Z(\widetilde{w})}
\left(
\sum_{l=1}^{r}
\bigl(1+\langle\rho,s_{i_r}\cdots s_{i_{l+1}}(\alpha_{i_l})\rangle\bigr)\xi_l
\right).
\]
In particular, 
\[
K^{-1}_{Z(\widetilde{w})}
\cong  \mathcal{O}_{Z(\widetilde{w})}
\left(
\sum_{l=1}^{r}
\bigl(1+\operatorname{ht}((s_{i_r}\cdots s_{i_{l+1}}(\alpha_{i_l}))^{\vee}\bigr)\xi_l
\right).
\]
    
\end{corollary}

 \subsection{Divisors and Peaks for \texorpdfstring{$X(w)$}{X(w)}}\label{subsec: peaks for Xw}
 In this subsection, we introduce the notion of peaks and establish a bijective correspondence between the set of peaks and the divisors of $X(w)$.

Given a reduced expression $\tilde w = s_{i_1} \cdots s_{i_r}$ of $w \in W$, we define
\[
R^+(w) := \{ \alpha \in R^+ \mid w(\alpha) < 0 \} 
= \{ s_{i_r} s_{i_{r-1}} \cdots s_{i_2}(\alpha_{i_1}), \dots, s_{i_r}(\alpha_{i_{r-1}}), \alpha_{i_r} \}.
\]
The set $R^+(w)$ is well-defined; that is, it does not depend on the choice of reduced expression for $w$. Hence, the following definitions are also well-defined.

\begin{definition} \label{def:Peaks Xw}\ 

\begin{enumerate}
    \item For $w \in W$, we define the \emph{cover inversion set} of roots for $w$ as
    \[
        R^+_{w,B} := \{ \eta \in R^+(w) \mid \ell(ws_\eta) = \ell(w) - 1 \}.
    \]
    \item We define a subset of $R^+_{w,B}$, denoted by $\operatorname{Peaks}X(w)$, as
    \[
        \operatorname{Peaks}X(w) := \{ \eta \in R^+(w) \mid \ell(ws_\eta) = \ell(w) - 1, \; ws_\eta \in W^{P^w} \}.
    \]
\end{enumerate}
\end{definition}
\begin{example}
   Let $G= SL_4(\mathbb C)$ and $w=s_2s_1s_3$, then $S\setminus I^w =\{\alpha_1, \alpha_3\}$. Then, we have
   \[
   \operatorname{Peaks}X(w)= \{s_3s_1(\alpha_2), s_3(\alpha_1),\alpha_3\}=\{\alpha_1+ \alpha_2+ \alpha_3,\alpha_1,\alpha_3\}.\]
   \end{example}

    \begin{example}
    In $G=\operatorname{Spin}_8(\mathbb{C})$, consider $w=s_2s_3s_1s_2s_4s_1s_3\in W$. Then $S\setminus I^w=\{\alpha_1,\,\alpha_3,\,\alpha_4\}$ and \[
    \operatorname{Peaks}X(w)=\{s_3s_1s_4s_2s_1s_3(\alpha_2),\,s_3s_1(\alpha_4),\,s_3(\alpha_1),\,\alpha_3\}=\{\alpha_2+\alpha_4,\,\alpha_4,\,\alpha_1,\,\alpha_3\}.\]
\end{example}

For a given $\eta\in \operatorname{Peaks}X(w)$, there exists $k \in \{1,\cdots ,r\}$ such that
$$
\eta = s_{i_r} \cdots s_{i_{k+1}}(\alpha_{i_k}) \in R^+(w).
$$
Then
\[
ws_{\eta} = s_{i_1} \cdots s_{i_{k-1}} \, \hat{s}_{i_k} \, s_{i_{k+1}} \cdots s_{i_r},
\] 
i.e., the simple reflection in the position $i_k^\text{th}$ is omitted from the expression $\tilde w$. We denote such a reduced expression by $\widetilde{ws_{\eta}}$. Recall the morphism $\tilde \pi: Z(\tilde w) \to X(w)\subseteq G/P^w$ from Eq.\eqref{eq: pi map}. 
For each $\eta \in \operatorname{Peaks}X(w)$, $\xi_k$ is the divisor class of
\[
Z_k = Z(s_{i_1} \cdots \widehat{s}_{i_k} \cdots s_{i_r}) 
    = Z(\widetilde {ws}_{{\eta}}), \quad \text{where}\quad \eta= s_{i_r} \cdots s_{i_{k+1}}(\alpha_{i_k}).
    \]

Define 
\[
D_{\eta}:=\tilde \pi_{*}(\xi_k).
\]
We have the following.
\begin{proposition}\label{pro:bases for cl(X(w))}
    The set $\{D_{\eta}\mid \eta\in \operatorname{Peaks}X(w)\}$ forms a basis for the divisor class group $Cl(X(w))$ of $X(w)$.
\end{proposition}

\begin{proof} For each $\eta \in \operatorname{Peaks}X(w)$, the image $\tilde \pi\bigl(Z(\widetilde{ws}_{{\eta}})\bigl)$ of $\tilde{\pi}$ is a Schubert divisor $X(ws_{\eta})$ of $X(w)\subseteq G/P^w$. 
Moreover, by Remark~\ref{rmk:image1}, every Schubert divisor of $X(w)$ arises in this way.  It is well-known that the classes of Schubert divisors in $X(w)$ form a basis of the divisor class group of $X(w)$; see \cite[Proposition~2.2.8(i)]{Bri05}. Hence, the result follows.
\end{proof}

Now we consider the generalized decomposition $\widehat{w}$ of $w$ with $m=1$ to apply the results of Section \ref{map:pi}. 
With this we have $\widehat{w} = (w)$ i.e., $w_1 = w$, we have $\widehat{X}(\widehat{w}) = X(w)$. 
Then, for $\alpha \in S \setminus I^w$, we define the line bundle
\[
\mathcal{L}_{\omega_\alpha} := \mathcal{L}_{1,\alpha},
\] 
where $\mathcal{L}_{1,\alpha}$ is as in Eq.\eqref{eq: Line bundles on gBS}. 
Let $\tilde w = s_{i_1} \cdots s_{i_r}$ be a reduced expression of $w \in W$. 
For $1 \leq j \leq r$ with $\alpha_{i_j} \in \operatorname{Supp}(w)$, define
\begin{equation}\label{eq:m(j) for Xw}
m(j) := \max \{ q \mid 1\leq q \leq r~,~ s_{i_j} = s_{i_q} \}. 
\end{equation}

\begin{example}\label{ex:index}
In $SL_4(\mathbb{C})$, consider the maximal element $w\in W$, and fix a reduced expression 
\[\tilde w=s_1s_2s_3s_1s_2s_1 = s_{i_1}s_{i_2}s_{i_3}s_{i_4}s_{i_5}s_{i_6}.\] Also, note that
\[S\setminus I^w = \text{Supp}(w)= \{\alpha_1= \alpha_{i_1}=\alpha_{i_4}=\alpha_{i_6}, \alpha_2=\alpha_{i_2}=\alpha_{i_5},\alpha_3=\alpha_{i_3}\}.\] Then, by Eq.\eqref{eq:m(j) for Xw}, we get \[ m(1)=m(4)=m(6)=6,\quad m(2)=m(5)=5,\quad \text{and}\quad m(3)=3.\]
\end{example}

\begin{remark}\label{lem: pull-back from Xw}
    By Proposition~\ref{pull back} in the case $m=1$, for each $1\leq j \leq r$ we have
$
\tilde \pi^*\mathcal{L}_{\omega_\alpha} = \mathcal{L}_{m(j)}.
$

\end{remark}
Keeping the same notation as above, we now compute the divisor classes of the line bundles $\mathcal{L}_{\omega_\alpha}$ in terms of the peaks $\eta \in \operatorname{Peaks}X(w)$ and the anticanonical divisor of $X(w)$.

\begin{proposition}
    \label{lem:divisor classes} For $\alpha\in S\setminus I^w$, the divisor class of the line bundle $\mathcal{L}_{\omega_{\alpha}}$ is given by \[
[\mathcal{L}_{\omega_{\alpha}}]=\sum_{\eta\in \operatorname{Peaks}X(w)}\langle\omega_{\alpha},{\eta}\rangle D_\eta.
    \]

\end{proposition}

\begin{proof}
By Proposition \ref{Demazure}, we can rewrite the divisor classes of the line bundles $\mathcal{L}_j \in \operatorname{Pic} Z(\tilde w)$ as
\[
[\mathcal{L}_j] = \sum_{k=1}^j \big\langle \omega_{i_j}, s_{i_j} \cdots s_{i_{k+1}}(\alpha_k) \big\rangle \, \xi_k.
\]

Let $\alpha = \alpha_{i_j} \in S \setminus I^w$. Since $\tilde{\pi}_{*}\mathcal O_{Z(\tilde w)}=\mathcal O_{X(w)}$, by the projection formula, we get
\begin{align*}
[\mathcal{L}_{\omega_{i_j}}] 
&= \tilde \pi_*[\tilde \pi^* \mathcal{L}_{\omega_{i_j}}] \\
&= \tilde \pi_*[\mathcal{L}_{m(j)}] \qquad \text{(by Remark~\ref{lem: pull-back from Xw})} \\
&= \sum_{k=1}^{m(j)} \big\langle \omega_{i_{m(j)}}, s_{i_{m(j)}} \cdots s_{i_{k+1}}(\alpha_k) \big\rangle \, \tilde \pi_*(\xi_k).
\end{align*}
Note that if $m(j)<r$, then $\alpha_{i_\ell}\neq \alpha_{i_j}$ for all $m(j)<\ell\leq r$ (see Eq.\eqref{eq:m(j) for Xw}). Thus, for each $m(j)<\ell\leq r$, we have
\[
\left\langle
\omega_{i_{m(j)}},
s_{i_r}\cdots s_{i_{\ell+1}}(\alpha_{i_\ell})
\right\rangle
=0.
\] Hence, 
\[
[\mathcal{L}_{\omega_{i_j}}]
=
\sum_{k=1}^{r}
\bigl\langle
\omega_{i_{m(j)}},
s_{i_r}\cdots s_{i_{k+1}}(\alpha_{i_k})
\bigr\rangle\,
\tilde{\pi}_*(\xi_k).
\]
Combining this with the definition of $\operatorname{Peaks}X(w)$ and Proposition~\ref{pro:bases for cl(X(w))}, we obtain
\[
[\mathcal{L}_{\omega_{\alpha}}] = \sum_{\eta \in \operatorname{Peaks}X(w)} \langle \omega_{\alpha}, {\eta} \rangle \, D_\eta.
\]
\end{proof}

\subsubsection{Anticanonical divisor for \texorpdfstring{$X(w)$}{X(w)}}

Since Schubert varieties are normal and Cohen-Macaulay with rational singularities, the Bott-Samelson desingularization
\[
\tilde \pi: Z(\tilde w) \to X(w)
\] 
allows us to compute the anticanonical class of $X(w)$ via
\[
[-K_{X(w)}] =\tilde \pi_*([-K_{Z(\tilde w)}]),
\] 
see \cite[Proposition~2.2.5]{Bri05}. By Corollary \ref{cor:antican}, we have \begin{align*}
    [-K_{X(w)}]&=\tilde \pi_*\big([-K_{Z(\tilde w)}]\big)
    =\sum_{k=1}^{r}\big(\operatorname{ht}(s_{i_r}\cdots s_{i_{k+1}}(\alpha_{i_k})^{\vee})+1\big)\tilde \pi_*(\xi_k).
\end{align*}
   
Now, by Proposition \ref{pro:bases for cl(X(w))}, we obtain the following. \begin{equation}\label{eq:ACS for X(w)}
    [-K_{X(w)}]= \sum_{\eta\in \operatorname{Peaks}X(w)}(\operatorname{ht}(\eta^{\vee})+1) D_{\eta}.
\end{equation}

\subsection{Fano criterion for \texorpdfstring{$X(w)\subseteq G/P^w$}{X(w)}}\label{subsub: Fano for Xw}
\subsubsection{}
In this subsection, we assume that $G$
is of simply laced type. We review the notions from \cite{CKM25} and reprove the results using our notation.  
We start by recalling the following:
\begin{definition} Let $w\in W$.
  An element $\eta \in R^+(w)$ is called \emph{decomposable} if $\eta = \eta_1 + \eta_2$ for some $\eta_1, \eta_2 \in R^+(w)$. It is called \emph{indecomposable} if it is not decomposable in $R^+(w)$.
\end{definition}
We now recall the following characterization of indecomposable elements.
\begin{proposition}[\cite{CKM25},~Proposition~3.4]\label{pro: not decomposable} Assume that $G$ is of simply laced type. 
    Fix $w\in W$ and suppose $\eta\in R^+(w)$. Then, $\eta$ is indecomposable if and only if $ \eta \in R^+_{w,B}$ (i.e., $\ell(ws_{\eta})=\ell(w)-1$).
    \end{proposition}
    
\begin{definition}
    Let $G$ be of simply laced type and let $w\in W$. A subset $\mathcal{M}\subseteq R^+_{w,B}$ is called \emph{$P^w$-adaptation} of $S\setminus I^w$ if there exists a bijection $$\mu_w: S\setminus I^w\to \mathcal{M}$$ such that
    $$\mu_w(\alpha)\in \alpha+ \bigoplus_{\beta\in I^w}\mathbb Z_{\geq0}\beta \quad\quad \text{for all} ~\alpha\in S\setminus I^w.$$
\end{definition}
\begin{lemma}\label{lem: M' set}
    Let $G$ be of simply laced type, and let $\mathcal{M}$ be a $P^w$-adaptation of $S\setminus I^w$. Suppose that there exist $\alpha\in S\setminus I^w$ and $\beta\in I^w$ such that $ws_{\mu_{w}(\alpha)}(\beta)<0$. Then $\mu'_{w}(\alpha):=s_{\mu_{w}(\alpha)}(\beta)$ satisfies the following conditions:
    \begin{enumerate}
        \item $\mu_{w}'(\alpha)\in R^+_{w,B}$;
        \item $\mu'_w(\alpha)\in \alpha+ \bigoplus_{\beta\in I^w}\mathbb Z_{\geq0}\beta$;
        \item $\operatorname{ht}(\mu'_w(\alpha)^\vee)> \operatorname{ht}(\mu_w(\alpha)^\vee)$.
    \end{enumerate}
   Furthermore, for $\gamma\in S\setminus I^w$ such that $\gamma \neq \alpha$, define $\mu'_w(\gamma) := \mu_w(\gamma)$. 
Then, we obtain another $P^w$-adapted subset $\mathcal{M}'$ of $S \setminus I^w$ defined by
\[
\mathcal{M}' = \{ \mu'_w(\gamma) \mid \gamma \in S \setminus I^w \}.
\]
\end{lemma}
\begin{proof}
    Consider
\[
\mu'_w(\alpha) = \beta - \langle \beta, \mu_w(\alpha) \rangle \mu_w(\alpha).
\]
Since $w s_{\mu_w(\alpha)}(\beta) < 0$ and $w(\beta) \in R^+$, it follows that 
$\langle \beta, \mu_w(\alpha) \rangle \neq 0$. 
Moreover, by the definition of $\mu_w$, we have $\mu_w(\alpha) \neq \beta$. 
Hence, 
\[
- \langle \beta, \mu_w(\alpha) \rangle > 0.
\]
As $G$ is of simply laced type, we obtain
\[
- \langle \beta, \mu_w(\alpha) \rangle = 1.
\]
This, together with $w s_{\mu_w(\alpha)}(\beta) < 0$, implies that
\begin{equation}\label{eq: existence}
\mu'_w(\alpha) = \beta + \mu_w(\alpha)
\end{equation}
belongs to $R^+(w)$. This proves $(2)$.    

    To prove $(1)$, suppose that $\mu'_w(\alpha) \notin R^+_{w,B}$. 
Then, by Proposition \ref{pro: not decomposable}, $\mu'_w(\alpha)$ is decomposable in $R^+(w)$. 
Hence, there exist $\mu_1, \mu_2 \in R^+(w)$ such that
\[
\mu'_w(\alpha) = \mu_1 + \mu_2.
\]
Now consider
\begin{align*}
s_\beta(\mu'_w(\alpha)) 
&= -\beta + s_\beta(\mu_w(\alpha)) 
= -\beta + s_{\mu_w(\alpha)}(\beta), \\
s_\beta(\mu_1) + s_\beta(\mu_2) 
&= -\beta + \mu'_w(\alpha) 
= -\beta + \mu_1 + \mu_2.
\end{align*}

Note that $\langle \mu'_w(\alpha), \beta \rangle = 1$. 
Hence, either $\langle \mu_1, \beta \rangle = 0$ or 
$\langle \mu_2, \beta \rangle = 0$. 
Without loss of generality, assume that $\langle \mu_1, \beta \rangle = 0$. 
Then $\langle \mu_2, \beta \rangle = 1$. Thus,
\[
s_\beta(\mu_1) = \mu_1, 
\qquad 
s_\beta(\mu_2) = \mu_2 - \beta.
\]

Since 
\[
w(\mu_2) = w s_\beta(\mu_2) + w(\beta) < 0
\]
and $w(\beta) > 0$ (as $\beta \in I^w$), it follows that 
$w s_\beta(\mu_2) < 0$. Thus, $s_\beta(\mu_2) \in R^+(w)$.

Using these observations and Eq.\eqref{eq: existence}, we obtain
\[
\mu_w(\alpha) = \mu_1 + s_\beta(\mu_2),
\]
which contradicts the assumption that $\mu_w(\alpha)$ is indecomposable in $R^+(w)$ 
(see Proposition~\ref{pro: not decomposable}). 

The proof of $(3)$ follows from Eq.\eqref{eq: existence} and the definition of $\mu_w$. 

 From $(1)$ and $(2)$, it follows that 
$\mathcal{M}'$ is a $P^w$-adapted subset of $S \setminus I^w$. This completes the proof.
\end{proof}

By Lemma~\ref{lem: M' set}, we inductively construct a subset 
$\mathcal{B}_{P^w} \subseteq \operatorname{Peaks} X(w)$ 
that is in bijection with $S \setminus I^w$. 

We begin with
\[
\mathcal{M}_0 := S \setminus I^w,
\]
which is clearly a $P^w$-adapted subset of $S \setminus I^w$. 
Suppose that, for some $d \in \mathbb{Z}_{\ge 0}$, we have already constructed a $P^w$-adapted subset
\[
\mathcal{M}_d = \{ \mu_w(\gamma) \mid \gamma \in S \setminus I^w \}.
\]
If $w s_{\mu_w(\gamma)} \in W^{P^w}$ for all $\gamma \in S \setminus I^w$, then we define
\[
\mathcal{B}_{P^w} := \mathcal{M}_d.
\]
Otherwise, choose $\alpha \in S \setminus I^w$ and $\beta \in I^w$ such that 
$w s_{\mu_w(\alpha)}(\beta) < 0$. 
Using Lemma~\ref{lem: M' set}, define
\[
\mathcal{M}_{d+1} := \mathcal{M}_d'.
\]
Now, we repeat the construction for $\mathcal B_{P^w}$. 
In this construction, at each step there is an element whose height increases strictly (see Lemma \ref{lem: M' set}(3)). Since 
the height function is bounded, the process terminates after finitely many steps.

\begin{lemma}\label{lem:B_w,P^w}
Suppose that $G$ is of simply laced type and $w \in W^{P^w}$. 
The set $\mathcal{B}_{P^w}$ constructed above satisfies the following properties:
\begin{enumerate}
    \item $\mathcal{B}_{P^w} \subseteq \operatorname{Peaks} X(w)$.
    \item $\mathcal{B}_{P^w} = \{ \mu_w(\gamma) \mid \gamma \in S \setminus I^w \}$, where
    \[
    \mu_w(\gamma) \in \gamma + \bigoplus_{\beta \in I^w} \mathbb{Z}_{\ge 0} \beta,
    \]such that $\left\langle \omega_{\alpha},\mu_w(\gamma)\right\rangle=\delta_{\alpha\gamma}$, for $\alpha\in S\setminus I^{w}$.
\end{enumerate}
\end{lemma}

\begin{proof}
The statement follows directly from the construction of the set $\mathcal{B}_{P^w}$ together with Lemma~\ref{lem: M' set}.
\end{proof}

\begin{remark} \label{B5}\ 

\begin{enumerate}
    \item 
    The assumption that $G$ is simply laced plays a crucial role in the construction of the set $\mathcal{B}_{P^w}$. In the non simply laced case, the above method of constructing the set $\mathcal{B}_{P^w}$ may not work.    
    For example, suppose that $G$ is of type $B_5$ and $w=s_4s_5s_4\in W^{P^w}$. Then 
    \[I^w=\{\alpha_1,\alpha_2,\alpha_3,\alpha_5\}, \quad S\setminus I^w=\{\alpha_4\} \quad \text{and} \quad \operatorname{Peaks}X(w)=\{\alpha_4+2\alpha_5\}.\]
    Note that $ws_{\alpha_{4}}(\alpha_5)<0$, where $\alpha_5\in I^w$ and $\alpha_4\in S\setminus I^w$. In this case, Lemma \ref{lem: M' set}$(1)$ does not hold, since $ws_{s_4(\alpha_5)}=1$, which implies that ${s_4(\alpha_5)}\notin R^+_{w,B}$ (as we would need $\ell(ws_{s_4(\alpha_5)})=\ell(w)-1$). 
    \item By (1), we observe that $S\setminus I^w$ does not always embed in $\operatorname{Peaks}X(w)$.   
    \item Since our setting involves only the Schubert variety $X(w)\subseteq G/P^w$, the function $\mu_w$ is much simpler than the one considered in \cite[Section 5.2]{CKM25}.
    \end{enumerate}
\end{remark}

\subsubsection{}
Suppose that the Schubert variety $X(w)$ is Gorenstein, then there exist some $c_\alpha\in \mathbb{Z}$ such that \begin{align*}
[-K_{X(w)}]&=\sum_{\alpha\in S\setminus I^w}c_\alpha [\mathcal{L}_{\omega_\alpha}]\\
\sum_{\eta\in \operatorname{Peaks}X(w)}(\operatorname{ht}(\eta^\vee)+1)D_{\eta}&=\sum_{\alpha\in S\setminus I^w}c_\alpha \Big(\sum_{\eta\in \operatorname{Peaks}X(w)}\left\langle\omega_\alpha,\eta\right\rangle D_\eta\Big). \end{align*} Now, by using Proposition \ref{pro:bases for cl(X(w))}, we obtain the following linear system \begin{equation}\label{eq:linear system}
\operatorname{ht}(\eta^{\vee})+1=\sum_{\alpha\in S \setminus I^w}c_\alpha \left\langle\omega_\alpha,\eta\right\rangle ~\quad\forall ~\eta\in \operatorname{Peaks}X(w). \end{equation}

In order to solve this linear system, we define a total ordering on the set $S\setminus I^w$, depending on the chosen reduced expression of $w\in W$. Suppose $\tilde w=s_{i_1}\cdots s_{i_m}$ is a reduced expression of $w$, then for all $\alpha_{i_j}, \alpha_{i_k}\in S\setminus I^w$, we define \begin{equation}\label{eq:ordering for lines Xw}
\alpha_{i_j}\preceq \alpha_{i_k} ~~~\text{if}~~~m(j)\leq m(k),
\end{equation}
where $m(j)$'s are defined in Eq.\eqref{eq:m(j) for Xw}. 
\begin{example}\label{eg:oredering}
    In $G=SL_4(\mathbb{C})$, let $w=s_1s_2s_3s_1s_2s_1$ be the longest element of $W$, then $S\setminus I^w= \{\alpha_1,\alpha_2,\alpha_3\}$.
    Then, by the definition of total ordering
    \[
    S\setminus I^w=\{\alpha_3\prec\alpha_2\prec\alpha_1\}.
    \]
\end{example}

\begin{lemma}\label{lem:identity}
With the ordering $\preceq$ defined in Eq.\eqref{eq:ordering for lines Xw}, the matrix \[
M_{P^w}:=\big(\langle\omega_\alpha,\mu_w(\beta)\rangle\big)^T_{\alpha,\beta\in S\setminus I^w}
\]
is the identity matrix.
\end{lemma}
\begin{proof}
    From Lemma \ref{lem:B_w,P^w}$(2)$, for any $\alpha\in S\setminus I^w$, we have \[\mu_w(\alpha)\in\alpha+\bigoplus_{\beta\in I^w}\mathbb{Z}_{\geq0}\beta.\] From this, we have $\langle \omega_\alpha ,\mu_w(\beta)\rangle=\delta_{\alpha\beta}$. Now by the ordering $\preceq$, we get the result.
\end{proof}
\begin{remark}\label{rem:nonsimplylaced1} \ 

\begin{enumerate}
    \item  In general,
Lemma \ref{lem:identity} does not hold for non-simply laced cases. For example, in type $B_5$, consider $w=s_1s_5s_4$. Then, we have 
\[
 S\setminus I^w=\{\alpha_1,\alpha_4\},~\operatorname{Peaks}X(w)=\{\alpha_1,\alpha_4+\alpha_5\}\quad\text{and}\quad \mathcal{B}_{P^w}=\{\mu_w(\alpha_1)=\alpha_1,~ \mu_w(\alpha_4)={\alpha_4+\alpha_5}\}.
 \]
Thus, \[M_{P^w}=\begin{pmatrix}
    1&0\\0&2
\end{pmatrix},
 \]which is not an identity matrix (not even invertible over $\mathbb{Z}$).
 \item For non-simply laced groups $G$, if there is a $\mu_w$ function such that $M_{P^w}$ is unipotent upper triangular, then our method works. 
 \end{enumerate}
  \end{remark}
\begin{theorem}\label{thm: Gorenstein Xw}
The Schubert variety $X(w)\subseteq G/P^w$ is Gorenstein if and only if for any $\eta\in \operatorname{Peaks}X(w)\setminus \mathcal{B}_{P^w}$, we have
\begin{equation}\label{eq:height relations for Xw}
\operatorname{ht}(\eta^\vee)+1=\sum_{\alpha\in S \setminus I^w}d_\alpha\Big(\operatorname{ht}\big(\mu_w(\alpha)^{\vee}\big)+1\Big) ,
\end{equation}
with $d_\alpha:=\langle\omega_\alpha,\eta\rangle\in \mathbb{Z}_{\geq0}$.
\end{theorem}
\begin{proof}
   By Eq.\eqref{eq:linear system} together with Lemma \ref{lem:identity}, we obtain that for any 
$\alpha \in S \setminus I^w$ such that $\mu_w(\alpha) \in \operatorname{Peaks}X(w)$, we have
\[
c_\alpha = \operatorname{ht}\big(\mu_w(\alpha)^\vee\big) + 1 .
\] 
The set $\{c_\alpha\mid \alpha\in S\setminus I^w\}$ is the solution of the linear system of equations in Eq.\eqref{eq:linear system} if and only if for any $\eta\in \operatorname{Peaks}X(w)\setminus \mathcal{B}_{P^w}$, we have \[
\operatorname{ht}(\eta^\vee)+1=\sum_{\alpha\in S \setminus I^w}\Big(\operatorname{ht}\big(\mu_w(\alpha)^\vee\big)+1\Big) \langle\omega_\alpha,\eta\rangle.
\]
Hence, the result follows.\end{proof}

The following corollary is an immediate consequence of the theorem:
\begin{corollary}\label{cor: Fano for Xw}
    If the Schubert variety $X(w)\subseteq G/P^w$ is Gorenstein, then it is Fano.
\end{corollary}
\begin{proof}
    If $X(w)$ is Gorenstein, then the set $$\{c_\alpha=\operatorname{ht}(\mu_w(\alpha)^\vee)+1\mid \alpha\in S\setminus I^w\}$$ is the solution of the linear system of equations in Eq.\eqref{eq:linear system}. Since $\mu_w(\alpha)$'s are positive roots, it follows that $c_\alpha=\operatorname{ht}(\mu_w(\alpha)^\vee)+1>0$ and this proves the result.\end{proof}

\begin{example}
    Consider $G$ and $w$ the same as in Example \ref{eg:oredering}, then
    \[
    \operatorname{Peaks}X(w)=\{\alpha_3, \alpha_2, \alpha_1\}.
    \] Here $I^w=\emptyset$. Thus, by the definition of $\mu_w$ and $\mathcal{B}_{P^w}$, we have \begin{align*}\mu_w(\alpha_1)=\alpha_1,~~ \mu_w(\alpha_2)=\alpha_2,~~\text{and}~~\mu_w(\alpha_3)&=\alpha_3,
    \end{align*} and \[
    \mathcal{B}_{P^w}=\{ \mu_w(\alpha_1),\mu_w(\alpha_2),\mu_w(\alpha_3)\},
    \] 

Since $\operatorname{Peaks}X(w)=\mathcal{B}_{P^w}$, by Theorem \ref{thm: Gorenstein Xw}, the Schubert variety $X(w_0)$ is Gorenstein and by Corollary \ref{cor: Fano for Xw}, it is Fano.
      \end{example}

\begin{example}
   Let $G= SL_4$ and $w=s_2s_1s_3$, then $S\setminus I^w =\{\alpha_1, \alpha_3\}$. Also, we have
   \[
   \operatorname{Peaks}X(w)= \{\alpha_1+ \alpha_2+ \alpha_3, \alpha_1, \alpha_3\}~~\text{and}~~B_{w,P^w}=\{\mu_w(\alpha_1)=\alpha_1,\mu_w(\alpha_3)=\alpha_3\},
   \]
   Note that $\operatorname{Peaks}X(w)\setminus \mathcal{B}_{P^w}=\{\eta:=\sum_{i=1}^{3}\alpha_i\}$ and $$\operatorname{ht}(\eta^\vee)+1=4,\quad \operatorname{ht}(\mu_w(\alpha_j)^\vee)+1=2\quad \text{for j=1,3}.$$
Thus, by Theorem \ref{thm: Gorenstein Xw} and Corollary \ref{cor: Fano for Xw}, the Schubert variety $X(w)$ is Fano. 
\end{example}

\begin{example}
    In $G=\operatorname{Spin_8}(\mathbb{C})$, consider $w=s_2s_3s_1s_2s_4s_1s_3\in W$. Then $S\setminus I^w=\{\alpha_1,\alpha_3,\alpha_4\}$,  \[
    \operatorname{Peaks}X(w)=\{\alpha_2+\alpha_4,\alpha_4,\alpha_1,\alpha_3\}\]and\[B_{w,P^w}=\{\mu_w(\alpha_1)=\alpha_1,\mu_w(\alpha_3)=\alpha_3,\mu_w(\alpha_4)=\alpha_4\}.
    \] 
    Note that $\operatorname{Peaks}X(w)\setminus \mathcal{B}_{P^w}=\{\eta:=\alpha_2+\alpha_4\}$ and \[
    \operatorname{ht}(\eta^\vee)+1=3,\quad \operatorname{ht}(\mu_w(\alpha_j)^\vee)+1=2\quad \text{for $j=1,3,4$}.
    \]Since $\sum_{\alpha\in S\setminus I^w}d_\alpha(\operatorname{ht}(\mu_w(\alpha)^\vee)+1)=2$, by Theorem \ref{thm: Gorenstein Xw}, we conclude that the Schubert variety $X(w)$ is not Gorenstein.\end{example}

\section{Fano generalized Bott-Samelson varieties}\label{sec: Fano gBS}
In this section, we generalize the notion of peaks to generalized Bott-Samelson varieties. We compute their anticanonical divisor, characterize the Gorenstein cases, and give a description of those generalized Bott-Samelson varieties that are Fano or weak Fano.

Throughout, we fix an element $w \in W$ and choose an admissible generalized (not necessarily reduced) decomposition
\[
\widehat{w} = (w_1, \dots, w_m)
\]
of $w$. We also recall from Lemma~\ref{lem:reduced decomposition}, the associated decomposition in terms of simple reflections
\[\tilde{w} = (s_{1,1}, \dots, s_{r_m,m}),\]

where each factor $w_j$ is given with a reduced expression
\begin{equation}\label{eq:order exp}
    \tilde{w}_j = \prod_{k=1}^{r_j} s_{k,j}
\quad \text{for } 1 \le j \le m.
\end{equation}

\subsection{Peaks and Divisors on \texorpdfstring{$\widehat{X}(\widehat{w})$}{X(w-hat)}}
 We first define the notion of Peaks inductively for generalized Bott-Samelson varieties.   
Given $1\leq j\leq m$, we define  \[\operatorname{Peaks}\widehat{X}(w_1,\dots,w_j):=w_j^{-1}\big(\operatorname{Peaks}\widehat{X}({w_1,\dots,w_{j-1}})\big)\sqcup \operatorname{Peaks}X(w_j).\] 

\begin{remark}\ 

\begin{enumerate}
    \item  $\operatorname{Peaks}\widehat{X}(\widehat{w})$ is a finite sequence of roots in the root system $R$; repetitions are allowed.
    \item 
Note that if the generalized decomposition $\widehat{w}$ is reduced, then for each $1 \leq j \leq m$, the set
$\operatorname{Peaks}\widehat{X}(w_1,\dots,w_j)$
is a subset of the set of positive roots $R^+$. In particular, $\operatorname{Peaks}\widehat{X}(\widehat w)$ is a finite sequence of positive roots without repetitions.  

\end{enumerate}
\end{remark}

For any $\widehat{\eta}\in \operatorname{Peaks}\widehat{X}(\widehat{w})$, note that $Z(\widetilde{ws}_{{\widehat{\eta}}})$ is a divisor of $Z(\tilde w)$. By Eq.\eqref{eq:tilde pi},  $\tilde{\pi}(Z(\widetilde{ws}_{{\widehat{\eta}}}))$ is a divisor in $\widehat{X}(\widehat{w})$.
Therefore, we define \begin{align}\label{eq: divisors for gBS}
\widehat{D}_{\widehat{\eta}}:=\tilde{\pi}_{*}([Z(\widetilde{ws}_{{\widehat{\eta}}})]). 
\end{align}

\begin{proposition}\label{pro: divisor classes for gBS}
    The set $\mathcal A:=\{\widehat{D}_{\widehat{\eta}}\mid \widehat{\eta}\in \operatorname{Peaks}\widehat{X}(\widehat{w})\}$ forms a basis of the divisor class group $\operatorname{Cl}(\widehat{X}(\widehat{w}))$.
    \end{proposition} 
      \begin{proof} Since the set 
\[
\{Z(\widetilde{ws}_{{\widehat{\eta}}}) \mid \widehat{\eta}\in \operatorname{Peaks}\widehat{X}(\widehat{w})\}
\]
is contained in a basis of the divisor class group of $Z(\tilde{w})$, the set $\mathcal A$ is linearly independent. Recall that each $\mathrm{Cl}(X(w_i))$ is a finitely generated free abelian group, see Proposition \ref{pro:bases for cl(X(w))}. 
Since the morphism 
\[
\widehat{f} : \widehat{X}(\widehat{w}) \to \widehat{X}(w_1)
\]
is a locally trivial fibration with fiber $\widehat{X}(w_2, \dots, w_m)$, by induction it follows that 
$\mathrm{Cl}(\widehat{X}(\widehat{w}))$ is a free abelian group of rank 
\[
\sum_{i=1}^m \operatorname{rank}(\mathrm{Cl}(X(w_i))). 
\]
Therefore, the proof follows from the definition of $\operatorname{Peaks}\widehat{X}(\widehat{w})$ together with Proposition \ref{pro:bases for cl(X(w))}.  
\end{proof}

For $1\leq j\leq m, \alpha\in S\setminus I^{w_j}$, recall the construction of the line bundle $\mathcal{L}_{j,\alpha}$ from Eq.\eqref{eq: Line bundles on gBS}.

\begin{proposition}\label{prop: divisor classes for line bundles on gBS}
   The divisor class of the line bundle $\mathcal{L}_{j,\alpha}$ is given by \[
    [\mathcal{L}_{j,\alpha}]=\sum_{\widehat{\eta}\in \operatorname{Peaks}\widehat{X}(w_1,\dots,w_j)}\langle\omega_{\alpha},{\widehat{\eta}}\rangle \widehat{D}_{(w_{j+1}\cdots w_{m})^{-1}(\widehat{\eta})},
    \] where $(w_{j+1}\cdots w_{m})^{-1}(\widehat{\eta})\in \operatorname{Peaks}\widehat{X}(\widehat{w})$.
\end{proposition}
\begin{proof}
    Let $\alpha=\alpha_{d,j}\in S\setminus I^{w_j}$ for some $1\leq d\leq r_j$. Then, using  Proposition \ref{pull back}, Eq.\eqref{Eq:rational singular} and projection formula, we have \begin{align*}
        [\mathcal{L}_{j,\alpha}]&=\tilde{\pi}_*\big([\tilde{\pi}^*\mathcal{L}_{j,\alpha}]\big)=\tilde{\pi}_*\big([\mathcal{L}_{m_j(d),j}]\big) \\&=\sum_{1\leq k\leq m_j(d)}\left\langle\omega_{m_j(d),j},s_{m_j(d),j}\cdots s_{k+1,j}(\alpha_{k,j})\right\rangle\tilde{\pi}_*(\xi_{k,j}) \quad\quad \\ &\quad \quad \quad \quad \quad  +\sum_{1\leq p< j}\sum_{1\leq k\leq r_p}\left\langle\omega_{m_j(d),j},s_{m_j(d),j}\cdots s_{k+1,p}(\alpha_{k,p})\right\rangle \tilde{\pi}_*(\xi_{k,p}),
    \end{align*}
    the last equality follows by Proposition \ref{Demazure}.
    Note that if $m_j(d)<r_j$, then for any $m_j(d)<\ell\leq r_j$, we have $\alpha_{m_j(d),j}\neq \alpha_{\ell,j}$ (see Eq.\eqref{eq:m_j(d)}). Thus, we get\begin{align*}
       [\mathcal{L}_{j,\alpha}] &=\sum_{1\leq k\leq m_j(d)}\left\langle\omega_{m_j(d),j},\,s_{r_j,j}\cdots s_{m_j(d)+1,j}\,( s_{m_j(d),j}\cdots s_{k+1,j}\,(\alpha_{k,j}))\right\rangle\tilde{\pi}_*(\xi_{k,j}) \quad\quad \\ &\quad \quad \quad \quad \quad  +\sum_{1\leq p< j}\sum_{1\leq k\leq r_p}\left\langle\omega_{m_j(d),j},\,s_{r_j,j}\cdots s_{m_j(d)+1,j}\,( s_{m_j(d),j}\cdots s_{k+1,p}(\alpha_{k,p}))\right\rangle \tilde{\pi}_*(\xi_{k,p}).
    \end{align*} Moreover, for each $m_j(d)<\ell\leq r_j$, we have
\[
\left\langle
\omega_{{m_j(d),j}},\,
s_{r_j,j}\cdots s_{\ell+1,j}(\alpha_{\ell,j})
\right\rangle
=0.
\]Hence, \begin{align*}
       [\mathcal{L}_{j,\alpha}] &=\sum_{1\leq k\leq r_j}\left\langle\omega_{m_j(d),j},\,s_{r_j,j}\cdots s_{k+1,j}\,(\alpha_{k,j})\right\rangle\tilde{\pi}_*(\xi_{k,j}) \quad\quad \\ &\quad \quad \quad \quad \quad  +\sum_{1\leq p< j}\sum_{1\leq k\leq r_p}\left\langle\omega_{m_j(d),j},\,s_{r_j,j}\cdots s_{k+1,p}\,(\alpha_{k,p})\right\rangle \tilde{\pi}_*(\xi_{k,p}).
    \end{align*}
 
    Recall from Section \ref{map:pi} that for any $1\leq j \leq m$ and for any $1\leq d\leq r_j$, $\xi_{d,j}$ is the divisor class of\begin{equation}\label{eq:Z_{k,q}}
    \begin{aligned}    
    Z_{d,j}&=Z(s_{1,1}\cdots s_{r_1,1}\cdots s_{1,j}\cdots \widehat{s}_{d,j}\cdots s_{r_j,j}\cdots s_{1,m}\cdots s_{r_m,m})\\&=Z(\widetilde{ws}_{{\widehat{\eta}}}),\quad \text{where}\quad \widehat{\eta}=(w_{j+1}\cdots w_{m})^{-1}s_{r_j,j}\cdots s_{d,j}.
   \end{aligned} \end{equation}
 Note that $\omega_\alpha=\omega_{m_j(d),j}$. 
    By the definition of $\operatorname{Peaks}\widehat{X}(w_1,\dots,w_j)$ together with Proposition \ref{pro: divisor classes for gBS}, we have
    \begin{align}\label{eq:Line divisors for gBS}
[\mathcal{L}_{j,\alpha}]&=\sum_{t=1}^{j}\sum_{\eta\in \operatorname{Peaks}X(w_t)}\left\langle\omega_{\alpha},(w_{t+1}\cdots w_j)^{-1}\eta\right\rangle\widehat{D}_{(w_{t+1}\cdots w_m)^{-1}\eta} \quad  \\
&=\sum_{\widehat{\eta}\in \operatorname{Peaks}\widehat{X}(w_1,\dots,w_j)}\langle\omega_{\alpha},{\widehat{\eta}}\rangle \widehat{D}_{(w_{j+1}\cdots w_{m})^{-1}(\widehat{\eta})}.\notag
\end{align}  This completes the proof.
    \end{proof}

\begin{proposition}\label{anti-can}
    The anticanonical divisor on $\widehat{X}(\widehat{w})$ is $$[-K_{\widehat{X}(\widehat{w})}]=\sum_{\widehat{\eta}\in \operatorname{Peaks}\widehat{X}(\widehat{w})}\big(\operatorname{ht}(\widehat{\eta}^\vee)+1\big)\widehat{D}_{\widehat{\eta}}.$$
\end{proposition}
\begin{proof}
Since the generalized Bott-Samelson varieties are normal and Cohen-Macaulay with rational singularities (see Lemma \ref{rmk:smoothness} and Proposition \ref{prop:rational-sing}), the Bott-Samelson resolution
\[
\tilde{\pi} \colon Z(\tilde{w}) \longrightarrow \widehat{X}(\widehat{w})
\]
allows us to compute the anticanonical class of $\widehat{X}(\widehat{w})$ via pushforward. More precisely,
\begin{equation}\label{eq: push forward to get ACD for gBS}
\tilde{\pi}_*\big([-K_{Z(\tilde{w})}]\big)
=
[-K_{\widehat{X}(\widehat{w})}],
\end{equation}
see for example \cite[Proposition~2.2.5]{Bri05}.
By Corollary \ref{cor:antican}, we have \begin{align}\label{eq:anticanonical}
[-K_{\widehat{X}(\widehat{w})}]
&= \tilde{\pi}_*\big([-K_{Z(\tilde{w})}]\big) \notag\\
&= \sum_{j=1}^{m}\sum_{d=1}^{r_j}
\big(\operatorname{ht}(w_m^{-1}\cdots w_{j+1}^{-1}s_{r_j,j}\cdots s_{d+1,j}(\alpha_{d,j})^\vee)+1\big)
\tilde{\pi}_*(\xi_{d,j}) \notag\\
&= \sum_{j=1}^m\sum_{\eta\in \operatorname{Peaks}X(w_j)}
\Big(\operatorname{ht}\big(((w_{j+1}\cdots w_{m})^{-1}\eta)^\vee\big)+1\Big)
\widehat{D}_{(w_{j+1}\cdots w_m)^{-1}\eta}.
\end{align}

    Now by the definition of $\operatorname{Peaks}\widehat{X}(\widehat{w})$, we have \begin{equation}
        [-K_{\widehat{X}(\widehat{w})}]=\sum_{\widehat{\eta}\in \operatorname{Peaks}\widehat{X}(\widehat{w})}\big(\operatorname{ht}(\widehat{\eta}^\vee)+1\big)\widehat{D}_{\widehat{\eta}}.
    \end{equation}This completes the proof.
\end{proof}

We recall the following theorem from \cite{Mat89}.
\begin{theorem}[\cite{Mat89},~Cor. 23.3, Thm. 23.4]\label{thm:Gorenstein}
Let $\psi: X\to Y $ be a flat
morphism of algebraic varieties. Then $X$ is Cohen-Macaulay (resp. Gorenstein) if and only if $Y$
and every fiber is Cohen-Macaulay (resp. Gorenstein).
\end{theorem}

\begin{corollary}\label{thm: Gorenstein gBS}
    The variety $\widehat{X}(\widehat{w})$ is Gorenstein if and only if each Schubert variety $X(w_i)$ is Gorenstein.
\end{corollary}
\begin{proof} We prove the result by induction on $m$.  
For $m=1$, the statement is clear.  
For $m=2$, the morphism
\[
f_1:\widehat{X}(w_1,w_2)\to X(w_1)
\]
is a locally trivial fibration with fiber $X(w_2)$, hence flat.  
By Theorem \ref{thm:Gorenstein}, the result holds for $m=2$.  Assume the result for $m-1$ and consider
\[
f_m:\widehat{X}(\widehat{w})\to \widehat{X}(w_1,\dots,w_{m-1}),
\]
which is a locally trivial fibration with fibers isomorphic to $X(w_m)$ and thus flat.  
Applying Theorem \ref{thm:Gorenstein} again, the result follows for $m$.
\end{proof}

\begin{remark}
    Corollary \ref{thm: Gorenstein gBS} together with Theorem \ref{thm: Gorenstein Xw} gives the characterization of Gorenstein generalized Bott-Samelson varieties.
\end{remark}

It is shown in \cite[Theorem~5.3]{BC18} that for any reduced expression 
$\tilde{w}$ of $w \in W$, the anticanonical divisor 
$[-K_{Z(\tilde{w})}]$ is big. As a consequence, we have the following:
\begin{proposition}\label{prop:big}
Suppose that $\widehat{X}(\widehat{w})$ is Gorenstein. 
Then the anticanonical divisor $[-K_{\widehat{X}(\widehat{w})}]$ 
is big.
\end{proposition}
\begin{proof}
  By Eq.\eqref{eq: push forward to get ACD for gBS}, we have \[\tilde{\pi}_*\big([-K_{Z(\tilde{w})}]\big)
=
[-K_{\widehat{X}(\widehat{w})}].
 \]
 Recall that the pushforward of a big line bundle by a birational morphism between projective varieties is big, see \cite[{Page 67}]{Kol98}. Thus, we conclude the result.
 \end{proof}

\subsection{Fano generalized Bott-Samelson varieties}\label{subsub: Fano}

Suppose that $\widehat{X}(\widehat{w})$ is Gorenstein. Then we have \begin{align}\label{eq: c_{i,alpha} for Fano}
[-K_{\widehat{X}(\widehat{w})}]
&= \sum_{j=1}^m\sum_{\alpha\in S\setminus I^{w_j}}c_{j,\alpha}[\mathcal{L}_{j,\alpha}] ,
\end{align}
for some $c_{j,\alpha}\in \mathbb{Z}$. Thus, by Eq.\eqref{eq:Line divisors for gBS} and Eq.\eqref{eq:anticanonical}, we obtain:
 \begin{align}\label{eq:Peaks}
\sum_{j=1}^m \sum_{\alpha \in S \setminus I^{w_j}} c_{j,\alpha}
&\Bigg(
  \sum_{i=1}^{j} \sum_{\eta \in \operatorname{Peaks}X(w_i)}
  \langle \omega_{\alpha}, (w_{i+1}\cdots w_j)^{-1}\eta \rangle\,
  \widehat{D}_{(w_{i+1}\cdots w_m)^{-1}\eta}
\Bigg)
\notag\\
&= \sum_{j=1}^m \sum_{\eta \in \operatorname{Peaks}X(w_j)}
\Big(
  \operatorname{ht}\big(((w_{j+1}\cdots w_m)^{-1}\eta)^\vee\big)+1
\Big)
\widehat{D}_{(w_{j+1}\cdots w_m)^{-1}\eta}.
\end{align}
Recall that the set
\[
\{\widehat{D}_{\widehat{\eta}} \mid \widehat{\eta} \in \operatorname{Peaks}\widehat{X}(\widehat{w})\}
\]
is linearly independent in the divisor class group of $\widehat{X}(\widehat{w})$ (Proposition~\ref{pro: divisor classes for gBS}).
Therefore, for each $1\leq j \leq m$ and each $\eta \in \operatorname{Peaks} X(w_j)$, equating coefficients in Eq.\eqref{eq:Peaks} yields the following linear system:
\begin{align}\label{eq:linear systems}
\sum_{\alpha \in S \setminus I^{w_j}} 
c_{j,\alpha}\,\langle \omega_{\alpha}, \eta \rangle
&+ \sum_{t=j+1}^{m} \sum_{\beta \in S \setminus I^{w_t}}
c_{t,\beta}\,
\langle \omega_{\beta}, (w_{j+1}\cdots w_t)^{-1}\eta \rangle \notag\\
&= \operatorname{ht}\big(((w_{j+1}\cdots w_m)^{-1}\eta)^\vee\big)+1.
\end{align}
By convention, the second sum on the left-hand side is omitted when $j=m$.
Our next goal is to solve this system of linear equations.  
\subsubsection{}\label{subsection: M1}
In order to solve the system of equations in Eq.\eqref{eq:linear systems}, we derive a sublinear system of full rank corresponding to the sets $\mathcal{B}_{P^{w_j}}$. 
To construct this sublinear system, we first introduce a total order $\preceq_{\widehat{w}}$ on the collection
\[
\bigsqcup_{1 \le j \le m} \bigl(S \setminus I^{w_j}\bigr),
\]
which depends on the chosen expression of $\tilde{w}_j$ of $w_j$ for each $1\leq j\leq m$ in Eq.\eqref{eq:order exp}, and is defined as follows.

{\bf Total order $\preceq_{\widehat{w}}$:} Recall the definition of $m_j(d)$ defined in Eq.\eqref{eq:m_j(d)}.

\noindent \emph{Case {1}:}
For $1 \le i< j \le m$, let us consider $\alpha_{d_i,i} \in S \setminus I^{w_i}$ and $\alpha_{d_j,j} \in S \setminus I^{w_j}$ for some $1 \le d_i \le r_i$ and $1 \le d_j \le r_j$. 
We define a total order $\preceq_{\widehat{w}}$ for these elements by
\[
\alpha_{d_i,i} \prec_{\widehat{w}} \alpha_{d_j,j}.
\]
\emph{Case 2:} For $1\leq j\leq m$, let $\alpha_{d_j,j}\in S\setminus I^{w_{j}}$ and $\alpha_{d'_j,j}\in S\setminus I^{w_{j}}$ for some $1\leq d_j,d_j'\leq r_j$. Then, we define 
\[
\alpha_{d_j,j}\preceq_{\widehat{w}}\alpha_{d'_j,j} \qquad \text{if} \qquad m_j(d_j)\leq m_j(d_j').
\] 
We illustrate this in the following example.
\begin{example}
   Let $G = SL_5(\mathbb{C})$, and consider 
\[
w = (w_1, w_2, w_3) = (s_3s_2 s_3 s_1,\, s_2 s_4,\, s_1 s_3) \quad \text{and}\]
 \[
\tilde{w} = (s_{1,1}, s_{2,1}, s_{3,1},s_{4,1}, s_{1,2}, s_{2,2}, s_{1,3}, s_{2,3}) = (s_3,s_2, s_3, s_1, s_2, s_4, s_1, s_3),
\] 
two admissible generalized decompositions of 
$
w =s_3 s_2 s_3 s_1 s_2 s_4 s_1 s_3.
$
Then we have
\begin{align*}
S \setminus I^{w_1} &= \{\alpha_{1,1} =\alpha_{2,1} = \alpha_3,\; \alpha_{3,1} = \alpha_1\},\\
S \setminus I^{w_2} &= \{\alpha_{1,2} = \alpha_2,\; \alpha_{2,2} = \alpha_4\},\\
S \setminus I^{w_3} &= \{\alpha_{1,3} = \alpha_1,\; \alpha_{2,3} = \alpha_3\}.
\end{align*}
By the definition of the total order $\preceq_{\widehat{w}}$, these elements satisfy
\[
\alpha_{1,1}=_{\widehat{w}}\alpha_{2,1} \prec_{\widehat{w}} \alpha_{3,1} \prec_{\widehat{w}} \alpha_{1,2} \prec_{\widehat{w}} \alpha_{2,2} \prec_{\widehat{w}} \alpha_{1,3} \prec_{\widehat{w}} \alpha_{2,3}.
\]
\end{example}

From now on, we assume that $G$ is of simply laced type. This assumption is needed only for the existence of the functions $\mu_{w_j}$ (see Lemma~\ref{lem:B_w,P^w}). Once these functions are available, the same method works for $G$ of non simply laced type as well.

Given $1\leq j\leq m$, let \[n_j:=\sum_{1\leq i\leq j}\#(S\setminus I^{w_i}).\]

Now, with respect to the total order $\preceq_{\widehat{w}}$, we construct certain matrices associated with the sets $S \setminus I^{w_j}$, defined as follows.  
For $j = 1$, we set
\[
\widehat{M}_{P^{w_1}} := \begin{pmatrix} M_{P^{w_1}} & \vdots & M_1 \end{pmatrix},
\]
where $M_{P^{w_1}}$ is the square matrix of order $n_1$ defined in Lemma~\ref{lem:identity}, and
\[
M_1 := \begin{pmatrix}
\langle \omega_{\alpha}, (w_2 \cdots w_t)^{-1} \mu_{w_1}(\beta) \rangle
\end{pmatrix}^T_{\alpha \in \bigsqcup_{2\leq t\leq m} S \setminus I^{w_t},~ \beta\in S\setminus I^{w_1}}
\]
is a $n_1 \times (n_m - n_1)$ matrix.

For $1 < j \le m-1$, we define
\[
\widehat{M}_{P^{w_j}} := 
\begin{pmatrix} [0] & \vdots & M_{P^{w_j}} & \vdots & M_j \end{pmatrix},
\]
where $M_{P^{w_j}}$ is the square matrix of order $(n_j - n_{j-1})$ defined in Lemma~\ref{lem:identity},  
$[0]$ is a zero matrix of size $(n_j - n_{j-1}) \times n_{j-1}$, and
\[
M_j := \begin{pmatrix}
\langle \omega_{\alpha}, (w_{j+1} \cdots w_t)^{-1} \mu_{w_j}(\beta) \rangle
\end{pmatrix}^T_{\alpha \in \bigsqcup_{j+1\leq t\leq m} S \setminus I^{w_t},~ \beta\in S\setminus I^{w_j}}
\]
is a $(n_j - n_{j-1}) \times (n_m - n_j)$ matrix.

\begin{remark}\label{rmk:restriction to sets}
   Recall the total order $\preceq$ on $S \setminus I^{w}$ defined in Eq.\eqref{eq:ordering for lines Xw}. 
Note that, for each $j$, the restriction of $\preceq_{\widehat{w}}$ to $S \setminus I^{w_j}$ coincides with this order $\preceq$ on $S \setminus I^{w_j}$.
\end{remark}

We now define a square matrix of order $n_m$ associated with the set of matrices 
$\{\widehat{M}_{P^{w_j}} \mid 1 \le j \le m\}$ as
\[
\widehat{M} := 
\begin{pmatrix}
\widehat{M}_{P^{w_1}}\\
\vdots\\
\widehat{M}_{P^{w_j}}\\
\vdots\\
\widehat{M}_{P^{w_m}}
\end{pmatrix}.
\]
We have the following lemma.
\begin{lemma}\label{lem:integrable}
    With respect to the ordering $\preceq_{\widehat{w}}$
    defined above, 
    the square matrix $\widehat{M}$ is a unipotent upper triangular matrix with entries in $\mathbb{Z}$.
\end{lemma}
\begin{proof}
By Remark~\ref{rmk:restriction to sets} together with Lemma~\ref{lem:identity}, for $1\leq j\leq m$, we have $M_{P^{w_j}}$ is the identity matrix. Hence, we obtain 
\[
\widehat{M} =
\begin{pmatrix}
I_{n_1} &M_1&  \cdots &  \cdots & \cdots\\
0       & I_{n_2-n_1} & M_2 & \cdots  & \cdots\\
\vdots  & \vdots & \ddots & I_{n_{m-1}-n_{m-2}} & M_{m-1}\\
0       & 0      & \cdots & 0 & I_{n_m-n_{m-1}}
\end{pmatrix},
\]
where $I_r$ denotes the identity matrix of order $r$. 
\end{proof}

The above lemma shows that the linear system in Eq.~\eqref{eq:linear systems} contains a subsystem of full rank. Therefore, to solve the linear system in Eq.~\eqref{eq:linear systems}, it suffices to solve the subsystem of equations corresponding to $\widehat{M}$ (and hence to the sets $\mathcal{B}_{P^{w_j}}$), with respect to the ordering $\preceq_{\widehat{w}}$: That is,
\begin{equation}\label{eq:Mhat matrices}
\widehat{M} 
\begin{pmatrix}
c_{j,\alpha}
\end{pmatrix}_{1\le j\le m,\, \alpha \in S \setminus I^{w_j}} 
=
\begin{pmatrix}
\operatorname{ht}\big(((w_{j+1}\cdots w_m)^{-1} \mu_{w_j}(\alpha))^\vee\big)+1
\end{pmatrix}_{1\le j\le m,\, \alpha \in S \setminus I^{w_j}}. 
\end{equation}
Thus, we obtain the following sublinear system: 
\begin{align}\label{eq:sublinear}
\sum_{\alpha \in S \setminus I^{w_j}} 
c_{j,\alpha}\,\langle \omega_{\alpha}, \mu_{w_j}(\gamma) \rangle
&+ \sum_{t=j+1}^{m} \sum_{\beta \in S \setminus I^{w_j}}
c_{t,\beta}\,
\left\langle \omega_{\beta}, (w_{j+1}\cdots w_t)^{-1}\mu_{w_j}(\gamma) \right\rangle \notag\\
&= \operatorname{ht}\big(((w_{j+1}\cdots w_m)^{-1}\mu_{w_j}(\gamma))^\vee\big)+1.
\end{align}
 By Lemma~\ref{lem:identity}, we solve the system \eqref{eq:sublinear} inductively to compute the coefficients $c_{j,\alpha}$. 
Hence, if $\widehat{X}(\widehat{w})$ is Gorenstein, then the coefficients $c_{j,\alpha}$ give the unique  solution of the system of equations in Eq.~\eqref{eq:linear systems}. In particular, for each $1 \leq j \leq m$ and $\alpha \in S \setminus I^{w_j}$, we have
\begin{align}\label{eq:sublinearsol}
c_{j,\alpha}
& = \big(\operatorname{ht}\big(((w_{j+1}\cdots w_m)^{-1}\mu_{w_j}(\alpha))^\vee\big)+1\big)- \sum_{t=j+1}^{m} \sum_{\beta \in S \setminus I^{w_j}}
c_{t,\beta}\,
\left\langle \omega_{\beta}, (w_{j+1}\cdots w_t)^{-1}\mu_{w_j}(\alpha) \right\rangle .
\end{align}

 We now conclude this section by proving our main theorem.

\begin{theorem}\label{thm: Fano gBS} Suppose $G$ is of simply laced type. The generalized Bott-Samelson variety $\widehat{X}(\widehat{w})$ is Fano (resp. weak Fano) if and only if the following conditions hold:

\begin{enumerate}
    \item For each $1 \le j \le m$ and every 
    \[
    \eta \in \operatorname{Peaks} X(w_j) \setminus \mathcal{B}_{P^{w_j}},
    \]
    we have
    \begin{equation}\label{eq:height relations for Xwi 2}
    \operatorname{ht}(\eta^\vee) + 1
    = \sum_{\alpha \in S \setminus I^{w_j}} 
      d_\alpha \big( \operatorname{ht}(\mu_{w_j}(\alpha)^\vee) + 1 \big),
    \end{equation}
with $d_\alpha :=\langle \omega_{\alpha}, \eta \rangle \in \mathbb{Z}_{\ge 0}$.
   In other words, $d_\alpha$ is the coefficient of $\alpha^\vee$ in the expansion of $\eta^\vee$.
    
    \item For all $1 \le j \le m$ and all $\alpha \in S \setminus I^{w_j}$, we have 
    \[
    c_{j,\alpha} > 0 \quad (\text{resp. } c_{j,\alpha} \ge 0).
    \]
\end{enumerate}
\end{theorem}
\begin{proof}
Assume that $\widehat{X}(\widehat{w})$ is Fano. Then $\widehat{X}(\widehat{w})$ is Gorenstein, and by Theorem~\ref{thm: Gorenstein gBS}, each $X(w_j)$ is also Gorenstein. Therefore, condition (1) follows from Theorem~\ref{thm: Gorenstein Xw}.  

Moreover, if $\widehat{X}(\widehat{w})$ is Gorenstein, there exist integers 
$c_{j,\alpha} \in \mathbb{Z}$ such that
\begin{align}
\sum_{j=1}^m \sum_{\alpha \in S \setminus I^{w_j}} 
c_{j,\alpha}\,[\mathcal{L}_{j,\alpha}]
= [-K_{\widehat{X}(\widehat{w})}]\notag
\end{align}
(see Eq.\eqref{eq: c_{i,alpha} for Fano}).
Furthermore, by Theorem~\ref{thm:very ample}, the line bundle 
$\mathcal{O}\!\big([-K_{\widehat{X}(\widehat{w})}]\big)$ is very ample (resp. nef) 
if and only if $c_{j,\alpha} > 0$ (resp. $c_{j,\alpha} \ge 0$) for all $1 \le j \le m$ and 
$\alpha \in S \setminus I^{w_j}$. Hence, condition (2) holds.

Conversely, if condition (1) holds, then by Theorem~\ref{thm: Gorenstein gBS}, the variety $\widehat{X}(\widehat{w})$ is Gorenstein. The condition (2) ensures that 
$\mathcal{O}\!\big([-K_{\widehat{X}(\widehat{w})}]\big)$ is very ample, thanks to Theorem~\ref{thm:very ample}. Therefore, we conclude that $\widehat{X}(\widehat{w})$ is Fano.
Similarly, the assertion for weak Fano follows by Theorem \ref{thm:very ample} together with Proposition \ref{prop:big}. 
\end{proof}

\begin{remark}  
    Recall from Corollary~\ref{cor: Fano for Xw} that, in our setting, if $X(w) \subseteq G/P^w$ is Gorenstein, then it is Fano. Therefore, if $\widehat{X}(\widehat{w})$ is Fano, then each $X(w_j)$ must also be Fano.
However, the converse is not true in general; see Example~\ref{ex: not Fano}.
\end{remark}

\section{Some special cases}\label{sec:Recover}
Throughout this section, $G$ is not necessarily of simply laced type.
In this 
section, we give explicit formulas for the coefficients $c_{i, \alpha}$'s for several special cases of generalized Bot-Samelson varieties, namely  Bott-Samelson varieties,  $G$-Bott-Samelson varieties and
minuscule generalized Bott-Samelson varieties. Moreover, we recover the results of \cite{BS24} (and those of \cite{BC18} in the reduced case), as well as the results of \cite{Per07}.

\subsection{\texorpdfstring{$\mu_w-$}~functions for flag Bott-Samelson varieties}
In this subsection, we prove the existence of the functions $\mu_{w_i}$'s (see Section~\ref{subsub: Fano for Xw}) for flag Bott-Samelson varieties. Let $J \subseteq S$, and let $w_{0,J}$ denote the longest element of the subgroup $W_J$ of $W$. Then
\[
I^{w_{0,J}} \cap \operatorname{Supp}(w_{0,J})=\emptyset.
\]
Consequently, by Definition~\ref{def:gBS}, any generalized decomposition of the form
\[
\widehat{w}=(w_{0,J_1},\dots,w_{0,J_m}),
\]
where $J_k\subseteq S$ for each $1\leq k\leq m$, is admissible. Furthermore, in this case, the variety $\widehat{X}(\widehat{w})$ is called the \emph{flag Bott-Samelson variety} associated with $\widehat{w}$.

For each $1 \leq k \leq m$, let $w_k:=w_{0,J_k}$. Since $I^{w_{{k}}}=S\setminus \operatorname{Supp}(w_{k})$, by the definitions of $R^{+}(w)$ and $\operatorname{Peaks}X(w)$ (see Subsection~\ref{subsec: peaks for Xw}), it follows that $ R^+(w_{k})=S\setminus I^{w_{k}} 
$, 
and then \begin{align*}    \operatorname{Peaks}X(w_{k})&=\{\alpha\in S\setminus I^{w_{k}}\mid\, \ell(w_{k}s_\alpha)=\ell(w_{k})-1, \,\,w_{k}s_\alpha\in W^{P^{w_{k}}}\}= S\setminus I^{w_{k}}.\end{align*}

Let \[
\mathcal{M}_0=S\setminus I^{w_{k}}
\subseteq R^+_{w_{k},B}.
\]
Now consider the identity map
\begin{equation}\label{eq:mu map is identity}
\mu_{w_{k}} :
S\setminus I^{w_{k}}
\longrightarrow
\mathcal{M}_0,
\qquad
\alpha \longmapsto \alpha.
\end{equation}
Then, $\mathcal{M}_0$ is a
$P^{w_{k}}$-adaptation of $S\setminus I^{w_{k}}$. Moreover, by the construction of $\mathcal{B}_{P^w}$, we have \[
\mathcal{B}_{P^{w_k}}=\mathcal{M}_0=S\setminus I^{w_k}, \qquad \forall \,\,1\leq k\leq m.
\]

Consequently, we obtain the following.

\begin{proposition}Keep the same notation as above. For the flag Bott-Samelson varieties, we let 
\[
\mu_{w_k}:S\setminus I^{w_k}\longrightarrow \mathcal B_{P^{w_k}}.
\]
be the identity map.
In particular, $ \langle\omega_\alpha,\mu_{w_k}(\beta)\rangle=\delta_{\alpha \beta}$
for every $\alpha, \beta\in S\setminus I^{w_k}$.
\end{proposition}
\begin{remark}
    For flag Bott-Samelson varieties, we have  $\operatorname{Peaks}X(w_{k})= \mathcal{B}_{P^{w_k}}$ for each $1 \leq k \leq m$. Thus, condition $(1)$ of Theorem \ref{thm: Fano gBS} is automatically satisfied. 
\end{remark}

Any Bott-Samelson variety arises as a flag Bott-Samelson variety by taking
$\#(J_k)=1$ for every $1\leq k\leq m$. Similarly, every $G$-Bott-Samelson variety is also a special case of a flag Bott-Samelson variety, obtained by assuming that
\[J_1=S,\qquad \text{and}\qquad
\#(J_k)=1,\qquad \text{for all }2\leq k\leq m.
\]

\subsection{Fano Bott-Samelson varieties}
Let $w\in W$, and let $\tilde w=s_{i_1}s_{i_2}\cdots s_{i_r}$ be an expression (not necessarily reduced) of $w$.

Take $m = r$ and $w_j = s_{i_j}$ for all $1 \leq j \leq r$ in the construction of $\widehat{X}(\hat{w})$. Then we obtain
$
\widehat{X}(\hat{w}) \simeq Z(\tilde{w}).
$ For a given simple root $\alpha_i \in S$, we have
\[
\mathrm{Peaks}(X(s_i)) = \mathcal{B}_{P^{s_i}} = \{\alpha_i\}.
\]
Thus, condition~(1) of Theorem \ref{thm: Fano gBS} is vacuous. Note that
\[
\mathrm{Peaks}(Z(\tilde{w})) = \left\{ s_{i_r}s_{i_{r-1}}\cdots s_{i_2}(\alpha_{i_1}),\, \dots,\, s_{i_r}(\alpha_{i_{r-1}}), \, \alpha_{i_r} \right\}.
\]

Moreover, for each $1 \leq j \leq r$, we have
\[
S \setminus I^{s_{i_j}} = \{\alpha_{i_j}\} \quad \text{and} \quad \mu_{s_{i_j}}(\alpha_{i_j}) = \alpha_{i_j}.
\]
Hence, we obtain
\begin{equation}
\sum_{j=1}^r 
c_{j,\alpha_{i_j}} \,[\mathcal{L}_{j,\alpha_{i_j}}]
= [-K_{Z(\tilde{w})}].
\end{equation}

For $1 \leq j \leq r$, let
\[
c_j := c_{j,\alpha_{i_j}}
\quad\text{and}\quad
d_j := \left\langle \sum_{t=j}^r \alpha_{i_t},\, \alpha_{i_j} \right\rangle.
\]

Then the coefficients $c_j$'s can be computed inductively as follows.

\begin{theorem}\label{thm:BSCia} For each $1 \leq j \leq r-1$, 
\[
c_j = d_j - \sum_{t=j+1}^{r}\langle \omega_{i_t},\alpha_{i_j}\rangle c_t\quad \text{and} \quad c_r=2.
    \]        
\end{theorem}

\begin{proof}
    By the construction of the coefficients $c_j$ (see Eq.\eqref{eq:sublinear}), we obtain a system of linear equations of the form
\begin{equation}\label{eq:bsmatrices}
A(\tilde{w}) (c_1, \dots, c_r)^T = b(\tilde{w})^T,
\end{equation}
where
\[
A(\tilde{w}) = (a_{jk}) \quad \text{with} \quad
a_{jk}
=
\left\langle
\omega_{i_k},
\, s_{i_k}\cdots s_{i_{j+1}}(\alpha_{i_j})
\right\rangle,
\]
and
\[
b(\tilde{w}) = (b_1, \ldots, b_r) \quad \text{with} \quad
b_j = \operatorname{ht}\!\big(((s_{i_{j+1}}\cdots s_{i_r})^{-1}\alpha_{i_j})^\vee\big) + 1.
\]
In order to solve this system,  we define two matrices associated with the expression $\tilde{w}$ as follows:
\[
N(\tilde{w}) := (a'_{jk})_{1 \leq j,k \leq r}, 
\quad \text{where} \quad
a'_{jk} =
\begin{cases}
\langle \omega_{i_k}, \alpha_{i_j} \rangle & \text{if } j \leq k,\\
0 & \text{otherwise},
\end{cases}
\]
and
\[
P(\tilde{w}) := (b'_{jk})_{1 \leq j,k \leq r}, 
\quad \text{where} \quad
b'_{jk} =
\begin{cases}
1 & \text{if } j = k,\\
\langle \alpha_{i_k}, \alpha_{i_j} \rangle & \text{if } j < k,\\
0 & \text{otherwise}.
\end{cases}
\]

By construction, both $N(\tilde{w})$ and $P(\tilde{w})$ are unipotent matrices and are invertible over $\mathbb{Z}$. 
The following identities will be established in Proposition~\ref{pro: matrix relation} below:
\[ P(\tilde{w}) \cdot A(\tilde{w}) = N(\tilde{w}) \quad \text{and} \quad P(\tilde{w}) \cdot b(\tilde{w})^T = d(\tilde{w}),\] where $d(\tilde{w})=(d_1, \dots, d_r)^T$.
Now consider the new system of equations: 
\[
N(\tilde{w}) (c_1, \dots, c_r)^T = d(\tilde{w}).\]
 By the definition of $N(\tilde{w})$, we solve this system of equations recursively using the row-reduced echelon form and obtain the unique solution
    \[
    c_j = d_j - \sum_{t=j+1}^{r}\langle \omega_{i_t},\alpha_{i_j}\rangle c_t, 
    \]
    for all $1\leq j\leq r$.    This completes the proof.
\end{proof}

\begin{proposition}\label{pro: matrix relation}
The following statements hold:
\begin{enumerate}
    \item $P(\tilde{w}) \cdot A(\tilde{w}) = N(\tilde{w})$.
    \item $P(\tilde{w}) \cdot b(\tilde{w})^T = d(\tilde{w})$, where $d(\tilde{w})=(d_1, \dots, d_r)^T$, with
    $d_k = \left\langle \sum_{t=k}^r \alpha_{i_t},\, \alpha_{i_k} \right\rangle
    \quad \text{for } 1 \leq k \leq r.
    $
\end{enumerate}
\end{proposition}
\begin{proof}
  For $j \leq k$, the $(j,k)$-th entry of the matrix $P(\tilde{w}) \cdot A(\tilde{w})$ is given by
\begin{equation}\label{eq:ajk'}
n'_{jk} :=
\left\langle \omega_{i_k},\, s_{i_k}\cdots s_{i_{j+1}}(\alpha_{i_j}) \right\rangle
+ \sum_{t=j+1}^{k}
\langle \alpha_{i_t},\, \alpha_{i_j} \rangle
\left\langle \omega_{i_k},\, s_{i_k}\cdots s_{i_{t+1}}(\alpha_{i_t}) \right\rangle,
\end{equation}
and all other entries are zero.  Observe that
\begin{align*}
\left\langle \omega_{i_k},\, s_{i_k}\cdots s_{i_{j+1}}(\alpha_{i_j}) \right\rangle
&= \left\langle s_{i_{j+1}}\cdots s_{i_k}(\omega_{i_k}),\, \alpha_{i_j} \right\rangle \\
&= \left\langle s_{i_{j+2}}\cdots s_{i_k}(\omega_{i_k})
- \big\langle s_{i_{j+2}}\cdots s_{i_k}(\omega_{i_k}),\, \alpha_{i_{j+1}} \big\rangle \alpha_{i_{j+1}},\,
\alpha_{i_j} \right\rangle \\
&= \left\langle s_{i_{j+2}}\cdots s_{i_k}(\omega_{i_k}),\, \alpha_{i_j} \right\rangle
- \langle \alpha_{i_{j+1}},\, \alpha_{i_j} \rangle
\left\langle \omega_{i_k},\, s_{i_k}\cdots s_{i_{j+2}}(\alpha_{i_{j+1}}) \right\rangle.
\end{align*}

Iterating this procedure, we obtain
\[
\left\langle \omega_{i_k},\, s_{i_k}\cdots s_{i_{j+1}}(\alpha_{i_j}) \right\rangle
= \langle \omega_{i_k},\, \alpha_{i_j} \rangle
- \sum_{t=j+1}^{k}
\langle \alpha_{i_t},\, \alpha_{i_j} \rangle
\left\langle \omega_{i_k},\, s_{i_k}\cdots s_{i_{t+1}}(\alpha_{i_t}) \right\rangle.
\]

Substituting this into Eq.\eqref{eq:ajk'}, we deduce that
$
n'_{jk} = \langle \omega_{i_k},\, \alpha_{i_j} \rangle = a'_{jk}.
$
Therefore, we conclude that
$P(\tilde{w}) \cdot A(\tilde{w}) = N(\tilde{w})$,
which proves (1).

  Proof of (2): By Corollary \ref{cor:antican} with Eq.(\ref{eq:ACS for BSv}), we have 
  \[
   b_k=\operatorname{ht}(s_{i_r}\cdots s_{i_{k+1}}(\alpha_{i_k})^\vee)+1=\sum_{j=k}^{r}\langle\beta_{j},\beta_{k}\rangle.
   \]
  We proceed by induction on $r$. If $r=1$, then $d_1=\langle \beta_1,\beta_1\rangle=\langle\alpha_1,\alpha_1\rangle$.  Now we assume that the result is true for $r-1$. 
   The matrix multiplication $P(\tilde{w})\cdot b(\tilde{w})^T$ simply gives 
$d_r=\langle\beta_r,\beta_r\rangle=\langle\alpha_{i_r},\alpha_{i_r}\rangle$, 
 and for $1\leq k\leq r-1$, the $k$-th entry of the column matrix $P(\tilde{w})\cdot b(\tilde{w})^T$ is 
   \begin{align*}  d_{k}:&=b_{k}+\sum_{t=k+1}^{r}\langle \alpha_{i_{t}},\alpha_{i_{k}}\rangle b_{t}\\
   &=\left(\sum_{j=k}^{r-1}\langle\beta_{j},\beta_k\rangle+\sum_{t=k+1}^{r-1}\langle\alpha_{i_t},\alpha_{i_{k}}\rangle\sum_{j=t}^{r-1}\langle\beta_{j},\beta_{t}\rangle\right)+ \left(\langle\beta_{r},\beta_{k}\rangle+\sum_{t=k+1}^{r}\langle\alpha_{i_{t}},\alpha_{i_k}\rangle\langle\beta_{r},\beta_{t}\rangle\right).
   \end{align*}
By induction, we have\[
\sum_{j=k}^{r-1}\langle\beta_{j},\beta_k\rangle+\sum_{t=k+1}^{r-1}\langle\alpha_{i_t},\alpha_{i_{k}}\rangle\sum_{j=t}^{r-1}\langle\beta_{j},\beta_{t}\rangle=\left\langle\sum_{t=k}^{r-1}\alpha_{i_{t}},\alpha_{i_k}\right\rangle.
\] Note that \[
\langle\alpha_{i_r},\alpha_{i_k}\rangle=-\langle\beta_{r},\beta_{k}\rangle-\sum_{t=k+1}^{r-1}\langle\alpha_{i_{t}},\alpha_{i_k}\rangle\langle\beta_{r},\beta_{t}\rangle.
\] This implies that \[
\langle\beta_{r},\beta_{k}\rangle+\sum_{t=k+1}^{r}\langle\alpha_{i_{t}},\alpha_{i_k}\rangle\langle\beta_{r},\beta_{t}\rangle=\langle\alpha_{i_{r}},\alpha_{i_k}\rangle.
\] Hence, we obtain \[
d_k=\left\langle\sum_{t=k}^{r-1}\alpha_{i_{t}},\alpha_{i_k}\right\rangle+\langle\alpha_{i_{r}},\alpha_{i_k}\rangle=\left\langle\sum_{t=k}^{r}\alpha_{i_{t}},\alpha_{i_k}\right\rangle.
\]
This completes the proof for $(2)$.
      \end{proof}

For $1\leq j\leq r$, set 
\[
m_{jj}:=\langle \alpha_{i_j}, \alpha_{i_j} \rangle=2 \quad \text{and} \quad m_{jk}:=\langle\alpha_{i_j}-\sum_{l=k+1}^jm_{jl}\omega_{i_l}, \alpha_{i_k}\rangle \quad \text{for all} \quad 1\leq k\leq j-1. \]
  \begin{corollary}\label{prop:BS} We have  
    \[
(c_1,\dots,c_r)=
\left(
\sum_{l=1}^{r}m_{l1},
\sum_{l=2}^{r}m_{l2},
\dots,
2
\right).
\]
\end{corollary}
\begin{proof}
The proof follows from the definitions of $m_{jk}$ and $d_j$, together with Theorem~\ref{thm:BSCia}.
\end{proof}

Therefore, by Theorem~\ref{thm: Fano gBS}, we derive the following (see \cite[Theorem 1.2]{BS24}):
\begin{theorem}\label{thm:FanoBS}
    The Bott-Samelson variety $Z(\tilde w)$ is Fano (weak Fano) if and only if $\sum_{l=j}^rm_{lj}>0$ (resp. $\geq 0$) for all $1\leq j\leq r$.
\end{theorem}

\subsection{Fano \texorpdfstring{$G$}{}-Bott-Samelson varieties} Now, consider the generalized decomposition 
\[
\widehat{w}=(w_0,w_1,\dots,w_r)=(w_{0},s_{i_1},\dots,s_{i_r})\] of $w=w_{0}s_{i_1}\cdots s_{i_r}\in W$, where $w_{0}$ is the maximal element in $W$, and $w_j=s_{i_j}$ for each $1\leq j\leq r$. Note that 
\[I^{w_0}\cap \operatorname{Supp}(w_0)=\emptyset \quad \text{and} \quad I^{s_{i_j}}\cap \operatorname{Supp}(s_{i_j})=\emptyset,\]
for $1\leq j\leq r$. Hence, $\widehat{w}$ is an admissible generalized decomposition. Moreover, we have 
\[P^{w_0}=B, \quad P_{w_0}=G\quad \text{and} \quad G_{w_0}=G. 
\]  
Therefore, \begin{align*}
    \widehat{X}(w_0,s_{i_1},\dots, s_{i_r})&=\overline{(P_{w_0}\cap G_{w_0})w_0(P^{w_0}\cap G_{w_0})}\times ^{P^{w_0}\cap G_{w_0}} \widehat{X}(s_{i_1},\dots, s_{i_r})\\&\simeq \overline{Gw_0B}\times^{B} \widehat{X}(s_{i_1},\dots, s_{i_r})\\&
    \simeq G\times^B Z(\tilde{w}).
\end{align*}
Thus, in this case, the generalized Bott-Samelson varieties coincide with the $G$-Bott-Samelson varieties.

Our goal in this subsection is to prove the following characterization of Fano and weak Fano $G$-Bott-Samelson varieties (see also \cite[Theorem 1.7]{BS24}). For a fixed integer $1 \leq j \leq r$, define
\[
\lambda_j:=\alpha_{i_j}-\sum_{l=1}^{j} m_{jl}\,\omega_{i_l}.
\]
\begin{theorem} \label{thm: GBS Fano} 
The variety $G\times^B Z(\tilde{w})$ is Fano (weak Fano) if and only if the following conditions hold:
\begin{enumerate}
    \item $\displaystyle \sum_{l=j}^{r} m_{lj} > 0$    (resp. $\displaystyle \sum_{l=j}^{r} m_{lj} \geq 0$)
    for all $1 \leq j \leq r-1$.  
    \item $2\rho + \displaystyle\sum_{j=1}^{r} \lambda_j$ is regular dominant   (resp. dominant).
\end{enumerate}
  \end{theorem}
  \begin{proof}

We begin by considering an associated decomposition of $\widehat{w}$,
\[
\tilde{w}=(s_{1,1},\dots,s_{r_1,1},s_{1,2},\dots,s_{1,r}),
\]
where we fix a reduced expression for $w_0$ such that  \[
w_0=\prod_{k=1}^{r_1}s_{k,1},\quad \text{and}\quad s_{i_j}=s_{1,j} \quad \forall \quad 1\leq j\leq r.
\] 
Recall the total ordering $\preceq_w$ defined in Subsection~\ref{subsection: M1}, then there exists some $1\leq u_k\leq r_1$, where $1\leq k\leq n$ such that the ordering over the set of simple roots $S$ is given as follows \[
S=\{\alpha_{u_1,1}\prec_{\widehat{w}}\alpha_{u_2,1}\prec_{\widehat{w}}\cdots \prec_{\widehat{w}}\alpha_{u_n,1}\}.
\]
Note that for the maximal element $w_0$, we have \[S\setminus I^{w_0}=S,\quad \text{and}\quad 
\operatorname{Peaks}X(w_0)=\mathcal{B}_{P^{w_0}}=S.
\] Also for a given simple root $\alpha_{i}\in S$, we have \[S \setminus I^{s_{i_j}} = \{\alpha_{i_j}\}, \quad \text{and}\quad 
\operatorname{Peaks}X(s_i)=\mathcal{B}_{P^{s_i}}=\{\alpha_i\}.
\] 
Thus, condition (1) of Theorem \ref{thm: Fano gBS} is vacuous, and so $\widehat{X}(\widehat{w})$ is Gorenstein. By the definition of peaks for $\widehat{X}(\widehat{w})$, we have the following. \begin{align*}
\operatorname{Peaks}\widehat{X}(\widehat{w})&=(s_{i_r}\cdots s_{i_1})\big(\operatorname{Peaks}X(w_0)\big) \sqcup  \operatorname{Peaks}Z(\tilde{w})\\
&=\bigsqcup_{1\leq k\leq n}\left\{s_{i_r}\cdots s_{i_1}(\alpha_{u_k,1})  \right\} \bigsqcup  \left\{ s_{i_r}s_{i_{r-1}}\cdots s_{i_2}(\alpha_{i_1}),\, \dots,\, s_{i_r}(\alpha_{i_{r-1}}), \, \alpha_{i_r} \right\}.
\end{align*}

Moreover, for each $1\leq k\leq n$ and for each  $1 \leq j \leq r$, we take
\[ \mu_{w_0}(\alpha_{u_k,1})=\alpha_{u_k,1},
 \quad \text{and} \quad \mu_{s_{i_j}}(\alpha_{i_j}) = \alpha_{i_j}.
\]
Since $\widehat{X}(\widehat{w})$ is Gorenstein. Hence, there exist some integers $c_{u_k}$ for $1\leq k\leq n$, and $c_j$ for $1\leq j\leq r$ such that \[
\sum_{1\leq k\leq n}c_{u_k}[\mathcal{L}_{1,\alpha_{u_k,1}}]+\sum_{j=1}^{r}c_j[\mathcal{L}_{j,\alpha_{i_j}}]=-[K_{\widehat{X}(\widehat{w})}].
\]

Now, by the construction of the coefficients $c_{u_k}$ and $c_j$, we obtain a system of linear equations of the form

\begin{equation}\label{eq: GBS soultion space}
\begin{pmatrix}
    I_{n}&M_1\\ [0]_{r\times n}&A(\tilde{w})
\end{pmatrix}
    (c_{u_1},\dots,c_{u_n},c_1,\dots,c_r)^T=
        (b',\, b(\tilde{w})
    )^T,
\end{equation} where \[
b'=(b_1',\dots,b_n') \quad \text{with}\quad b_k'=\operatorname{ht}(s_{i_r}\cdots s_{i_1}(\alpha_{u_k,1})^\vee)+1,
\]and the matrices $A(\tilde{w})$, $b(\tilde{w})$ and $M_1$ are the same as defined in Eq.\eqref{eq:bsmatrices} and in Subsection \ref{subsection: M1}, respectively.

Clearly, the values of the variables $c_1,\dots,c_r$ are determined exactly as in Corollary~\ref{prop:BS}. Precisely, we have \[
(c_1,\dots,c_r)=\left(
\sum_{l=1}^{r}m_{l1},
\sum_{l=2}^{r}m_{l2},
\dots,
2
\right).
\] We now determine the remaining variables $c_{u_k}$, for $1 \leq k \leq n$, by solving the corresponding system of linear equations. To this end, we fix a $1\leq k\leq n$ and define an expression \[
\tilde{w}_{k}:=(s_{u_k,1},s_{i_1},\dots,s_{i_r}).
\]
Since the solutions $c_{u_k}$ are independent of each other and depend only on the variables $c_1,\dots,c_r$. Thus, the solution of $c_{u_k}$ is given by the following system of linear equations 

\[
A(\tilde{w}_k)(c_{u_k},c_1\dots,c_r)^T=(b_k',\, b(\tilde{w}))^T=(b(\tilde{w}_k))^T.
\]

After multiplying the matrix $P(\tilde{w}_k)$ on both sides, and using Proposition \ref{pro: matrix relation}, we get \[
N(\tilde{w}_k)(c_{u_k},c_1\dots,c_r)^T=d(\tilde{w}_k).
\]
By the definitions of $N(\tilde{w}_k)$, $d(\tilde{w}_k)$ and using row-reduced echelon form, we obtain \begin{align*}
c_{u_k}&=\left\langle \alpha_{u_k}+ \sum_{j=1}^{r}\alpha_{i_j},\alpha_{u_k} \right\rangle-\sum_{j=1}^{r}c_j\left\langle\omega_{i_j}, \alpha_{u_k} \right\rangle\\
&= \left\langle \alpha_{u_k},\alpha_{u_k}\right\rangle+\left\langle\sum_{j=1}^{r}\alpha_{i_j},\alpha_{u_k}\right\rangle -\left\langle \sum_{j=1}^{r}c_j\omega_{i_j}, \alpha_{u_k} \right\rangle\\
&= \left\langle \alpha_{u_k},\alpha_{u_k}\right\rangle+ \left\langle \sum_{j=1}^r\alpha_{i_j}-\sum_{j=1}^r\bigg( \sum_{l=j}^r m_{lj} \bigg)\omega_{i_j}, \alpha_{u_k}\right\rangle, \quad \text{where }c_j= \sum_{l=j}^r m_{lj}\\
&=\left\langle 2\rho,\alpha_{u_k} \right\rangle+\left\langle\sum_{j=1}^{r}\lambda_j,\alpha_{u_k}\right\rangle\\
&=\left\langle 2\rho+ \sum_{j=1}^{r}\lambda_j,\alpha_{u_k}\right\rangle.
\end{align*}
For each $1 \leq k \leq n$, we associate the expression $\tilde{w}_k=(s_{u_k,1},s_{i_1},\dots,s_{i_r})$, and determine the corresponding coefficient $c_{u_k}$ by applying the above procedure. Therefore, by Theorem \ref{thm:very ample} (or by Theorem \ref{thm: Fano gBS}), $G\times^B Z(\tilde{w})$ is Fano (weak Fano) if and only if for each $1\leq j\leq r$ and for each $1\leq k\leq n$, \[c_{j}>0(\text{resp.}~\geq 0),\quad  \text{and} \quad c_{u_k}>0(\text{resp.}~\geq 0).\] This completes the proof.
\end{proof}

\subsection{Fano minuscule generalized Bott-Samelson varieties}\label{sec:minuscle}

In this subsection, we consider $\widehat{w}$ to be a good generalized reduced decomposition of a minuscule element $w\in W$ as in \cite[Section 5.2]{Per07} (see also Definition \ref{def:gBS}). 

Recall that any good generalized reduced decomposition of a minuscule element is admissible.
The good and minuscule conditions give the nice structure of the parabolic subgroups $P_{J_i}$ described in Definition \ref{def:gBS}. Precisely, the parabolic subgroups $P_{J_i}$ are the parabolic subgroups $P_{w_i\cdots w_m}$; for more details, see \cite[Proposition 5.3]{Per07}.

We first recall the notion of minuscule Schubert varieties.
\begin{definition}\ 

\begin{enumerate}    \item 
    A weight $\omega$ is called \emph{minuscule} if we have $\langle \omega,\gamma\rangle\leq 1$ for all positive roots $\gamma\in R^+$.
    \item An element $w\in W$ is called \emph{minuscule} if $\sum_{\alpha\in S\setminus I^w}\omega_{\alpha}$ is a minuscule weight.
    \item A flag variety $G/P$ is called \emph{minuscule} if there exists a minuscule element $w\in W$ such that $P^w=P$. Moreover, Schubert varieties of a minuscule flag variety are called \emph{minuscule Schubert varieties}.
    \end{enumerate}
\end{definition}

\begin{remark}\label{rmk:minuscluse1} To study minuscule flag varieties and their Schubert varieties, it suffices to restrict to simply laced groups; see \cite[Remark 3.4]{Per07}. So, in the rest of the section we assume that $G$ is of simply laced type. 
\end{remark}

From the proof of Theorem~$3.1$ in \cite{LMS79}, we observe the following:
\begin{corollary}\label{cor:LMS}
    Let $\tilde{w}=(s_{i_1},\dots,s_{i_r})$ be a reduced expression of a minuscule element $w\in W$. Then for $1\leq k\leq r$, we have\[
\langle \omega_{i_r},s_{i_r}\cdots s_{i_{k+1}}(\alpha_{i_k}) \rangle=1,
\]
\end{corollary}

\begin{proposition}\label{pro: Gorensteinmin} The minuscule generalized Bott-Samelson variety $\widehat{X}(\widehat{w})$ is Gorenstein if and only if, for each $1 \leq j \leq m$, the peaks of $X(w_j)$ have the same height.
\end{proposition}
\begin{proof} 
Since $w_j$ is minuscule, we have $|S\setminus I^{w_j}|=1$ for all $1 \leq j \leq m$. By Corollary~\ref{cor:LMS}, any choice of function from $S\setminus I^{w_j}\to \operatorname{Peaks}X(w_j)$ gives the required $\mu_{w_j}$ function as in Lemma \ref{lem:B_w,P^w}. 
Therefore, the proof follows from Theorem~\ref{thm: Gorenstein Xw} and Corollary~\ref{thm: Gorenstein gBS}.
\end{proof}

For $1\leq j\leq m$, let $\alpha_j$ be the element of $S\setminus I^{w_j}$. If the variety $\widehat{X}(\widehat{w})$ is Gorenstein, then we can explicitly solve for the coefficients $c_{j,\alpha_j}$ in Eq.\eqref{eq:sublinearsol} using Corollary \ref{cor:LMS}, Proposition~\ref{prop: divisor classes for line bundles on gBS}, and Proposition \ref{anti-can} we obtain 

\begin{corollary}\label{cor:minuscule} For $1\leq j\leq m-1$,     
\[
c_{j,\alpha_j}= \operatorname{ht}\bigl(((w_{j+1}\cdots w_m)^{-1}\mu_{w_j}(\alpha_j))^\vee\bigl)-\operatorname{ht}\bigl(((w_{j+2}\cdots w_m)^{-1}\mu_{w_{j+1}}(\alpha_{j+1}))^\vee\bigl)\]
and
\[
c_{m,\alpha_{m}}=\operatorname{ht}(\alpha_r^\vee)+1=2.
\]
\end{corollary}

\subsubsection{Quivers} We now recall  
the quivers associated with minuscule Weyl group elements, and the notion of peaks for minuscule generalized Bott-Samelson varieties introduced by Perrin; see \cite[Section~5, p.~1274]{Per07}. We observe that this notion of peaks coincides with the notion of peaks defined in Section \ref{sec: Fano gBS}.

For a given reduced expression $\tilde{w}=(s_{i_1},\dots,s_{i_r})$ of $w=s_{i_1}\cdots s_{i_r}$, we define the quiver $Q_{\tilde{w}}$ associated to $\tilde{w}$ as follows:
\begin{definition}
The \emph{successor} $s(i)$ (resp. \emph{predecessor} $p(i)$) of an element $1\leq i\leq r$ is defined by $s(i)=\min \{k\in [1,r]\mid~k>i~\text{and}~\alpha_k=\alpha_i\}$ (resp. by $p(i)=\max \{k\in [1,r]\mid~k<i~\text{and}~\alpha_k=\alpha_i\}$).
\end{definition}
\begin{remark}
    Note that the successor and the predecessor of an element do not always exist.
\end{remark}
\begin{definition}
    The quiver $Q_{\tilde{w}}$ is the set of vertices $[1,r]$ together with the arrows defined as: there is an arrow from $i$ to $j$ if $\langle\alpha_i,\alpha_j\rangle\neq 0$ and $i<j<s(i)$ (or only $i<j$ if $s(i)$ does not exist).
\end{definition}

\begin{remark}\ 

\begin{enumerate}
\item The quiver $Q_{\tilde{w}}$ comes with a coloration of its vertices by simple roots, through the map $\beta: [1,r]\to S$ defined by $\beta(i)=\alpha_i$.
    \item 
    Recall that any minuscule element $w\in W$ has a unique reduced expression modulo commuting relations. Therefore, the quiver $Q_{\tilde{w}}$ does not depend on the reduced expression we choose for $w\in W^P$. 
    
    \end{enumerate}
\end{remark}
\begin{definition} \ 

\begin{enumerate}
    \item We can give a partial order relation $\preceq$ on $[1,r]$ by defining $i \preceq j$ if there is an arrow from $i$ to $j$. 
    \item A vertex of $Q_{\tilde{w}}$ is called \emph{peak} if it is minimal with respect to the partial ordering $\preceq$ and we denote by $\operatorname{Peaks}(Q_{\tilde{w}})$ the set of peaks of $Q_{\tilde{w}}$.  
    \end{enumerate}
\end{definition}

\begin{example}\label{quiver}
    Let $G/P^{s_4}= SL_{8}(\mathbb{C})/P^{s_4}$, where $P^{s_4}$ is the maximal parabolic subgroup associated with the minuscule fundamental weight  $\overline{\omega}=\omega_{\alpha_4}$. The longest element $w_{0,I^{s_4}}$ in $W^P$ is 
    $$w_{0,I^{s_4}}=(s_4s_3s_2s_1)(s_5s_4s_3s_2)(s_6s_5s_4s_3)(s_7s_6s_5s_4).$$ Its quiver $Q_{\overline{\omega}}$ is the following (all the arrows point down,  and the vertices covered by $\Box$ denote the peaks of the quiver. The map $\beta$ is the vertical projection to the Dynkin diagram $A_7$):

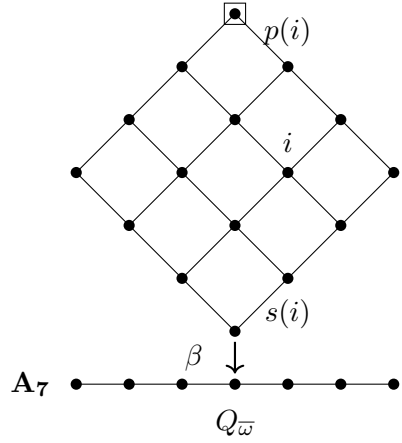
\begin{center}
\begin{tikzpicture}[scale=0.7]
  \fill (3,9) circle (3pt);
  \fill (2,8) circle (3pt);
  \fill (4,8) circle (3pt);
  \fill (1,7) circle (3pt);
  \fill (3,7) circle (3pt);
  \fill (5,7) circle (3pt);
  \fill (0,6) circle (3pt);
  \fill (2,6) circle (3pt);
  \fill (4,6) circle (3pt);
  \fill (6,6) circle (3pt);
  \fill (1,5) circle (3pt);
  \fill (3,5) circle (3pt);
  \fill (5,5) circle (3pt);
  \fill (2,4) circle (3pt);
  \fill (4,4) circle (3pt);
  \fill (3,3) circle (3pt);
  \fill (0,2) circle (3pt) node[left=6pt] {$\mathbf{A_7}$};
  \fill (1,2) circle (3pt);
  \fill (2,2) circle (3pt);
  \fill (3,2) circle (3pt) node [below=6pt] {$Q_{\overline{\omega}}$};
  \fill (4,2) circle (3pt);
  \fill (5,2) circle (3pt);
  \fill (6,2) circle (3pt);
  \node[draw, rectangle, inner sep=4pt] at (3,9) {};
  \draw (0,6) -- (1,5);
  \draw (1,5) -- (2,6);
  \draw (2,6) -- (1,7);
  \draw (2,6) -- (3,7);
  \draw (3,7) -- (2,8);
  \draw (3,7) -- (4,8) node[above=4pt] {$p(i)$};
  \draw (1,5) -- (2,4);
  \draw (2,4) -- (3,5);
  \draw (3,5) -- (2,6);
  \draw (3,5) -- (4,6) node[above=4pt] {$i$};
  \draw (4,6) -- (3,7);
  \draw (4,6) -- (5,7);
  \draw (2,4) -- (3,3);
  \draw (3,3) -- (4,4) node[below=4pt] {$s(i)$};
  \draw (4,4) -- (3,5);
  \draw (5,5) -- (4,6);
  \draw (4,4) -- (5,5);
  \draw (5,5) -- (6,6);
  \draw (6,6) -- (5,7);
  \draw (5,7) -- (4,8);
  \draw (4,8) -- (3,9);
  \draw (3,9) -- (2,8);
  \draw (2,8) -- (1,7);
  \draw (1,7) -- (0,6);
  \draw (0,2) -- (1,2);
  \draw (1,2) -- (2,2);
  \draw (2,2) -- (3,2);
  \draw (3,2) -- (4,2);
  \draw (4,2) -- (5,2);
  \draw (5,2) -- (6,2);
  \draw[->, thick] (3,2.8) -- (3,2.2) node[midway, left=8pt] {$\beta$};
\end{tikzpicture}
\end{center}
 By Definition \ref{def:Peaks Xw}, we have \[
\operatorname{Peaks}X(w_{0,I^{s_4}})=\{\eta:= (s_4s_5s_6s_7)(s_3s_4s_5s_6)(s_2s_3s_4s_5)s_1s_2s_3(\alpha_4)\}.
\]
\end{example} 
\begin{example}
Consider the Schubert variety $X(w)\subset SL_{8}(\mathbb{C})/P^{s_4}$ associated with $w=s_4s_1s_2s_3s_6s_5s_4$ (the minimal length representative in $W^{P^{s_4}}$). Its quiver is the following (all the arrows point down): 
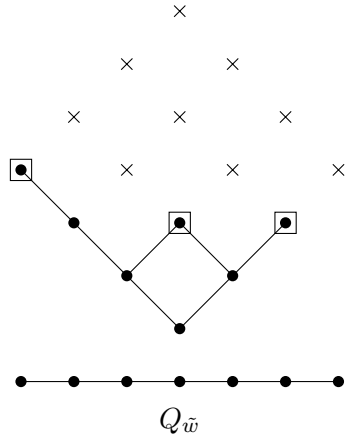
\begin{center}
\begin{tikzpicture}[scale=0.7]
  \fill (3,9) node[mark size=3pt,color=black] {\pgfuseplotmark{x}};
  \fill (2,8) node[mark size=3pt,color=black] {\pgfuseplotmark{x}};
  \fill (4,8) node[mark size=3pt,color=black] {\pgfuseplotmark{x}};
\fill (1,7) node[mark size=3pt,color=black] {\pgfuseplotmark{x}};
  \fill (3,7) node[mark size=3pt,color=black] {\pgfuseplotmark{x}};
  \fill (5,7) node[mark size=3pt,color=black] {\pgfuseplotmark{x}};
  \fill (0,6) circle (3pt);
  \fill (2,6) node[mark size=3pt,color=black] {\pgfuseplotmark{x}};
  \fill (4,6) node[mark size=3pt,color=black] {\pgfuseplotmark{x}};
  \fill (6,6) node[mark size=3pt,color=black] {\pgfuseplotmark{x}};
  \fill (1,5) circle (3pt);
  \fill (3,5) circle (3pt);
  \fill (5,5) circle (3pt);
  \fill (2,4) circle (3pt);
  \fill (4,4) circle (3pt);
  \fill (3,3) circle (3pt);
  \fill (0,2) circle (3pt);
  \fill (1,2) circle (3pt);
  \fill (2,2) circle (3pt);
  \fill (3,2) circle (3pt)node [below=6pt] {$Q_{\tilde{{w}}}$};
  \fill (4,2) circle (3pt);
  \fill (5,2) circle (3pt);
  \fill (6,2) circle (3pt);
  
  \draw (0,6) -- (1,5);
  \draw (1,5) -- (2,4);
  \draw (2,4) -- (3,5);
  \draw (2,4) -- (3,3);
  \draw (3,3) -- (4,4);
  \draw (4,4) -- (3,5);
  \draw (4,4) -- (5,5);
  \draw (0,2) -- (1,2);
  \draw (1,2) -- (2,2);
  \draw (2,2) -- (3,2);
  \draw (3,2) -- (4,2);
  \draw (4,2) -- (5,2);
  \draw (5,2) -- (6,2);
\node[draw, rectangle, inner sep=4pt] at (0,6) {};
  \node[draw, rectangle, inner sep=4pt] at (3,5) {};
  \node[draw, rectangle, inner sep=4pt] at (5,5) {};
\end{tikzpicture}
\end{center}
Now, we have $$\operatorname{Peaks}X(w)=\{\eta_1=s_4s_5(\alpha_6),\eta_2=s_4s_5s_6s_3s_2(\alpha_1),\eta_3=s_4s_5s_6s_3s_2s_1(\alpha_4)\}.
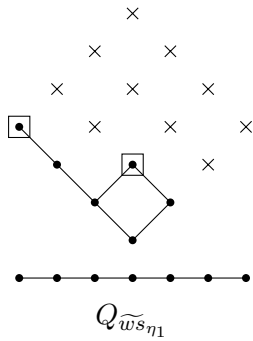
$$

\vspace{1cm}\begin{tikzpicture}[scale=0.5]
  \fill (3,9) node[mark size=3pt,color=black] {\pgfuseplotmark{x}};
  \fill (2,8) node[mark size=3pt,color=black] {\pgfuseplotmark{x}};
  \fill (4,8) node[mark size=3pt,color=black] {\pgfuseplotmark{x}};
\fill (1,7) node[mark size=3pt,color=black] {\pgfuseplotmark{x}};
  \fill (3,7) node[mark size=3pt,color=black] {\pgfuseplotmark{x}};
  \fill (5,7) node[mark size=3pt,color=black] {\pgfuseplotmark{x}};
  \fill (0,6) circle (3pt);
  \fill (2,6) node[mark size=3pt,color=black] {\pgfuseplotmark{x}};
  \fill (4,6) node[mark size=3pt,color=black] {\pgfuseplotmark{x}};
  \fill (6,6) node[mark size=3pt,color=black] {\pgfuseplotmark{x}};
  \fill (5,5) node[mark size=3pt,color=black] {\pgfuseplotmark{x}};
  \fill (1,5) circle (3pt);
  \fill (3,5) circle (3pt);
  \fill (2,4) circle (3pt);
  \fill (4,4) circle (3pt);
  \fill (3,3) circle (3pt);
  \fill (0,2) circle (3pt);
  \fill (1,2) circle (3pt);
  \fill (2,2) circle (3pt);
  \fill (3,2) circle (3pt)node [below=6pt] {$Q_{\widetilde{ws}_{\eta_1}}$};
  \fill (4,2) circle (3pt);
  \fill (5,2) circle (3pt);
  \fill (6,2) circle (3pt);
  
  \draw (0,6) -- (1,5);
  \draw (1,5) -- (2,4);
  \draw (2,4) -- (3,5);
  \draw (2,4) -- (3,3);
  \draw (3,3) -- (4,4);
  \draw (4,4) -- (3,5);

  \draw (0,2) -- (1,2);
  \draw (1,2) -- (2,2);
  \draw (2,2) -- (3,2);
  \draw (3,2) -- (4,2);
  \draw (4,2) -- (5,2);
  \draw (5,2) -- (6,2);
\node[draw, rectangle, inner sep=4pt] at (0,6) {};
  \node[draw, rectangle, inner sep=4pt] at (3,5) {};
\end{tikzpicture}\quad\quad\quad\quad\quad\quad
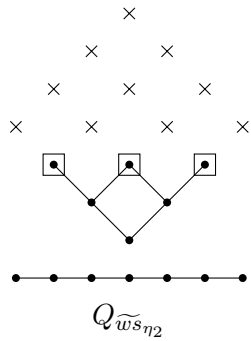
\begin{tikzpicture}[scale=0.5]
  \fill (3,9) node[mark size=3pt,color=black] {\pgfuseplotmark{x}};
  \fill (2,8) node[mark size=3pt,color=black] {\pgfuseplotmark{x}};
  \fill (4,8) node[mark size=3pt,color=black] {\pgfuseplotmark{x}};
\fill (1,7) node[mark size=3pt,color=black] {\pgfuseplotmark{x}};
  \fill (3,7) node[mark size=3pt,color=black] {\pgfuseplotmark{x}};
  \fill (5,7) node[mark size=3pt,color=black] {\pgfuseplotmark{x}};
  \fill (0,6) node[mark size=3pt,color=black] {\pgfuseplotmark{x}};
  \fill (2,6) node[mark size=3pt,color=black] {\pgfuseplotmark{x}};
  \fill (4,6) node[mark size=3pt,color=black] {\pgfuseplotmark{x}};
  \fill (6,6) node[mark size=3pt,color=black] {\pgfuseplotmark{x}};
  \fill (1,5) circle (3pt);
  \fill (3,5) circle (3pt);
  \fill (5,5) circle (3pt);
  \fill (2,4) circle (3pt);
  \fill (4,4) circle (3pt);
  \fill (3,3) circle (3pt);
  \fill (0,2) circle (3pt);
  \fill (1,2) circle (3pt);
  \fill (2,2) circle (3pt);
  \fill (3,2) circle (3pt)node [below=6pt] {$Q_{\widetilde{ws}_{\eta_2}}$};
  \fill (4,2) circle (3pt);
  \fill (5,2) circle (3pt);
  \fill (6,2) circle (3pt);

  \draw (1,5) -- (2,4);
  \draw (2,4) -- (3,5);
  \draw (2,4) -- (3,3);
  \draw (3,3) -- (4,4);
  \draw (4,4) -- (3,5);
  \draw (4,4) -- (5,5);
  \draw (0,2) -- (1,2);
  \draw (1,2) -- (2,2);
  \draw (2,2) -- (3,2);
  \draw (3,2) -- (4,2);
  \draw (4,2) -- (5,2);
  \draw (5,2) -- (6,2);
\node[draw, rectangle, inner sep=4pt] at (1,5) {};
  \node[draw, rectangle, inner sep=4pt] at (3,5) {};
  \node[draw, rectangle, inner sep=4pt] at (5,5) {};
\end{tikzpicture}\quad\quad\quad\quad\quad\quad
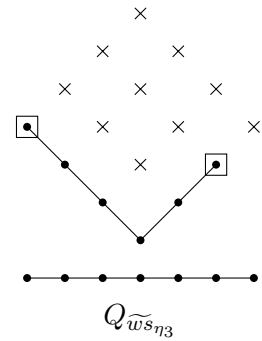
\begin{tikzpicture}[scale=0.5]
  \fill (3,9) node[mark size=3pt,color=black] {\pgfuseplotmark{x}};
  \fill (2,8) node[mark size=3pt,color=black] {\pgfuseplotmark{x}};
  \fill (4,8) node[mark size=3pt,color=black] {\pgfuseplotmark{x}};
\fill (1,7) node[mark size=3pt,color=black] {\pgfuseplotmark{x}};
  \fill (3,7) node[mark size=3pt,color=black] {\pgfuseplotmark{x}};
  \fill (5,7) node[mark size=3pt,color=black] {\pgfuseplotmark{x}};
  \fill (0,6) circle (3pt);
  \fill (2,6) node[mark size=3pt,color=black] {\pgfuseplotmark{x}};
  \fill (3,5) node[mark size=3pt,color=black] {\pgfuseplotmark{x}};
  \fill (4,6) node[mark size=3pt,color=black] {\pgfuseplotmark{x}};
  \fill (6,6) node[mark size=3pt,color=black] {\pgfuseplotmark{x}};
  \fill (1,5) circle (3pt);
  \fill (5,5) circle (3pt);
  \fill (2,4) circle (3pt);
  \fill (4,4) circle (3pt);
  \fill (3,3) circle (3pt);
  \fill (0,2) circle (3pt);
  \fill (1,2) circle (3pt);
  \fill (2,2) circle (3pt);
  \fill (3,2) circle (3pt)node [below=6pt] {$Q_{\widetilde{ws}_{\eta_3}}$};
  \fill (4,2) circle (3pt);
  \fill (5,2) circle (3pt);
  \fill (6,2) circle (3pt);

  \draw (0,6) -- (1,5);
  \draw (1,5) -- (2,4);
 
  \draw (2,4) -- (3,3);
  \draw (3,3) -- (4,4);
 
  \draw (4,4) -- (5,5);
  \draw (0,2) -- (1,2);
  \draw (1,2) -- (2,2);
  \draw (2,2) -- (3,2);
  \draw (3,2) -- (4,2);
  \draw (4,2) -- (5,2);
  \draw (5,2) -- (6,2);
\node[draw, rectangle, inner sep=4pt] at (0,6) {};

  \node[draw, rectangle, inner sep=4pt] at (5,5) {};
\end{tikzpicture}

\end{example}

Observe that the set $\operatorname{Peaks}Q_{\tilde{w}}$ is in bijection with the set $\operatorname{Peaks}X(w)$. Denote this bijection by $f$. Moreover, the set of Schubert divisors inside $X(w)$ is $\{\, X(ws_{f(i)})\mid i\in \operatorname{Peaks}Q_{\tilde{w}}\,\}.$

We now prove the following proposition. 
\begin{proposition}\label{cor: one-to-one}
    The quivers of the Schubert divisors of the minuscule Schubert variety $X(w)$ are in one-to-one correspondence with the set of peaks of $Q_{\tilde{w}}$. In particular, there exists a bijection $f$ between $\operatorname{Peaks}X(w)$ and $\operatorname{Peaks}Q_{\tilde w}$.
\end{proposition}
\begin{proof}
 Fix a reduced expression $\tilde{w}=s_{i_1}\cdots s_{i_r}$ of $w$.
 For any vertex $j\in[1,r]$, consider \[
 A_j:=\{1\leq i\leq r\mid i\preceq j\},
 \]and note that \[\#A_{j}=\begin{cases}
     1&\text{if }j\in \operatorname{Peaks}Q_{\tilde{w}},\\\geq 2& \text{otherwise}.
 \end{cases}\] Let $\tilde{w}^j$ be a reduced expression obtained from the reduced expression $\tilde{w}$ by removing the simple root $\alpha_{i_j}$. More precisely \[\
 \tilde{w}^j=\widetilde{ws}_{\eta},\qquad \text{where $\eta=s_{i_r}\cdots s_{i_{j+1}}(\alpha_{i_j})$}.\] Then, the quiver $Q_{\tilde{w}^j}$ is a subquiver of $Q_{\tilde w}$ with vertices $[1,r]\setminus A_{j}$ (see for example \cite[Proposition~4.5]{Per07}). In this case, the dimension of $X(ws_{\eta})$ is $(r-\#A_j)$. Thus, $X(ws_{\eta})$ is Schubert divisor if and only if $j\in \operatorname{Peaks}Q_{\tilde{w}}$. By Proposition \ref{pro:bases for cl(X(w))}, this is equivalent to $\eta\in \operatorname{Peaks}X(w)$ if and only if $j\in \operatorname{Peaks}Q_{\tilde{w}}$. This completes the proof.
\end{proof}
 
\begin{remark} \label{rmk: peaks}\ 

\begin{enumerate}
    \item 
    By Proposition \ref{cor: one-to-one}, we conclude that the definition of $\operatorname{Peaks}X(w)$ is equivalent to the definition of $\operatorname{Peaks}(Q_{\tilde{w}})$.
    
    \item For a minuscule element $w\in W$, Perrin constructs some partitions of $Q_{\tilde{w}}$ into quivers $(Q_{{\tilde{w}}_j})_{1\leq j\leq m}$ with $w_i$ minuscule elements such that $\widehat{w}=(w_1,\dots,w_m)$ is a good generalized decomposition of $w$ and  $$\operatorname{Peaks}(Q_{\tilde{w}})=\bigcup_{1\leq j\leq m}\operatorname{Peaks}(Q_{\tilde{w}_j}),$$see \cite[~Section~5.4]{Per07}.
    \item 
    By Proposition \ref{cor: one-to-one}, the definition of $\operatorname{Peaks}\widehat{X}(\widehat{w})$ defined in Section \ref{sec: Fano gBS} is bijectively corresponds to the notion of the peaks by Perrin as in (2) and we denote such bijection by $\widehat{f}$.
\end{enumerate}

\end{remark}
By Remark~\ref{rmk: peaks}(3), Proposition~\ref{pro: Gorensteinmin} for  $m=1$ gives the results of \cite{WY06} and \cite[Proposition~1.12]{Per09} on the characterization of Gorenstein minuscule Schubert varieties. 

 Furthermore, Proposition~\ref{anti-can} yields the following expression for the anticanonical divisor:
 $$[-K_{\widehat{X}(\widehat{w})}]=\sum_{\widehat{\eta}\in \operatorname{Peaks}\widehat{X}(\widehat{w})}(\operatorname{ht}(\widehat{\eta}^\vee)+1\big)\widehat{D}_{\widehat{\eta}}=\sum_{j\in \operatorname{Peaks}(Q_{\tilde{w}})}\big(\operatorname{ht}(\widehat{f}(j)^\vee)+1\big){\widehat{D}}_{\widehat{f}(j)}
$$ 
(see \cite[Fact 6.6]{Per07}).
Consequently, Corollary~\ref{cor:minuscule} recovers Corollary~6.7 of \cite{Per07}, that is, \[
\mathcal{O}[-K_{\widehat{X}(\widehat{w})}]=\sum_{1\leq j\leq m}c_{j,\alpha_j}\mathcal{L}_{j,\alpha_j}.
\] 
Thus, our results recover the corresponding results for minuscule generalized Bott-Samelson varieties.

\section{Examples}\label{sec:examples}
In this section, we illustrate our method through explicit calculations for Fano and weak Fano generalized Bott-Samelson varieties.
\begin{example}
Let $G = SL_4(\mathbb{C})$ and consider 
\[
\widehat{w} = (w_1, w_2) = (s_2 s_1 s_3, s_2s_1s_3s_2),
\] 
a non reduced admissible generalized decomposition of $w = s_2 s_1 s_3 s_2s_1s_3s_2 \in W$. 
By straightforward computation, we have
\[
S \setminus I^{w_1} = \{\alpha_1, \alpha_3\}, \quad 
S \setminus I^{w_2} = \{\alpha_2\}.
\]
The associated decomposition $\tilde{w}$ corresponding to $\widehat{w}$ is given as follows
\[
\tilde{w} = (s_{1,1}, s_{2,1}, s_{3,1}, s_{1,2},s_{2,2},s_{3,2},s_{4,2}) = (s_2, s_1, s_3, s_2, s_1, s_3, s_2),
\] with \[
w_1=\prod_{k=1}^{3}s_{k,1},\qquad\text{and}\qquad w_2=\prod_{k=1}^{4}s_{k,2}.
\]
 The total order $\preceq_{\widehat{w}}$ defined in Subsection~\ref{subsection: M1} gives 
\[
\alpha_{2,1} \prec_{\widehat{w}} \alpha_{3,1} \prec_{\widehat{w}} \alpha_{4,2}.
\]
By Definition \ref{def:Peaks Xw}, we have 
\[
\operatorname{Peaks} X(w_1) = \{\alpha_1 + \alpha_2 + \alpha_3, \alpha_1, \alpha_3\}, \quad
\operatorname{Peaks} X(w_2) = \{\alpha_1 + \alpha_2 + \alpha_3\}.
\]
Since $S\setminus I^{w_1}\subseteq \operatorname{Peaks}X(w_1)$. Thus, we consider the identity map \[
\mu_{w_1}:S\setminus I^{w_1}\to \operatorname{Peaks}X(w_1), \quad \alpha\mapsto \alpha,
\] and the map \[\mu_{w_2}:S\setminus I^{w_2}\to \operatorname{Peaks}X(w_2),\quad \alpha_2\mapsto \alpha_1 + \alpha_2 + \alpha_3.
\]Then, by the construction of the set $\mathcal{B}_{P^{w}}$, we have
\[
\mathcal{B}_{P^{w_1}} = \{\mu_{w_1}(\alpha_1),  \, \mu_{w_1}(\alpha_3)\}=\{\alpha_1,\, \alpha_3\}, \quad \text{and}\quad
\mathcal{B}_{P^{w_2}} = \{\mu_{w_2}(\alpha_2)\}=\{\alpha_1 + \alpha_2 + \alpha_3\}.
\]
Now, consider
\[
\operatorname{Peaks} X(w_1) \setminus \mathcal{B}_{ P^{w_1}} = \{\eta := \alpha_1 + \alpha_2 + \alpha_3\}, 
\quad \operatorname{Peaks} X(w_2) \setminus \mathcal{B}_{P^{w_2}} = \emptyset,
\]
and note that
\[
\operatorname{ht}(\eta^\vee) + 1 = 4 = \big(\operatorname{ht}(\mu_{w_1}(\alpha_1)^\vee) + 1\big) + \big(\operatorname{ht}(\mu_{w_1}(\alpha_3)^\vee) + 1\big).
\]
Then condition (1) of Theorem~\ref{thm: Fano gBS} is satisfied, so $\widehat{X}(\widehat{w})$ is Gorenstein.
The set of peaks of $\widehat{X}(\widehat{w})$ is given by \begin{align*}
\operatorname{Peaks}\widehat{X}(\widehat{w})&=w_2^{-1}\bigl(\operatorname{Peaks}X(w_1)\bigl)\sqcup \operatorname{Peaks}X(w_2)\\&=\{w_2^{-1}(\alpha_1+\alpha_2+\alpha_3),\,w_2^{-1}(\alpha_1),\,w_2^{-1}(\alpha_3)\}\sqcup \{\alpha_1 + \alpha_2 + \alpha_3\}\\&=\{\eta_{1,1}:=-\alpha_2,\,\eta_{2,1}:=\alpha_3,\,\eta_{3,1}:=\alpha_1,\,\eta_{1,2}:=\alpha_1 + \alpha_2 + \alpha_3\}.
\end{align*} Thus, by Proposition \ref{anti-can}, we get \begin{align*}
[-K_{\widehat{X}(\widehat{w})}]
&=
\sum_{\widehat{\eta}\in \operatorname{Peaks}\widehat{X}(\widehat{w})}
\bigl(\operatorname{ht}(\eta^\vee)+1\bigr)\,
\widehat{D}_{\widehat{\eta}} \\
&=
0\widehat{D}_{\widehat{\eta}_{1,1}}+2\widehat{D}_{\widehat{\eta}_{2,1}}+2\widehat{D}_{\widehat{\eta}_{3,1}}+4\widehat{D}_{\widehat{\eta}_{1,2}}.
\end{align*}
Next, to determine the Fano property of $\widehat{X}(\widehat{w})$, we solve the system of linear equations given in Eq.\eqref{eq:sublinear}. Recall the matrices $\widehat{M}_{P^{w}}$ and $\widehat{M}$ constructed in Subsection~\ref{subsub: Fano}, we have
\[
\widehat{M}_{P^{w_1}} =\begin{pmatrix}
    1&0&\left\langle\omega_2,\,w_2^{-1}(\mu_{w_1}(\alpha_1))\right\rangle\\0&1&\left\langle\omega_2,\,w_2^{-1}(\mu_{w_1}(\alpha_3))\right\rangle
\end{pmatrix}= \begin{pmatrix} 1 & 0 & 0 \\ 0 & 1 & 0 \end{pmatrix},\]and similarly\[ \quad
\widehat{M}_{P^{w_2}} =\begin{pmatrix} 0 & 0 & M_{P^{w_2}} \end{pmatrix}= \begin{pmatrix} 0 & 0 & 1 \end{pmatrix}.
\]
Thus, we get
\[
\widehat{M} =\begin{pmatrix}
    \widehat{M}_{P^{w_1}}\\\widehat{M}_{P^{w_2}}
\end{pmatrix}= \begin{pmatrix} 1 & 0 & 0 \\ 0 & 1 & 0 \\ 0 & 0 & 1 \end{pmatrix}\]

Applying Eq.\eqref{eq:Mhat matrices} with the total order $\preceq_{\widehat{w}}$, we get
\begin{align*}
\begin{pmatrix} 1 & 0 & 0 \\ 0 & 1 & 0 \\ 0 & 0 & 1 \end{pmatrix} 
\begin{pmatrix} c_{1,\alpha_1} \\ c_{1,\alpha_3} \\ c_{2,\alpha_2} \end{pmatrix} 
&= \begin{pmatrix} \operatorname{ht}((w_2^{-1}\mu_{w_1}(\alpha_1))^\vee) + 1 \\ 
\operatorname{ht}((w_2^{-1}\mu_{w_1}(\alpha_3))^\vee) + 1 \\ 
\operatorname{ht}(\mu_{w_2}(\alpha_2)^\vee) + 1 \end{pmatrix}  \\
\begin{pmatrix} c_{1,\alpha_1} \\ c_{1,\alpha_3} \\ c_{2,\alpha_2} \end{pmatrix} 
&=  \begin{pmatrix} 2 \\ 2 \\ 4 \end{pmatrix}.
\end{align*}
Hence, \[
\mathcal{O}([-K_{\widehat{X}(\widehat{w})}])=2\mathcal{L}_{1,\alpha_1}+2\mathcal{L}_{1,\alpha_3}+4\mathcal{L}_{2,\alpha_2}.
\]
Since all $c_{i,\alpha} > 0$, hence by Theorem~\ref{thm: Fano gBS}, we conclude that $\widehat{X}(\widehat{w})$ is Fano.
\end{example}
\begin{example} 
    Let $G=\operatorname{Spin}_{12}(\mathbb{C})$,  and consider $
    \widehat{w}=(w_1,w_2)=(s_1s_4s_5,s_4s_3)
    $ is an admissible
    generalized reduced decomposition of $w=s_1s_4s_5s_4s_3$ in $W$. Then, we have\[
    S\setminus I^{w_{1}}=\{\alpha_1,\alpha_5\}\quad\text{and}\quad S\setminus I^{w_{2}}=\{\alpha_3\}.
    \]Since \[
    \tilde{w}=(s_{1,1},s_{2,1},s_{3,1},s_{1,2},s_{2,2})=(s_1,s_4,s_5,s_4,s_3)
    \]is also an admissible
    generalized reduced decomposition of $w$. It follows from the total ordering $\preceq_{\widehat{w}}$ that\[\alpha_{1,1}\prec_{\widehat{w}}\alpha_{3,1}\prec_{\widehat{w}}\alpha_{2,2}.
    \] We also have 
    \[  \operatorname{Peaks}X(w_1)=\{s_5s_4(\alpha_1),s_5(\alpha_4)\}
    =\{\alpha_1,\alpha_4+\alpha_5\} \quad\text{and}\quad
\operatorname{Peaks}X(w_2)=\{s_3(\alpha_4)\}=\{\alpha_3+\alpha_4\}.
    \] Next, consider the maps \begin{align*}\mu_{w_1}: S\setminus I^{w_1}&\to \operatorname{Peaks}X(w_1),\quad \alpha_1\mapsto\alpha_1 \quad\text{and}\quad \alpha_5\mapsto\alpha_4+\alpha_5\\
    \mu_{w_2}: S\setminus I^{w_2}&\to \operatorname{Peaks}X(w_2), \quad\quad \alpha_4\mapsto\alpha_3+\alpha_4.
    \end{align*} 
    Then, we have \[
    \mathcal{B}_{P^{w_1}}=\{\mu_{w_1}(\alpha_1)=\alpha_1, \mu_{w_1}(\alpha_5)=\alpha_4+\alpha_5\},\quad\text{and} \quad \mathcal{B}_{P^{w_2}}=\{\mu_{w_2}(\alpha_3)=\alpha_3+\alpha_4\}.
    \]
    Hence,
\[
\operatorname{Peaks} X(w_1) \setminus \mathcal{B}_{ P^{w_1}} = \emptyset\quad \text{and} \quad 
\quad \operatorname{Peaks} X(w_2) \setminus \mathcal{B}_{P^{w_2}} = \emptyset .
\]
Thus, condition (1) of Theorem~\ref{thm: Fano gBS} is satisfied vacuously, so $\widehat{X}(\widehat{w})$ is Gorenstein.
Next, by the definition of matrices $\widehat{M}_{P^{w_i}}$ and $\widehat{M}$ constructed in Subsection~\ref{subsub: Fano}, we have\[
\widehat{M}_{P^{w_1}}=\begin{pmatrix}
    1&0&\left\langle\omega_3,w_{2}^{-1}(\mu_{w_1}(\alpha_1))\right\rangle\\0&1&\left\langle\omega_3,w_{2}^{-1}(\mu_{w_1}(\alpha_4))\right\rangle
\end{pmatrix}=\begin{pmatrix}
    1&0&0\\0&1&0
\end{pmatrix}\quad\text{and}\quad \widehat{M}_{P^{w_2}}=\begin{pmatrix}
    0&0&1
\end{pmatrix}.
\]Then, \[
\widehat{M}=\begin{pmatrix}
    1&0&0\\0&1&0\\0&0&1
\end{pmatrix}
\]Applying Eq.\eqref{eq:Mhat matrices} with the total order $\preceq_{\widehat{w}}$, we get\[
\widehat{M}=\begin{pmatrix}
    1&0&0\\0&1&0\\0&0&1
\end{pmatrix}\begin{pmatrix}
   c_{1,\alpha_1} \\ c_{1,\alpha_5} \\ c_{2,\alpha_3}
\end{pmatrix}=\begin{pmatrix} \operatorname{ht}((w_2^{-1}\mu_{w_1}(\alpha_1))^\vee) + 1 \\ 
\operatorname{ht}((w_2^{-1}\mu_{w_1}(\alpha_5))^\vee) + 1 \\ 
\operatorname{ht}(\mu_{w_2}(\alpha_3)^\vee) + 1 \end{pmatrix} =\begin{pmatrix}
    2\\2\\3
\end{pmatrix}
\]This implies that $c_{1,\alpha_1}=c_{1,\alpha_5}=2$ and $c_{2,\alpha_3}=3$, and we have\[
\mathcal{O}[-K_{\widehat{X}(\widehat{w})}]=2\mathcal{L}_{1,\alpha_1}+2\mathcal{L}_{1,\alpha_5}+3\mathcal{L}_{2,\alpha_3}.
\] Therefore, by Theorem \ref{thm: Fano gBS}, the variety $\widehat{X}(\widehat{w})$ is Fano.
\end{example}

\begin{example}
For type $B_2$, consider $G:=\operatorname{Spin}_5(\mathbb{C})$ and let us assume a non reduced admissible generalized decomposition  \[
\widehat{w}=(w_1,w_2,w_3)=(w_0,s_1,s_2)
\] of $w=w_0s_1s_2\in W$, where $w_0$ is the maximal element in $W$. In this case, we have \[\widehat{X}(\widehat{w})\simeq G\times ^B Z(s_1,s_2).\] 
For each simple root $\alpha_i\in S$, where $1\leq i\leq 2$, we have $S\setminus I^{s_{\alpha_i}}=\{\alpha_i\}$. Also, note that $I^{w_0}=\emptyset$ for $w_0$. Thus, we get \[
S\setminus I^{w_0}=\{\alpha_1,\alpha_2\}.
\]Now, we fix a reduced expression for $w_0$ as \[
\tilde{w}_0=(s_2,s_1,s_2,s_1)
\]such that $w_0=s_2s_1s_2s_1$. Then, the associated decomposition $\tilde{w}$ corresponding to $\widehat{w}$ is \[
\tilde{w}=(s_{1,1},s_{2,1},s_{3,1},s_{4,1},s_{1,2},s_{1,3})=(s_2,s_1,s_2,s_1,s_1,s_2),
\] 
where
\[
w_1=\prod_{k=1}^{4}s_{k,1}, \qquad
w_2=s_{1,2}, \qquad \text{and} \qquad
w_3=s_{1,3}.
\]
By the total order $\preceq_{\widehat{w}}$, 
the induced ordering on the collection
\[
\bigsqcup_{i=1}^{3}\left(S\setminus I^{w_i}\right)
\]
is given by
$
\alpha_{3,1}\prec_{\widehat{w}}
\alpha_{4,1}\prec_{\widehat{w}}
\alpha_{1,2}\prec_{\widehat{w}}
\alpha_{1,3}.
$ For each $1\leq i\leq 3$, the set of peaks of $X(w_i)$ is given by
 \[
\operatorname{Peaks}X(w_1)=\{\alpha_1,\alpha_2\},\quad \operatorname{Peaks}X(w_2)=\{\alpha_1\}\quad \text{and}\quad \operatorname{Peaks}X(w_3)=\{\alpha_2\}.
\]
Also, $S\setminus I^{w_i}\subseteq \operatorname{Peaks}X(w_i)$ for each $1\leq i\leq 3$. Thus, take $\mu_{w_i}$ to be indentity map, which gives \[
\mathcal{B}_{P^{w_1}}=\{\mu_{w_1}(\alpha_1)=\alpha_1,\, \mu_{w_1}(\alpha_2)=\alpha_2\},\quad \mathcal{B}_{P^{w_2}}=\{\mu_{w_2}(\alpha_1)=\alpha_1\},\quad \text{and} \quad \mathcal{B}_{P^{w_3}}=\{\mu_{w_3}(\alpha_2)=\alpha_2\}.
\]
Then condition $(1)$ of Theorem \ref{thm: Fano gBS} holds. Hence,
$\widehat{X}(\widehat{w})$ is Gorenstein.

Also, note that the set of peaks for $\widehat{X}(\widehat{w})$ is given by\begin{align*}
    \operatorname{Peaks}\widehat{X}(\widehat{w})&=(w_2w_3)^{-1}\bigl(\operatorname{Peaks}X(w_1)\bigl)\,\sqcup\, w_3^{-1}\bigl(\operatorname{Peaks}X(w_2)\bigl)\,\sqcup\,\bigl(\operatorname{Peaks}X(w_3)\bigl)\\&= s_2s_1\{\alpha_1,\alpha_2\}\sqcup s_2\{\alpha_1\}\sqcup\{\alpha_2\}\\&=\{-\alpha_1-2\alpha_2,\,\alpha_1+\alpha_2\}\sqcup\{\alpha_1+2\alpha_2\}\sqcup\{\alpha_2\}\\
    &=\{\widehat{\eta}_{1,1}:=-\alpha_1-2\alpha_2,\,\widehat{\eta}_{2,1}:=\alpha_1+\alpha_2,\,\widehat{\eta}_{1,2}:=\alpha_1+2\alpha_2,\,\widehat{\eta}_{1,3}:=\alpha_2\}.
    \end{align*} Then, by Proposition \ref{anti-can}, we obtain \begin{align*}
[-K_{\widehat{X}(\widehat{w})}]
&=
\sum_{\widehat{\eta}\in \operatorname{Peaks}\widehat{X}(\widehat{w})}
\bigl(\operatorname{ht}(\eta^\vee)+1\bigr)\,
\widehat{D}_{\widehat{\eta}} \\
&=
-\widehat{D}_{\widehat{\eta}_{1,1}}
+4\widehat{D}_{\widehat{\eta}_{2,1}}
+3\widehat{D}_{\widehat{\eta}_{1,2}}
+2\widehat{D}_{\widehat{\eta}_{1,3}}.
\end{align*}
     To determine the Fano property of $\widehat{X}(\widehat{w})$, we solve the system of linear equations given in Eq.\eqref{eq:sublinear}. We begin by recalling the constructions of the matrices $\widehat{M}_{P^w}$, which yields \begin{align*}
     \widehat{M}_{P^{w_1}}&=
\begin{pmatrix}
1 & 0 & \left\langle \omega_1,\, s_1(\mu_{w_1}(\alpha_2)) \right\rangle & \left\langle \omega_2,\, s_2s_1(\mu_{w_1}(\alpha_2)) \right\rangle \\
0 & 1 & \left\langle \omega_1,\, s_{1}(\mu_{w_1}(\alpha_1)) \right\rangle & \left\langle \omega_2,\, s_2s_{1}(\mu_{w_1}(\alpha_1)) \right\rangle
\end{pmatrix}\\
&=\begin{pmatrix}
    1&0&2&1\\0&1&-1&-1
\end{pmatrix}.
\end{align*}Similarly, we have \[
\widehat{M}_{P^{w_2}}=\begin{pmatrix}
    0&0&1&1
\end{pmatrix}, \qquad \text{and} \qquad \widehat{M}_{P^{w_3}}=\begin{pmatrix}
    0&0&0&1
\end{pmatrix}.
\]Then, by the definition of $\widehat{M}$, we get \[
\widehat{M}= \begin{pmatrix}
    1&0&2&1\\0&1&-1&-1\\0&0&1&1\\0&0&0&1
\end{pmatrix},\qquad \text{and} \qquad \widehat{M}^{-1}=\begin{pmatrix}
    1&0&-2&1\\0&1&1&0\\0&0&1&-1\\0&0&0&1
\end{pmatrix}.
\]Applying Eq.~\eqref{eq:Mhat matrices} with respect to the total order $\preceq_{\widehat{w}}$, we obtain \[
\begin{pmatrix}
    1&0&2&1\\0&1&-1&-1\\0&0&1&1\\0&0&0&1
\end{pmatrix} \begin{pmatrix}
    c_{1,\alpha_2}\\c_{1,\alpha_1}\\c_{2,\alpha_1}\\c_{\alpha_{3,2}}
\end{pmatrix}=\begin{pmatrix}
    \operatorname{ht}\bigl(((w_2w_3)^{-1}\mu_{w_1}(\alpha_2))^\vee\bigl)+1\\\operatorname{ht}\bigl(((w_2w_3)^{-1}\mu_{w_1}(\alpha_1))^\vee\bigl)+1\\\operatorname{ht}\bigl(((w_3)^{-1}\mu_{w_2}(\alpha_1))^\vee\bigl)+1\\\operatorname{ht}\bigl((\mu_{w_3}(\alpha_2))^\vee\bigl)+1
\end{pmatrix}=
    \begin{pmatrix}
        \operatorname{ht}(\widehat{\eta}_{2,1}^\vee)+1\\
        \operatorname{ht}(\widehat{\eta}_{1,1}^\vee)+1\\
        \operatorname{ht}(\widehat{\eta}_{1,2}^\vee)+1\\
        \operatorname{ht}(\widehat{\eta}_{1,3}^\vee)+1
    \end{pmatrix}=\begin{pmatrix}
        4\\-1\\3\\2
    \end{pmatrix}.
\]Multiplying by the inverse of the matrix $\widehat{M}$, we successively obtain \[
\begin{pmatrix}
    c_{1,\alpha_2}\\c_{1,\alpha_1}\\c_{2,\alpha_1}\\c_{\alpha_{3,2}}
\end{pmatrix} = \begin{pmatrix}
    1&0&-2&1\\0&1&1&0\\0&0&1&-1\\0&0&0&1
\end{pmatrix}\begin{pmatrix}
        4\\-1\\3\\2
    \end{pmatrix}=\begin{pmatrix}
        0\\2\\1\\2
    \end{pmatrix}.
\]Hence, \[
\mathcal{O}\bigl([-K_{\widehat{X}(\widehat{w})}]\bigl)=0\mathcal{L}_{1,\alpha_2}+2\mathcal{L}_{1,\alpha_1}+\mathcal{L}_{2,\alpha_1}+2\mathcal{L}_{3,\alpha_2}.
\] Since $c_{1,\alpha_2}=0$ and all the remaining coefficients are strictly positive, it follows from Theorem~\ref{thm: Fano gBS} that $\widehat{X}(\widehat{w})$ is weak Fano but not Fano. 
\end{example}

\begin{example}\label{ex: not Fano}
    Let $G=SL_5(\mathbb{C})$ and consider $$\widehat{w}=(w_1,w_2,w_3)=(s_2s_3s_1,s_2s_4,s_1s_3)$$ be an admissible generalized reduced decomposition of $w=s_2s_3s_1s_2s_4s_1s_3\in W$. Note that \[
    S\setminus I^{w_1}=\{\alpha_1,\alpha_3\},~S\setminus I^{w_2}=\{\alpha_2,\alpha_4\},~\text{and}~S\setminus I^{w_3}=\{\alpha_1,\alpha_3\}.
    \]
    Since $$\tilde{w}=(s_{1,1},s_{2,1},s_{3,1},s_{1,2},s_{2,2},s_{1,3},s_{2,3})=(s_2,s_3,s_1,s_2,s_4,s_1,s_3)$$ is also an admissible generalized reduced decomposition of $w$. Now, by the definition of total ordering $\preceq_{\widehat{w}}$,
    we get 
    \[    \alpha_{2,1}\prec_{\widehat{w}}\alpha_{3,1}\prec_{\widehat{w}}\alpha_{1,2}\prec_{\widehat{w}}\alpha_{2,2}\prec_{\widehat{w}}\alpha_{1,3}\prec_{\widehat{w}}\alpha_{2,3}.
    \]     
    Also, 
    \[
    \operatorname{Peaks}X(w_1)=\{\alpha_1+\alpha_2+\alpha_3,\alpha_3,\alpha_1\},~\operatorname{Peaks}X(w_2)=\{\alpha_2,\alpha_4\}~\text{and}~\operatorname{Peaks}X(w_3)=\{\alpha_1,\alpha_3\}.
    \]
    By the definition of $\mathcal{B}_{P^{w_i}}$, we get 
    \[
    \mathcal{B}_{P^{w_1}}=\{\mu_{w_1}(\alpha_1)=\alpha_1,\mu_{w_1}(\alpha_3)=\alpha_3\}\quad \text{and}\quad \mathcal{B}_{P^{w_j}}=\operatorname{Peaks}X(w_j)~~\text{for $j=2,3$}.
    \]
    Then, 
    \[
    \operatorname{Peaks}X(w_1)\setminus \mathcal{B}_{P^{w_1}}=\{\eta:=\sum_{i=1}^{3}\alpha_i\}.\]
    For $j=2,3$, we have 
    \[\operatorname{Peaks}X(w_j)\setminus \mathcal{B}_{P^{w_j}}=\emptyset.
    \]
    Moreover,
    \[
    \operatorname{ht}(\eta^\vee)+1=4=\big(\operatorname{ht}(\mu_{w_1}(\alpha_1)^\vee)+1\big)+\big(\operatorname{ht}(\mu_{w_1}(\alpha_3)^\vee)+1\big).
    \]
   All of these together satisfy condition (1) of Theorem~\ref{thm: Fano gBS}. 
Hence, the variety $\widehat{X}(\widehat{w})$ is Gorenstein. 

Next, by the definition of matrices $\widehat{M}_{P^{w_i}}$ and $\widehat{M}$ constructed in Subsection~\ref{subsub: Fano}, we have 
\[
    \widehat{M}_{P^{w_1}}=\begin{pmatrix}
        1&0&1&1&1&1\\ 0&1&1&0&0&1
    \end{pmatrix},\quad \quad \quad \widehat{M}_{P^{w_2}}=\begin{pmatrix}
        0&0&1&0&1&1\\0&0&0&1&0&1
    \end{pmatrix},
    \]
    and 
    \[\widehat{M}_{P^{w_3}}=
    \begin{pmatrix}
        0&0&0&0&1&0\\0&0&0&0&0&1
    \end{pmatrix}.
    \]
    Then,
\[
\widehat{M}
=
\begin{pmatrix}
1 & 0 & 1 & 1 & 1 & 1 \\
0 & 1 & 1 & 0 & 0 & 1 \\
0 & 0 & 1 & 0 & 1 & 1 \\
0 & 0 & 0 & 1 & 0 & 1 \\
0 & 0 & 0 & 0 & 1 & 0 \\
0 & 0 & 0 & 0 & 0 & 1
\end{pmatrix},
\qquad
\widehat{M}^{-1}
=
\begin{pmatrix}
1 & 0 & -1 & -1 & 0 & 1 \\
0 & 1 & -1 & 0 & 1 & 0 \\
0 & 0 & 1 & 0 & -1 & -1 \\
0 & 0 & 0 & 1 & 0 & -1 \\
0 & 0 & 0 & 0 & 1 & 0 \\
0 & 0 & 0 & 0 & 0 & 1
\end{pmatrix}.
\]
By Eq.\eqref{eq:Mhat matrices} with the total order $\preceq_{\widehat{w}}$, we obtain
\begin{align*}
    \begin{pmatrix}
1 & 0 & 1 & 1 & 1 & 1 \\
0 & 1 & 1 & 0 & 0 & 1 \\
0 & 0 & 1 & 0 & 1 & 1 \\
0 & 0 & 0 & 1 & 0 & 1 \\
0 & 0 & 0 & 0 & 1 & 0 \\
0 & 0 & 0 & 0 & 0 & 1
\end{pmatrix}\begin{pmatrix}
    c_{1,\alpha_3}\\
    c_{1,\alpha_1}\\
    c_{2,\alpha_2}\\
    c_{2,\alpha_4}\\
    c_{3,\alpha_1}\\
    c_{3,\alpha_3}\\
\end{pmatrix}&=\begin{pmatrix}
    \operatorname{ht}\big(((w_2w_3)^{-1}\mu_{w_1}(\alpha_3))^\vee\big)+1\\
    \operatorname{ht}\big(((w_2w_3)^{-1}\mu_{w_1}(\alpha_1))^\vee\big)+1\\
    \operatorname{ht}\big(((w_3)^{-1}\mu_{w_2}(\alpha_2))^\vee\big)+1\\
    \operatorname{ht}\big(((w_3)^{-1}\mu_{w_2}(\alpha_4))^\vee\big)+1\\
    \operatorname{ht}\big(\mu_{w_3}(\alpha_1)^\vee\big)+1\\
    \operatorname{ht}\big(\mu_{w_3}(\alpha_3)^\vee\big)+1
\end{pmatrix}=\begin{pmatrix}
    5\\
    3\\
    4\\
    3\\
    2\\
    2
\end{pmatrix}\\
\begin{pmatrix}
    c_{1,\alpha_3}\\
    c_{1,\alpha_1}\\
    c_{2,\alpha_2}\\
    c_{2,\alpha_4}\\
    c_{3,\alpha_1}\\
    c_{3,\alpha_3}\\
\end{pmatrix}&= \begin{pmatrix}
1 & 0 & -1 & -1 & 0 & 1 \\
0 & 1 & -1 & 0 & 1 & 0 \\
0 & 0 & 1 & 0 & -1 & -1 \\
0 & 0 & 0 & 1 & 0 & -1 \\
0 & 0 & 0 & 0 & 1 & 0 \\
0 & 0 & 0 & 0 & 0 & 1
\end{pmatrix}\begin{pmatrix}
    5\\
    3\\
    4\\
    3\\
    2\\
    2
\end{pmatrix}.
\end{align*}
By solving this, we get 
\[
c_{1,\alpha_3}=0,~c_{1,\alpha_1}=1,~ c_{2,\alpha_2}=0,~c_{2,\alpha_4}=1,~c_{3,\alpha_1}=2,\quad\text{and}\quad c_{3,\alpha_3}=2.
\] 
Therefore, by Theorem \ref{thm: Fano gBS}, the variety $\widehat{X}(\widehat{w})$ is weak Fano but not Fano. However, the variety $X(w_i)$ is Fano for all $1\leq i\leq 3$. 
    \end{example}

\begin{example} 
Let $G$ be a simple algebraic group of type $G_2$, and consider the reduced expression
$
\widetilde{w}=(s_2,s_1,s_2,s_1)=(w_1,w_2,w_3,w_4).
$ The total ordering $\preceq_{\widehat{w}}$ on the collection $$\bigsqcup_{k=1}^{4} S\setminus I^{w_k}$$ is given by $\alpha_2\prec_{\tilde{w}}\alpha_1\prec_{\tilde{w}}
\alpha_2\prec_{\tilde{w}} \alpha_2$. For $i=1,3$ and $j=2,4$, a straightforward computation shows that
\[
\operatorname{Peaks}X(w_i)
=
S\setminus I^{w_i}
=
\{\alpha_2\}
\qquad\text{and}\qquad
\operatorname{Peaks}X(w_j)
=
S\setminus I^{w_j}
=
\{\alpha_1\}.
\]
 Hence, for $i=1,3$ and $j=2,4$ we have 
\[
\mathcal{B}_{P^{w_i}}
=
\{\mu_{w_i}(\alpha_2)=\alpha_2\}
\qquad\text{and}\qquad
\mathcal{B}_{P^{w_j}}
=
\{\mu_{w_j}(\alpha_1)=\alpha_1\}.
\]
Then the condition $(1)$ of Theorem \ref{thm: Fano gBS} holds. Hence,
$\widehat{X}(\widehat{w})$ is Gorenstein. On the other hand, the set of peaks of the Bott-Samelson variety $Z(\widetilde{w})$ is
\begin{align*}
\operatorname{Peaks}Z(\widetilde{w})
&=
\{s_1s_2s_1(\alpha_2),\,s_1s_2(\alpha_1),\,s_1(\alpha_2),\,\alpha_1\}\\
&=
\{\eta_1:=3\alpha_1+2\alpha_2,\,\eta_2:=2\alpha_1+\alpha_2,\,\eta_3:=3\alpha_1+\alpha_2,\,\eta_4:=\alpha_1\}.
\end{align*} For $1\leq i\leq 4$, let $\xi_i$ denote the divisor class of $Z(\widetilde{ws}_{{\eta_i}})$. Then, by Proposition \ref{anti-can}, we obtain \begin{align*}
[-K_{Z(\tilde{w})}]&=\sum_{k=1}^{4}\bigl(\operatorname{ht}(\eta_k^\vee)+1\bigl)\xi_k= 4\xi_1+6\xi_2+3\xi_3+2\xi_4.
\end{align*}
Using the definition of $\widehat{M}_{P^w}$, we have \[
\widehat{M}_{P^{w_1}}=\begin{pmatrix}
    1&\left\langle\omega_1,s_1(\alpha_2) \right\rangle&\left\langle\omega_2,s_2s_1(\alpha_2) \right\rangle&\left\langle\omega_1,s_1s_2s_1(\alpha_2) \right\rangle
\end{pmatrix}=\begin{pmatrix}
    1&1&2&1
\end{pmatrix},\]

\[
\widehat{M}_{P^{w_2}}=\begin{pmatrix}
    0&1&\left\langle\omega_{2},s_2(\alpha_1)\right\rangle&\left\langle\omega_1,s_1s_2(\alpha_1)\right\rangle
\end{pmatrix}=\begin{pmatrix}
    0&1&3&2
\end{pmatrix},\]

\[
\widehat{M}_{P^{w_3}}=\begin{pmatrix}
    0&0&1&\left\langle\omega_1,s_1(\alpha_2)\right\rangle
    \end{pmatrix}=\begin{pmatrix}
        0&0&1&1
\end{pmatrix}
\quad \text{and}\quad
\widehat{M}_{P^{w_4}}=\begin{pmatrix}
    0&0&0&1
\end{pmatrix}.
\]
Then, we get \[
\widehat{M}=\begin{pmatrix}\widehat{M}_{P^{w_1}}\\ \widehat{M}_{P^{w_2}}\\ \widehat{M}_{P^{w_3}}\\ \widehat{M}_{P^{w_4}}
\end{pmatrix}=\begin{pmatrix}1&1&2&1\\ 0&1&3&2\\ 0&0&1&1\\ 0&0&0&1
\end{pmatrix}.
\] Applying Eq.\eqref{eq:Mhat matrices} with the total order $\preceq_{\widehat{w}}$, we obtain 
\[
\begin{pmatrix}
    1&1&2&1\\ 0&1&3&2\\ 0&0&1&1\\ 0&0&0&1
\end{pmatrix}\begin{pmatrix}
    c_{1}\\c_{2}\\c_{3}\\c_{4}
\end{pmatrix}=\begin{pmatrix}
   \operatorname{ht}\bigl(s_1s_2s_1(\mu_{w_1}(\alpha_2))^\vee\big)+1 \\
   \operatorname{ht}\bigl(s_1s_2(\mu_{w_2}(\alpha_1))^\vee\big)+1\\
   \operatorname{ht}\bigl(s_1(\mu_{w_3}(\alpha_2))^\vee\big)+1
        \\
        \operatorname{ht}\bigl((\mu_{w_4}(\alpha_1))^\vee\big)+1
   \end{pmatrix}=\begin{pmatrix}
  \operatorname{ht}(\eta_1^\vee)+1\\
  \operatorname{ht}(\eta_2^\vee)+1\\\operatorname{ht}(\eta_3^\vee)+1\\\operatorname{ht}(\eta_4^\vee)+1\\
\end{pmatrix}=\begin{pmatrix}
    4\\6\\3\\2
\end{pmatrix}.
\]
By solving this system, we get $c_{1}=1$, $c_{2}=-1$, $c_{3}=1$ and $c_{4}=2$. Thus, we get \[
\mathcal{O}[-K_{\widehat{X}(\widehat{w})}]=\mathcal{L}_{1,\alpha_2}-\mathcal{L}_{2,\alpha_1}+\mathcal{L}_{3,\alpha_2}+2\mathcal{L}_{4,\alpha_1}.
\] Hence, by Theorem~\ref{thm: Fano gBS} we conclude that $Z(s_2,s_1,s_2,s_1)$ is not weak Fano.
\end{example}

\begin{remark}
   In the above example, we can also compute the coefficients $(c_1,c_2,c_3,c_4)$ directly by using Theorem \ref{thm:BSCia}.
\end{remark}

\noindent{\bf Acknowledgements:} We would like to thank Michel Brion for helpful discussions.

\end{document}